\documentclass[a4paper,10pt]{amsart}

\usepackage[T1]{fontenc}
\usepackage[utf8]{inputenc}
\usepackage[english]{babel}
\usepackage{amsmath,amsthm,amssymb,bbm,enumerate,latexsym,tabularx}
\usepackage[all]{xy}
\usepackage{hyperref}

\renewcommand{\k}{\mathbbm k}

\theoremstyle{plain}
\newtheorem*{thm}{Theorem}
\newtheorem{THM}{Theorem}
\newtheorem*{prop}{Proposition}
\newtheorem{PROP}{Proposition}
\newtheorem{COR}{Corollary}
\newtheorem*{lem}{Lemma}
\newtheorem*{cor}{Corollary}

\theoremstyle{definition}
\newtheorem*{defn}{Definition}

\theoremstyle{remark}
\newtheorem*{rem}{Remark}
\newtheorem*{ex}{Example}

\numberwithin{equation}{subsubsection}

\subjclass[2010]{Primary 16E65, 16S40; Secondary 16E40, 16E45, 18G10}

\begin{document}

\title{Smash Products of Calabi-Yau Algebras by Hopf Algebras}

\author{Patrick Le Meur}

\address[Patrick Le Meur]{
  Laboratoire de Math\'ematiques Blaise Pascal, UMR6620 CNRS,
  Université Clermont Auvergne, Campus des Cézeaux, 3 place Vasarely,
  63178 Aubière cedex, France}

\curraddr{Universit\'e de Paris, Sorbonne Universit\'e, CNRS,
  Institut de Math\'ematiques de Jussieu-Paris Rive Gauche, IMJ-PRG,
  F-75013, Paris, France}

\email{patrick.le-meur@imj-prg.fr}

\thanks{First published in: Le Meur Patrick, Smash Products of Calabi-Yau Algebras by Hopf Algebras. J. Noncommut. Geom. 13 (2019), 887-961. DOI 10.4171/JNCG/341. © European Mathematical Society.}
\date{\today}

\keywords{Hopf algebra, smash product, Calabi-Yau algebra, skew
  Calabi-Yau algebra, Van den Bergh duality, Nakayama automorphism,
  homological determinant, weak homological determinant}

\begin{abstract}
  Let $H$ be a Hopf algebra and $A$ be an $H$-module algebra. This
  article investigates when the smash product $A\sharp H$ is (skew)
  Calabi-Yau, has Van den Bergh duality or is Artin-Schelter regular
  or Gorenstein. In particular, if $A$ and $H$ are skew Calabi-Yau,
  then so is $A\sharp H$ and its Nakayama automorphism is expressed
  using the ones of $A$ and $H$. This is based on a description of the
  inverse dualising complex of $A\sharp H$ when $A$ is a homologically
  smooth dg algebra and $H$ is homologically smooth and with
  invertible antipode. This description is also used to explain the
  compatibility of standard constructions of Calabi-Yau dg algebras
  with taking smash products.
\end{abstract}

\maketitle


\section*{Introduction}

The Calabi-Yau algebras were defined in \cite{G06} and are now widely
investigated. They appear in deformations of unimodular Poisson
structures (see \cite{MR3254768}, \cite{MR2734346} and \cite{vbd}). In
noncommutative geometry, many relevant Artin-Schelter regular algebras
are Calabi-Yau, like the Sklyannin algebras. The Calabi-Yau algebras
also appear as noncommutative resolutions of singularities, for
instance, as Jacobian algebras arising from brane tilings
(\cite{MR2592501}) or as skew group algebras of polynomial algebras
(see \cite{MR3357123} and \cite{MR3601076}).  Finally, there are
general constructions of candidates for being Calabi-Yau dg algebras,
such as the Ginzburg dg algebras of \cite{G06} or, more generally, the
(deformed) Calabi-Yau completions of \cite{MR2795754}, which are used
in the construction of cluster categories and their generalisations
(see \cite{A09a} and \cite{MR2739775}).

In these frameworks, many algebras of interest take the shape of a
smash product. While the works of Yekutieli (\cite{MR1762922}) and of
Brown and Zhang (\cite{MR2437632}) have shown that many interesting
Hopf algebras have a duality which is (weaker than and) close to being
Calabi-Yau, Reyes, Rogalski and Zhang initiated the study of the
algebras having this weaker duality (and called skew Calabi-Yau
algebras) by proving that, in the connected graded setting, being skew
Calabi-Yau is, on one hand, equivalent to being Artin-Schelter regular
and, on the other hand, relatively stable under taking smash products
with finite-dimensional Hopf algebras.


This motivates the work done in this article, which is, for a Hopf
algebra $H$ and an $H$-module differential graded algebra $A$, to
determine if $A\sharp H$ is Calabi-Yau (or, in case $A$ is an algebra,
if $A\sharp H$ has Van den Bergh duality or is skew Calabi-Yau).


Let $\k$ be a field. A differential graded (dg) ($\k$-)algebra $A$ is
called \emph{$n$-Calabi-Yau} if it is homologically smooth (that is,
$A\in {\mathrm{per}}(A^e)$) and ${\mathrm{RHom}}_{A^e}(A,A^e)[n]\simeq A$ in the
derived category $\mathcal D(A^e)$ of
$A^e=A\underset{\k}\otimes A^{\mathrm{op}}$. Often, a cofibrant replacement
of ${\mathrm{RHom}}_{A^e}(A,A^e)$ is referred to as an \emph{inverse dualising
  complex} of $A$.

When $A$ is a
$\k$-algebra, being Calabi-Yau means that $A$ admits a finite resolution by
finitely generated projective left $A^e$-modules (or, $A$-bimodules), and ${\mathrm{Ext}}_{A^e}^i(A,A^e)$ is isomorphic to $A$ as an $A$-bimodule if
$i=n$ and is zero otherwise. Recall the following weaker forms of duality.
\begin{itemize}
\item $A$ has \emph{Van den Bergh duality} in dimension $n$ if it is
  homologically smooth and the $A$-bimodule
  $\mathrm{Ext}^i_{A^e}(A,A^e)$ is invertible if $i=n$ and is zero
  otherwise.
\item $A$ is \emph{skew Calabi-Yau} in dimension $n$ when it has Van
  den Bergh duality in dimension $n$ and, moreover,
  $\mathrm{Ext}^n_{A^e}(A,A^e)$ is isomorphic to
  $A^{\mu_A}=\,^1A^{\mu_A}$ as an $A$-bimodule, for some automorphism
  $\mu_A \in {\mathrm{Aut}}_{\k-{\mathrm{alg}}}(A)$.
\end{itemize}
The naming in the former case refers to the sufficient conditions for
the duality theorem of Van den Bergh on the Hochschild (co)homology of
$A$ to hold true (see \cite[Theorem~1]{MR1443171}).  In the latter
case, $\mu_A$ is called a \emph{Nakayama automorphism.} It is uniquely
determined up to the composition with an inner automorphism. As usual,
given $\k$-algebra homomorphisms $\tau,\sigma\colon A\to A$, the piece
of notation $^\tau A^\sigma$ stands for the $\k$-vector space $A$ with
$A$-bimodule structure given by
$a\cdot x \cdot b = \tau(a)x\sigma(b)$.


This article hence describes an inverse dualising complex of
$A\sharp H$ when $A$ is a homologically smooth dg algebra acted on by
a homologically smooth Hopf algebra $H$ with invertible antipode. On
one hand, when $H$ is involutive, this description is applied to
express the deformed Calabi-Yau completions of $A\sharp H$ in terms of
smash products with $H$ of the deformed Calabi-Yau completions of
$A$. On the other hand, when $A$ is an algebra, this description is
applied to give necessary and/or sufficient conditions for $A\sharp H$
to have Van den Bergh duality or to be skew Calabi-Yau (with an
explicit Nakayama automorphism). As a consequence, the Nakayama
automorphisms of Artin-Schelter regular algebras have trivial
homological determinants as conjectured in \cite[Conjecture
6.4]{MR3250287}. Also, explicit Nakayama automorphisms are computed for
the smash products arising from actions of finite-dimensional Lie
algebras on polynomial algebras and from actions of
$\mathcal U_q(\mathfrak{sl}_2)$ on the quantum plane.


In this article, $H$ denotes a Hopf algebra with antipode $S$ and $A$
denotes an $H$-module dg algebra. ``$A$ is an algebra'' means that $A$
is concentrated in degree $0$ as a dg algebra. The smash product
$A\sharp H$ is denoted by $\Lambda$. When $A$ is augmented (or
connected ($\mathbb N$-)graded), it is also assumed that the
augmentation ideal is an $H$-submodule of $A$ (or that the action of
$H$ on $A$ preserves the grading, respectively).

\section{Main results and structure of the article}
\label{sec:main-results-sturct}

Assuming that $S$ is invertible is convenient and many Hopf algebras
which are relevant to the dualities considered here have this property
(see \cite{MR2437632}). This is the case of noetherian and Calabi-Yau
Hopf algebras (see \cite[Theorem 2.3]{MR2678828} whose proof can be
adapted to noetherian Hopf algebras with Van den Bergh duality). This
is actually the case for a broader class of Hopf algebras.

\begin{PROP}[\ref{sec:38}]
  \label{prop1}
  Any Hopf algebra with Van den Bergh duality has an invertible antipode.
\end{PROP}

When $S$ is invertible and $H$ is noetherian, it is proved in
\cite{MR2437632} that, if $H$ has Van den Bergh duality or is
Artin-Schelter regular, then it is skew Calabi-Yau with
$S^{-2}\circ \Xi^r_{\int_\ell}$ as a Nakayama automorphism. Here,
$\int_r$ is the right homological integral of $H$ and $\Xi_{\int_r}^r$
is the corresponding right winding automorphism of $H$. Combining this
result and Proposition~\ref{prop1} yields the following
characterisation.

\begin{THM}[\ref{sec:25}, \ref{sec:nak}, \ref{sec:char-calabi-yau}]
  \label{thm2}
  Let $H$ be a Hopf algebra with antipode $S$. The following
  conditions are equivalent 
  \begin{enumerate}[(i)]
  \item $H$ has Van den Bergh duality,
  \item  $\k_H\in
    {\mathrm{per}}(H^{\mathrm{op}})$, $S$ is invertible and the graded
    $\k$-vector space ${\mathrm{Ext}}^*_{H^{\mathrm{op}}}(\k_H,H)$ is
    finite-dimensional and concentrated in one degree,
  \item $H$ is skew Calabi-Yau.
  \end{enumerate}
  Under any of these conditions, $S^{-2}\circ \Xi_{\int_r\circ S}^r$
  is a Nakayama automorphism of $H$. In particular, $H$ is Calabi-Yau
  if and only if $\k_H\in {\mathrm{per}}(H^{\mathrm{op}})$, $S^2$ is an inner
  automorphism of $H$ and $H$ has the right Artin-Schelter property
  with trivial right homological integral.
\end{THM}
Note that it is proved in \cite[Theorem 2.3]{MR2678828} that, when $H$
is noetherian, $H$ is Calabi-Yau if and only if $S^2$ is an inner
automorphism and $H$ is Artin-Schelter regular with trivial left
homological integral.


This article is based on the description of
${\mathrm{RHom}}_{\Lambda^e}(\Lambda,\Lambda^e)$. When $S$ is invertible, there
exists a dg $A$-bimodule $D_A$ which is $H_{S^2}$-equivariant in the
sense of \cite{MR3250287} (see Section~\ref{sec:equivariant-modules})
and such that $D_A\simeq {\mathrm{RHom}}_{A^e}(A,A^e)$ in $\mathcal D(A^e)$. A
suitable extension of $D_A$ is then isomorphic to
${\mathrm{RHom}}_{\Lambda^e}(\Lambda,\Lambda^e)$ in the following sense. See
\ref{sec:4} for a general statement. See also \cite{MR2115022,
  MR3188338, MR2678828, MR2809906, MR2905560, MR3250287, MR3125886,
  MR2813562} for previous results describing
${\mathrm{RHom}}_{\Lambda^e}(\Lambda,\Lambda^e)$ when $A$ is a connected graded
algebra and $H$ is finite-dimensional, semisimple or cocommutative.

\begin{PROP}[\ref{sec:5}]
  \label{prop2}
  Let $H$ be a Hopf algebra with Van den Bergh duality in dimension $d$.
Let $A$ be an $H$-module dg algebra. Assume that $A$ is
  homologically smooth. Then, $\Lambda$ is homologically smooth and
  ${\mathrm{RHom}}_{\Lambda^e}(\Lambda,\Lambda^e) 
  \simeq D_A\sharp\,^\sigma H[-d]$ where $\sigma=(S^{-2}\circ
  \Xi_{\int_\ell}^r)^{-1} = S^2\circ \Xi_{\int_\ell\circ S}^r$.  
\end{PROP}
Here, $D_A\sharp \,^\sigma H$ is the dg $\Lambda$-bimodule associated
to the $H_{S^2}$-equivariant dg $A$-bimodule $D_A$ and defined in
\cite{MR3250287} (see \ref{sec:asharp-h-bimodules} for a reminder).

The description of ${\mathrm{RHom}}_{\Lambda^e}(\Lambda,\Lambda^e)$ can be used
to describe the deformed Calabi-Yau completions of $\Lambda$. Recall
that ${\mathrm{HH}}_{n-2}(A)\simeq H^0{\mathrm{Hom}}_{A^e}(D_A[n-1],A[1])$ when $D_A$
is cofibrant over $A^e$, which is possible to assume. The following
result was proved in \cite{L10c} when $H$ is the (semisimple) group
algebra of a finite group.

\begin{THM}[\ref{sec:calabi-yau-compl},
  \ref{sec:smash-products-with}]
  \label{sec:main-results-sturct-1}
  Let $H$ be an involutive Hopf algebra which is moreover Calabi-Yau
  in dimension $d$. Let $A$ be a homologically smooth $H$-module dg
  algebra. Let $n\in \mathbb Z$.
  \begin{enumerate}
  \item The $n$-Calabi-Yau completion $\Pi_n(A)$ is an $H$-module dg
    algebra and the dg algebras $\Pi_n(A)\sharp H$ and
    $\Pi_{n+d}(A\sharp H)$ are isomorphic.
  \item Given a deformed Calabi-Yau completion $\Pi_n(A,\alpha)$ such
    that $\alpha \in {\mathrm{HH}}_{n-2}(A)$ arises from an $H$-linear
    cocycle $D_A[n-1]\to A[1]$, then $\Pi_n(A,\alpha)$ is an
    $H$-module dg algebra and there is an associated $\overline \alpha
    \in {\mathrm{HH}}_{n+d-2}(\Lambda)$ such that $\Pi_n(A,\alpha)\sharp
    H\simeq \Pi_{n+d}(A\sharp H,\overline \alpha)$.
  \end{enumerate}
\end{THM}

The reader is referred to \ref{sec:calabi-yau-compl} and \ref{sec:33}
for generalisations to the case where $H$ is merely a Hopf algebra
with Van den Bergh duality.

 
When $A$ is an algebra, Proposition~\ref{prop2} can also be used to
characterise when $\Lambda$ has Van den Bergh duality.

\begin{THM}[\ref{sec:18}]
  \label{thm3}
  Let $A$ be an $H$-module algebra where $H$ is a Hopf algebra.
  Assume that the antipode $S$ is invertible and that both $A$ and $H$
  are homologically smooth. Then, the following assertions are
  equivalent
  \begin{enumerate}[(i)]
  \item $A$ and $H$ have Van den Bergh duality,
  \item $\Lambda$ has Van den Bergh duality.
  \end{enumerate} When these conditions are satisfied and $n$, $d$ are
  the corresponding homological dimensions of $A$ and $H$,
  respectively, then $\Lambda$ has dimension $n+d$ and
  \[{\mathrm{Ext}}^{n+d}_{\Lambda^e}(\Lambda,\Lambda^e)\simeq {\mathrm{Ext}}^n_{A^e}(A,A^e)\sharp\,^{(S^{-2}\circ
    \Xi_{\int_\ell}^r)^{-1}}H\,.\]
\end{THM}

This characterisation specialises to skew Calabi-Yau algebras. A
Nakayama automorphism of $A\sharp H$ was proved to exist and was
described in \cite[Theorem 0.2]{MR3250287} when $A$ is noetherian,
connected graded, and skew Calabi-Yau and $H$ is
finite-dimensional. This uses the homological determinant
(${\mathrm{hdet}}\colon H\to \k$) of the action of $H$ on $A$ (in the sense
of \cite{MR1758250,MR2568355}). In general, ${\mathrm{hdet}}$ is not
defined and \cite[Question 7.2]{MR3250287} asked for an extension of
its definition. As a partial answer, the concept of \emph{weak}
homological determinant ${\mathrm{whdet}}\colon H\to A$ is introduced in
\ref{sec:17} when $A$ is skew Calabi-Yau. It is determined by the
choice of a free generator of ${\mathrm{Ext}}^n_{A^e}(A,A^e)$ in
${\mathrm{mod}}(A)$ and it defines an algebra homomorphism
$\theta_{\mathrm{whdet}}\colon H\to \Lambda$ by
$h\mapsto {\mathrm{whdet}}(S^2(h_1))h_2$ (this replaces
$\Xi_{\mathrm{hdet}}^\ell$ when ${\mathrm{hdet}}$ is not defined).  The
following result extends \cite[Theorem 0.2]{MR3250287} which was
mentioned previously and answers \cite[Question 4.3]{MR3250287}.  See
\eqref{eq:56} and Table~\ref{tab:4} for examples where
$\mathrm{whdet}$ takes values outside $\k$.

\begin{THM}[\ref{sec:19-1} and \ref{sec:19-2}]
  \label{thm4}
  Let $H$ be a Hopf algebra with invertible antipode.  Let $A$ be an
  $H$-module algebra. Assume  that  $A$
  and $H$ are homologically smooth.
  \begin{enumerate}
  \item If $A$ and $H$ are skew Calabi-Yau, then so is $\Lambda$.
  \item If $\Lambda$ is skew Calabi-Yau, then so is $H$ and the action
    of $H$ on $A$ has a weak homological determinant. If, moreover, a
    homological determinant exists, then $A$ is skew Calabi-Yau.
  \end{enumerate}
  In the setting of (1), then 
  $\Lambda$ admits as a Nakayama automorphism
  \[
  \mu_\Lambda =\mu_A\sharp (\theta_{\mathrm{whdet}}\circ \mu_H)
  \]
  where $\mu_A$ is a Nakayama automorphism of $A$,
  ${\mathrm{whdet}}\colon H\to A$ is an associated weak homological
  determinant, and $\mu_H=S^{-2}\circ \Xi_{\int_\ell}^r$.
\end{THM}

Here, for given mappings $\alpha\colon A\to \Lambda$ and $\beta\colon
H\to \Lambda$, the piece of notation $\alpha\sharp \beta$ denotes the
mapping $\Lambda\to \Lambda,\,ah \mapsto \alpha(a)\beta(h)$.

Should $\mathrm{whdet}$ take its values in $\k$ then 
$\mu_\Lambda=\mu_A\sharp (\Xi_{\mathrm{whdet}}^{\ell}\circ
\mu_H)$. In particular, when $A$ is connected graded and skew
Calabi-Yau, then a generator of $\mathrm{Ext}^n_{A^e}(A,A^e)$ may be
chosen such that $\mathrm{whdet} = \mathrm{hdet}$. This yields
characterisations of when $A\sharp H$ is Calabi-Yau assuming that $H$
is so. See \cite[Theorem 3.4]{MR2678828} for a characterisation of
when $\mathcal U(\mathfrak g) \sharp \k G$ is Calabi-Yau, for 
finite-dimensional Lie algebras $\mathfrak g$ and finite groups $G$
(in zero characteristic), note that $\mathcal U(\mathfrak g)$ need
not be graded and that $\k G$ is then semisimple and Calabi-Yau in
dimension $0$.

\begin{THM}[\ref{sec:6}]
  \label{thm6}
    Let $H$ be a Calabi-Yau Hopf algebra. 
   Let $A$ be a connected graded  $H$-module algebra.  Let $h_0\in
  H^\times$ be such that $S^{-2}$ is the inner automorphism of $h_0$
  (see Theorem~\ref{thm2}). Then, $\Lambda=A\sharp H$ is
  Calabi-Yau if and only if the following conditions hold
  \begin{enumerate}[(a)]
    \item $A$ is skew Calabi-Yau,
    \item  ${\mathrm{hdet}}=\epsilon$,
    \item $(\exists k_A\in Z(H^\times))
          (\forall a\in A)\ \ 
          \mu_A(a)=(h_0k_A)\rightharpoonup a \underset{\text{in
              $\Lambda$}} = (h_0k_A)a (h_0 k_A)^{-1}$.
  \end{enumerate}
\end{THM}

When $A$ is, moreover, Calabi-Yau, Theorem~\ref{thm6} simplifies as
follows.
\begin{COR}[\ref{sec:28}]
  \label{sec:main-results-sturct-2}
  Let $H$ be a Calabi-Yau Hopf algebra. Let $A$ be a connected graded
  $H$-module algebra which is moreover Calabi-Yau. The following
  assertions are equivalent
  \begin{enumerate}[(i)]
  \item $A\sharp H$ is Calabi-Yau,
  \item ${\mathrm{hdet}}=\epsilon$.
  \end{enumerate}
\end{COR}

This characterisation was proved previously in the
following  situations
\begin{itemize}
\item  in \cite{MR2813562} assuming that $A$ is $p$-Koszul Calabi-Yau
  and that $H=\k G$ for any finite subgroup $G$ of ${\mathrm{Aut}}_{\k-{\mathrm{alg}}}(A)$ such that ${\mathrm{car}}(\k)$ does not divide ${\mathrm{Card}}(G)$,
  \item in \cite{MR2905560} assuming that $A$ is
    $p$-Koszul and Calabi-Yau and that $S^2={\mathrm{Id}}_H$,
  \item  in \cite[Corollary 3.4]{MR2809906} assuming that
    $A$ is $p$-Koszul and Calabi-Yau and $H=(\k G)^*$  for any finite
    group $G$.
  \end{itemize}

  It was conjectured in \cite[Conjecture 6.4]{MR3250287} that the
  Nakayama automorphisms of all connected graded Artin-Schelter
  Gorenstein algebras have trivial homological determinant and this
  was proved for noetherian and connected graded Artin-Schelter
  regular algebras in \cite[Corollary 5.4]{MR3557775}.  Combining
  Theorem~\ref{thm6} and the main result of \cite{MR3188338}, it is
  possible to prove that the noetherian hypothesis is unnecessary.
\begin{COR}[~\ref{sec:49}]
  \label{sec:main-results-struct}
  Let $A$ be a connected graded Artin-Schelter regular algebra
  (equivalently, a connected graded skew Calabi-Yau algebra, see
  \cite[Lemma 1.2]{MR3250287}). Let $\mu_A$ be its (graded) Nakayama
  automorphism. Let $H=\k \mathbb Z$ and consider the action of $H$ on $A$
  induced by $\mu_A$. Then, ${\mathrm{hdet}}(\mu_A)=1$.
\end{COR}
The above mentioned conjecture was proved previously in the following
cases,
\begin{itemize}
\item in \cite[Theorem 0.4]{MR3250287}, for noetherian connected
  graded Koszul Artin-Schelter regular algebras;
\item in \cite[Corollary 5.4]{MR3557775}, for noetherian and connected
  graded Artin-Schelter Gorenstein algebras of one of the following
  shapes
  \begin{itemize}
  \item graded twists of algebras which are finite over their affine
    centres
  \item quotients of noetherian Artin-Schelter regular algebras;
  \end{itemize}
\item in \cite[Theorem 1.6]{MR3421098}, for $m$-Koszul Artin-Schelter
  regular algebras;
\item in \cite[Theorem 3.11]{2017arXiv170605754C}, for certain $4$
  dimensional connected graded Artin-Schelter regular algebras which
  are normal extensions of $3$ dimensional ones.
\end{itemize}


The last main result of this text gives sufficient conditions for
$\Lambda$ to be Artin-Schelter Gorenstein/regular when $A$ is an
augmented $H$-module algebra. Note that the hypotheses below entail
that the antipode of $H$ is invertible.
\begin{THM}[\ref{sec:13}, \ref{sec:21}]
  \label{thm5}
  Let $H$ be a  Hopf algebra. Let $A$ be an
  augmented $H$-module algebra which is moreover noetherian. Assume
  that $A$ is Artin-Schelter 
  Gorenstein in dimension $n$.
  \begin{enumerate}
  \item If ${\mathrm{gl.dim.}}\,A<\infty$ and $H$ has Van den Bergh duality
    in dimension $d$, then  $\Lambda$ is Artin-Schelter regular
    in dimension $n+d$.
  \item If ${\mathrm{dim}}_\k\,H<\infty$, then $\Lambda$ is Artin-Schelter
    Gorenstein in dimension $n$.
  \item If $H$ has Van den Bergh duality in dimension $d$, then
    $A\sharp H$ is Artin-Schelter Gorenstein in dimension $n+d$.
  \end{enumerate}
\end{THM}
Part (2) was proved in \cite[Theorem 4.1]{MR3250287} when $A$ is
connected graded Artin-Schelter Gorenstein. Besides \cite[Proposition
3.8]{MR2809906} proved that, when $H$ is finite-dimensional and
semi-simple and $A$ is an $H$-module dg algebra concentrated in
nonnegative degrees and with zero component equal to $\k$, then $A$ is
Artin-Schelter Gorenstein if and only if so is $A\sharp H$.


This article is organised as
follows. Section~\ref{sec:basic-definitions} recalls useful
definitions, it sets conventions and it proves useful folklore
results.  Section~\ref{sec:homol-dual-hopf} proves
Proposition~\ref{prop1} and derives
Theorem~\ref{thm2}. Section~\ref{sec:equivariant-modules} is proves
needed properties of $H_{S^{2i}}$-equivariant dg $A$-bimodules
($i\in \mathbb Z$). Section~\ref{sec:proof-main-result} is devoted to
the description of $\mathrm{RHom}_{\Lambda^e}(\Lambda,\Lambda^e)$ and
the proof of
Proposition~\ref{prop2}. Section~\ref{sec:appl-constr-calabi} applies
this description to the compatibility of deformed Calabi-Yau
completions with taking smash products, it proves
Theorem~\ref{sec:main-results-sturct-1}.
Section~\ref{sec:appl-dual-non} uses this description to prove
Theorem~\ref{thm3} and Theorem~\ref{thm4}. As a corollary it proves
Theorem~\ref{thm6} and Corollary~\ref{sec:main-results-sturct-2}.  The
results of this section are applied in
Section~\ref{sec:exampl-acti-u_qm} to the computation of a Nakayama
automorphism of $A\sharp H$ when $A =\mathbb C_q[x,y]$ and
$H=\mathcal U_q(\mathfrak{sl}_2)$ ($q\in \mathbb C^\times$ not being a
root of unity).  Finally, Section~\ref{sec:appl-artin-schelt}
concentrates on the case where $A$ is an augmented $\k$-algebra. It
proves Theorem~\ref{thm5}.

For the ease of reading, an index of notation is provided at the end
of the article.

\section{Basic definitions and conventions}
\label{sec:basic-definitions}

\subsection{Conventions on notation}
\label{sec:conventions-notation}

For simplicity, $\otimes_\k$ is denoted by $\otimes$.


The counit of $H$ is denoted by $\epsilon$. The Sweedler notation
$h_1\otimes h_2$ is used for the comultiplication of $h\in H$,
omitting the summation symbol. The action of $h\in H$ on an element
$x$ of a left (or right) $H$-module is written as $h\rightharpoonup x$
(or, $x\leftharpoonup h$, respectively).


The category of left dg $A$-modules is denoted
by $\mathcal C(A)$. And $\mathcal C (A^{\mathrm{op}})$ is
identified with the category of right dg $A$-modules. The derived category of $A$ is denoted by
$\mathcal D(A)$ and defined as the localisation of $\mathcal C(A)$ at
the class of all quasi-isomorphisms. The perfect derived category of
$A$ is denoted by ${\mathrm{per}}(A)$ and defined as the smallest
triangulated subcategory of $\mathcal D(A)$ containing $A$ and stable
under taking direct summands.
When $A$ is an algebra, the category of left
$A$-modules is denoted by ${\mathrm{mod}}(A)$. And ${\mathrm{mod}}(A^{\mathrm{op}})$
is identified with the category of right $A$-modules.

For all $X\in\mathcal C(A)$, the suspension of $X$ is denoted by
$X[1]$. For all $X,Y\in \mathcal C(A)$, then
${\mathrm{Hom}}_A(X,Y)$ denotes the
following complex of vector spaces
\begin{itemize}
\item for $n\in \mathbb Z$, its component of degree $n$
  is the vector space of (homogeneous of degree zero) morphisms of
  graded vector spaces $f\colon X\to Y[n]$ such that
  $f(ax) = (-1)^{n\cdot {\mathrm{deg}}(a)} af(x)$ for all homogeneous
  $x\in X$ and $a\in A$,
  \item the differential is given by $f\mapsto d_Y\circ f-(-1)^{\mathrm{deg}(f)}f\circ d_X$.
  \end{itemize}
Hence, the morphism space $\mathcal C(A)(X,Y)$ equals $Z^0{\mathrm{Hom}}_A(X,Y)$.
  
No difference is made between  dg $A$-bimodules and  left dg
$A^e$-modules.  For such a dg $A$-bimodule
$M$,  the identity $a_1ma_2=(-1)^{{\mathrm{deg}}(m){\deg(a_2)}}(a_1\otimes
a_2)\cdot m$ holds 
 when $m\in M$ and
$a_2\in A$ are homogeneous. In particular, given $M,N\in \mathcal
C(A^e)$, $n\in \mathbb Z$ and $f\in {\mathrm{Hom}}_\k(M,N)^n$, then $f\in
{\mathrm{Hom}}_{A^e}(M,N)$ if and only if $f(axb) = (-1)^{n\cdot {\mathrm{deg}}(a)}
af(x)b$, for every $a\in A$ homogeneous, $x\in X$ and $b\in B$.

Here is a reminder of the features of $\mathcal C(A)$ (see
\cite{MR1258406} for details). A dg module $P\in \mathcal C(A)$ is
\emph{cofibrant} if, for every surjective quasi-isomorphism $X\to Y$
in $\mathcal C(A)$, then any morphism $P\to Y$ in $\mathcal C(A)$
lifts to $X$. There exists a model structure on $\mathcal C(A)$ whose
class of weak equivalences consists of all the quasi-isomorphisms, and
whose class of cofibrant objects consists of all the cofibrant left dg
$A$-modules. In particular,
\begin{itemize}
\item for every $X\in \mathcal C(A)$ there exists a quasi-isomorphism
  $P\to X$ in $\mathcal C(A)$ where $P$ is cofibrant (such a $P$ is
  called a \emph{cofibrant replacement} of $X$),
\item for every $P,X\in \mathcal C(A)$ such that $P$ is cofibrant, the
  canonical mapping $H^0{\mathrm{Hom}}_A(P,X)\to \mathcal D(A)(P,X)$ is
  bijective,
\item every cofibrant $P\in \mathcal C(A)$ is \emph{homotopically
    projective,} that is, for every quasi-isomorphism $X\to Y$ in
  $\mathcal C(A)$, then ${\mathrm{Hom}}_A(P,X)\to {\mathrm{Hom}}_A(P,Y)$ is a
  quasi-isomorphism.
\end{itemize}
Fibrant dg modules are defined dually and feature dual properties. In
particular, they are \emph{homotopically injective.}

The two following basic facts are used without further reference in
this article.
\begin{lem}
  Let $A,B$ be dg algebras.
  \begin{enumerate}
  \item Let $A\to B$ be a morphism of dg algebras. If  $B\simeq
    A\otimes V$ in 
    $\mathcal C(A)$ for some complex of vector spaces $V$, then the
    restriction-of-scalars functor $\mathcal C(B)\to \mathcal C(A)$
    preserves cofibrant objects.
  \item Let $A\to B$ be a morphism of dg algebras. Then, the
    extension-of-scalars functor $\mathcal C(A)\to \mathcal C(B)$
    preserves cofibrant objects.
  \item The restriction-of-scalars functor $\mathcal C(A\otimes B^{\mathrm{op}})\to \mathcal C(A)$ preserves fibrant objects.
  \end{enumerate}
\end{lem}


For the sake of simplicity, the verifications of module structures
omit the obvious quantifiers and use implicitly lower case letters for
elements in a space named with the corresponding upper case letter
($x\in X$, $y\in Y$, \emph{etc.}). Also $h,k,\ell$ always denote
elements of $H$, whereas $a,b$ always denote homogeneous elements of
$A$. As for the equalities presented as ``identities'', and which may
involve parameters ($a$, $h$, $d$, \emph{etc.}), it is implicit that
they hold true for all possible values of the parameters
($\forall a\in A$, $\forall h\in H$, $\forall d\in D$, \emph{etc.}).

\subsection{The smash product $\Lambda=A\sharp H$}
\label{sec:cross-prod-lambd}

\subsubsection{}
\label{sec:cross-prod-lambd-1}
Recall that a structure of $H$-module dg algebra on $A$ is a morphism
of complexes of vector spaces $H\otimes A\to A,\,h\otimes a \mapsto h
\rightharpoonup a$  (with $H$ in
degree $0$) such that the following identities hold true in $A$
\[
  \begin{array}{rclccrcl}
    1 \rightharpoonup a & = & a,
    &&&
       h\rightharpoonup (ab) & = & (h_1\rightharpoonup a)
                                   (h_2\rightharpoonup b),
    \\
    h\rightharpoonup 1 & = & \epsilon(h),
    &&&
    (hk)\rightharpoonup a & = & h\rightharpoonup (k\rightharpoonup
                                a)\,.
  \end{array}
  \]
  
  The dg algebra $A\sharp H$ has $A\otimes H$ as underlying complex of
  vector spaces. A tensor $a\otimes h$ is denoted by $ah$. The
  (associative) product of $A\sharp H$ is given by
\[
(ah) \times (bk) = (a\otimes h) \times (b\otimes k) = a
(h_1\rightharpoonup b) \otimes  h_2k = a (h_1\rightharpoonup b) h_2k\,.
\]

\subsubsection{}
\label{sec:cross-prod-lambd-2}

Assume that $S$ is invertible.
The following identities hold in $\Lambda^e$
\begin{equation}
  \label{eq:4}
\begin{array}{l}
  (h\otimes k)\times (a\otimes b) = ((h_1\rightharpoonup a)\otimes
  (S^{-1}(k_1)\rightharpoonup b)) \times (h_2\otimes k_2)\\
(a\otimes b) \times (h\otimes k) = (h_2\otimes k_2)\times
  ((S^{-1}(h_1)\rightharpoonup a) \otimes (k_1\rightharpoonup b))\,.
\end{array}
\end{equation}

The algebra $H^{\mathrm{op}}$ is a Hopf
algebra with coproduct given by $h\mapsto h_1\otimes h_2$ and antipode
$S^{-1}$. Also $H^e$ is a Hopf algebra with coproduct given by
$(h\otimes k)_1\otimes (h\otimes k)_2=(h_1\otimes k_1)\otimes
(h_2\otimes k_2)$ and antipode $S\otimes S^{-1}$.  

There is  a structure of dg $H^e$-module algebra on $A^e$ such that
\begin{equation}
  \label{eq:51}
  (h\otimes k) \rightharpoonup (a\otimes b)
=        (h\rightharpoonup
a)\otimes (S^{-1}(k)\rightharpoonup b)\,.
\end{equation}
 The resulting 
smash product dg algebra $A^e\sharp H^e$ is isomorphic to $\Lambda^e$
\emph{via} the mapping $A^e\otimes H^e\to \Lambda^e$ defined by $a\otimes
b\otimes h \otimes k \mapsto (a\otimes b)\times (h\otimes k)$. In
other words, the following identity holds true in $\Lambda^e$
\begin{equation}
  \label{eq:15}
  (h\otimes k) \times (a\otimes b) = (h\otimes k)_1\rightharpoonup
  (a\otimes b) \times (h\otimes k)_2\,.
\end{equation}

The natural structure of left dg $A^e$-module of $A^e$ extends to a
structure of left dg $\Lambda^e$-module such that
$(a h\otimes b 
k) \rightharpoonup (x\otimes y)=
(-1)^{{\mathrm{deg}}(b) (
  {\mathrm{deg}}(x) + {\mathrm{deg}}(y)
  )}
(a(h\rightharpoonup x))\otimes
(S^{-1}(k)\rightharpoonup (yb))$.

This structure and the structure of right dg $A^e$-module of $A^e$ do
not form a structure of $\Lambda^e-A^e$-bimodule. Instead, those two
structures are compatible in
the following sense (where $m\in A^e$)
\begin{equation}
  \label{eq:compat}
  \begin{array}{l}
    (h\otimes k)\rightharpoonup (m
    \leftharpoonup (a\otimes b))
     = 
          ((h\otimes k)_1\rightharpoonup m) \leftharpoonup ((h\otimes
    k)_2\rightharpoonup (a\otimes 
          b)) \\
    (a'\otimes b') \rightharpoonup (m \leftharpoonup
    (a\otimes b))
     = 
          ((a'\otimes b') \rightharpoonup m)
          \leftharpoonup (a\otimes b)\,.
  \end{array}
\end{equation}
(which is part of the identities defining the $H^e$-module dg algebra
structure on $A^e$).

\subsection{Duality conditions on dg algebras}
\label{sec:dual-cond-dg}

Let $n$ be a natural integer.

\subsubsection{}
\label{sec:dual-cond-dg-1}

The definitions of the dualities considered in this article are
recalled in the introduction. Note that, given a dg $A^e$-bimodule
$X$, then ${\mathrm{Hom}}_{A^e}(A,X)$ is a dg $A$-bimodule. Whence the functor
${\mathrm{Hom}}_{A^e}(A,-) \colon \mathcal C(A^e\otimes (A^e)^{\mathrm op}) \to
\mathcal C(A^e)$.
This is how ${\mathrm{RHom}}_{A^e}(A,A^e)$ is considered as an object of
$\mathcal D(A^e)$.  When $A$ is a $\k$-algebra,
${\mathrm{Ext}}^i_{A^e}(A,A^e)$ inherits of a structure of $A$-bimodule for
every $i$. Recall that, if $A$ is a $\k$-algebra, then $A$ is
Calabi-Yau if and only if it is skew Calabi-Yau and any Nakayama
automorphism for $A$ is inner (equivalently, the identity map of $A$
is a Nakayama automorphism).

\subsubsection{}
\label{sec:dual-cond-dg-2}

When $A$ is a skew Calabi-Yau algebra in dimension $n$, every free
generator $e$ of ${\mathrm{Ext}}^n_{A^e}(A,A^e)$ in
${\mathrm{mod}}(A)$ determines a unique Nakayama automorphism
$\mu\in {\mathrm{Aut}}_{\k-{\mathrm{alg}}}(A)$ such that the identity
$ea = \mu(a) e$ holds in ${\mathrm{Ext}}^n_{A^e}(A,A^e)$.
\begin{lem}
  Let  $A$ be a $\k$-algebra. Let $D\in {\mathrm{mod}}(A^e)$ be
  such that there exists $\nu\in {\mathrm{Aut}}_{\k-{\mathrm{alg}}}(A)$
  verifying $D\simeq A^\nu$ in ${\mathrm{mod}}(A^e)$. Let $d\in D$ be a
  free generator of $D$ in ${\mathrm{mod}}(A)$ and denote by $\mu\colon
  A\to A$ the algebra homomorphism such that the identity $da=\mu(a)d$
  holds in $D$. Then, $\mu\in {\mathrm{Aut}}_{\k-{\mathrm{alg}}}(A)$ and
  $D\simeq A^\mu$ in ${\mathrm{mod}}(A^e)$.
\end{lem}
\begin{proof}
  The mapping $A^\mu \to D\,,a\mapsto ad$ is an isomorphism in
  ${\mathrm{mod}}(A^e)$.  It therefore suffices to prove that $\mu$ is
  bijective.  There exists a free generator $d'\in D$ in
  ${\mathrm{mod}}(A)$ such that the identity $d'a=\nu(a)d'$ holds in
  $D$. Let $\alpha\in A^\times$ be such that $d'=\alpha d$. Then,
  $\mu(a) d = da = \alpha^{-1}d'a = \alpha^{-1}\nu(a)d' =
  \alpha^{-1}\nu(a) \alpha d$.
  Therefore, $\mu\circ \nu^{-1}\in {\mathrm{Aut}}_{\k-{\mathrm{alg}}}(A)$, and
  hence $\mu\in {\mathrm{Aut}}_{\k-{\mathrm{alg}}}(A)$.
\end{proof}

\subsubsection{}
\label{sec:dual-cond-dg-3}

When $A$ is moreover connected ($\mathbb N$-)graded, it is required
that there exists $\ell\in \mathbb Z$ and a homogeneous
$\mu_A\in {\mathrm{Aut}}_{\k-{\mathrm{alg}}}(A)$ such that
${\mathrm{Ext}}_{A^e}^n(A,A^e)\simeq A^{\mu_A}(\ell)$ as graded
$A$-bimodules for $A$ to be considered as skew Calabi-Yau in the
graded sense. Here, $-(\ell)$ denotes the degree shift of graded
modules.  The following lemma is used later on. Its proof is
elementary and omitted.
\begin{lem}
  Let $A$ be a connected graded $\k$-algebra. Assume that $A$ is
  skew Calabi-Yau in the ungraded sense, then so is it in the graded sense.
\end{lem}

\subsubsection{}
\label{sec:dual-cond-dg-4}

Assume that $A$ is an augmented algebra. Then, $A$ is said to satisfy
the \emph{left Artin-Schelter} condition in dimension $n$ if
\begin{equation}
  \label{eq:23}
  {\mathrm{dim}}_{\k}{\mathrm{Ext}}_{A}^i(\,_A\k,A) =
  \left\{
    \begin{array}{ll}
      1 & \text{for $i=n$}\\
      0 & \text{otherwise.}
    \end{array}\right.
\end{equation}
The right Artin-Schelter condition is defined analogously using right
$A$-modules instead of left $A$-modules. When $A$ satisfies both the
left and right Artin-Schelter conditions, then $A$ is said to satisfy
the Artin-Schelter condition.

The algebra $A$ is called Artin-Schelter Gorenstein when it satisfies
the Artin-Schelter condition
and the injective dimensions of $_AA$ and $A_A$ are finite and equal. It is called
Artin-Schelter regular when, moreover, ${\mathrm{gl.dim.}}(A)<\infty$.

\section{Homological dualities for Hopf algebras}
\label{sec:homol-dual-hopf}

The objective of this section is to prove Proposition~\ref{prop1}:  the antipode of $H$ is
invertible when $H$ has Van den Bergh duality.
This is done in Section~\ref{sec:invert-antip-van}.
For this purpose, a brief reminder on winding automorphisms is given in
Section~\ref{sec:homol-integr-wind}, next,  a general description
of an inverse dualising complex of $H$ is given in Section~\ref{sec:24}, and
a useful characterisation of homologically smooth Hopf algebras is
proved in Section~\ref{sec:23}. 
Some consequences
regarding Calabi-Yau duality and Nakayama automorphisms are discussed
in Section~\ref{sec:van-den-bergh}. In
particular, Theorem~\ref{thm2} is proved there.

\subsection{Winding automorphisms (see \cite[Section 4.5]{MR2437632})}
\label{sec:homol-integr-wind}
If $H$ satisfies the right Artin-Schelter condition in dimension
$d$, then ${\mathrm{Ext}}^d_{H^{\mathrm{op}}}(\k_H,H)$ is called the \emph{right
  homological integral} of $H$ and denoted by $\int_r$. The algebra
homomorphism $\pi\colon H\to \k$ such that the left $H$-module
structure of ${\mathrm{Ext}}^d_{H^{\mathrm{op}}}(\k_H,H)$ is given by
$h\rightharpoonup \alpha = \pi(h)\alpha$ is also denoted by
$\int_r$. It is called \emph{trivial} if $\int_r=\epsilon$ as maps
$H\to \k$. The \emph{left} homological integral $\int_{\ell}$ is
defined analogously using ${\mathrm{Ext}}_H^d(\,_H\k,H)$.

Let $\pi\colon H\to \k$ be any algebra homomorphism. The following
mappings
\[
\begin{array}{crcl}
  \Xi_{\pi}^\ell\colon & H & \to & H\\
  & h &\mapsto &\pi(h_1)h_2
\end{array}
\ \ 
\begin{array}{crcl}
  \Xi_{\pi}^r \colon & H & \to & H\\
  & h &\mapsto &h_1\pi(h_2)
\end{array}
\]
are algebra automorphism with respective inverses
$\Xi_{\pi\circ S}^\ell$ and $\Xi_{\pi\circ S}^r$. Since
$\pi\circ S^2=\pi$ (see \cite[(E1.2.2)]{MR3250287}), then both
$\Xi_{\pi}^\ell$ and $\Xi_\pi^r$ commute with $S^2$. The automorphisms
$\Xi_{\pi}^\ell$ and $\Xi_\pi^r$ are the \emph{left and right winding
  automorphisms} of $\pi$.

\begin{ex}
  Let $\mathfrak g$ be a $d$-dimensional Lie algebra ($d\in \mathbb
  N$). Assume that $H$ is the universal enveloping algebra $\mathcal
  U(\mathfrak g)$. Following \cite[Corollary 2.2]{MR1762922}, there is
  an isomorphism of right $H$-modules ${\mathrm{Ext}}^d_H(\,_H\k,H)\simeq
  \Lambda^d\mathfrak g^*$. In particular
  \begin{equation}
    \label{eq:50}
    (\forall X\in \mathfrak g)\ \ \int_\ell(X) = {\mathrm{Tr}}({\mathrm{ad}}_X)\,.
  \end{equation}
\end{ex}

\subsection{The inverse dualising complex of a Hopf algebra}
\label{sec:24}
The result below describes ${\mathrm{RHom}}_{H^e}(H,H^e)$ in terms of
${\mathrm{RHom}}_{H^{\mathrm{op}}}(\k_H,H)$. Given any left $H$-module $N$, denote by
$N\!\!\uparrow^{H^e}$ the $H$-bimodule with underlying vector space
$H\otimes N$ and with action by $H^e$ given by
$h(\ell \otimes n)k = S^2(h_1)\ell k\otimes (h_2\rightharpoonup n)$.
Note that $H\!\!\uparrow^{H^e}\simeq H^e$ in ${\mathrm{mod}}(H^e)$; More
precisely, the mapping $H\!\!\uparrow^{H^e}\to H^e$ defined by
$\ell \otimes n \mapsto n_2\otimes S(n_1) \ell$ is an isomorphism of
$H$-bimodules with inverse given by
$h\otimes k \mapsto S^2(h_1) k \otimes h_2$. The assignment
$N\mapsto N\!\!\uparrow^{H^e}$ defines an exact functor
${\mathrm{mod}}(H)\to {\mathrm{mod}}(H^e)$ preserving projectives. It is
isomorphic to the extension-of-scalars functor along the algebra
homomorphism $H\to H^e$ given by $h \mapsto h_2 \otimes S(h_1)$.
The resulting total derived functor
$\mathcal D(H) \to \mathcal D(H^e)$ is also denoted by
$N\mapsto N\!\!\uparrow^{H^e}$.

\begin{prop}
  Let $H$ be a Hopf algebra. Assume that $\k_H$ has a resolution in
  ${\mathrm{mod}}(H^{\mathrm{op}})$ by finitely generated projectives. Then,
  \begin{center}
    ${\mathrm{RHom}}_{H^e}(H,H^e)\simeq {\mathrm{RHom}}_{H^{\mathrm{op}}}(\k_H,H)\!\!\uparrow^{H^e}$ in $\mathcal D((H^e)^{\mathrm{op}})$.
  \end{center}
\end{prop}
\begin{proof}
  Let $P\to \k_H$ be a resolution in ${\mathrm{mod}}(H^{\mathrm{op}})$ by finitely generated projectives. Hence,
  ${\mathrm{RHom}}_{H^{\mathrm{op}}}(\k_H,H)\simeq {\mathrm{Hom}}_{H^{\mathrm{op}}}(P,H)$ in
  $\mathcal D(H)$. The stated isomorphism is proved in three
  steps. First, by proving that a projective resolution $\overline P$
  of $H$ in ${\mathrm{mod}}(H^e)$ can be deduced from $P$. Next, by proving
  that ${\mathrm{RHom}}_{H^e}(H,H^e)\simeq {\mathrm{Hom}}_{H^{\mathrm{op}}}(P,H\otimes
  H)$ (with an adequate structure of right $H$-module on $H\otimes
  H$). Finally, by proving that 
  ${\mathrm{Hom}}_{H^{\mathrm{op}}}(P,H\otimes H)\simeq {\mathrm{Hom}}_{H^{\mathrm{op}}}(P,H)\!\!\uparrow^{H^e}$.


  \textbf{Step 1 - } Given $X\in {\mathrm{mod}}(H^{\mathrm{op}})$, denote by $\overline X$ the
  $H$-bimodule equal to $H\otimes X$ as a vector space and with actions
  of $H$ given by $h(\ell \otimes x) k = h\ell k_1\otimes
  x\leftharpoonup k_2$. Then,
  \begin{itemize}
  \item $\overline{H_H}\simeq H^e$ in ${\mathrm{mod}}(H^e)$; More
    precisely, the mapping $\overline{H_H}\to H^e$ defined by $\ell
    \otimes x\mapsto \ell S(x_1) \otimes x_2$ is an isomorphism of
    $H$-bimodules with inverse given by $h\otimes k \mapsto h k_1
    \otimes k_2$,
  \item $\overline{\k_H}\simeq H$ in ${\mathrm{mod}}(H^e)$.
  \end{itemize}
  Thus, $\overline P \to \overline{\k_H}\simeq H$ is a projective
  resolution in ${\mathrm{mod}}(H^e)$. Hence, ${\mathrm{RHom}}_{H^e}(H,H^e)$ is
  isomorphic to 
  ${\mathrm{Hom}}_{H^e}(\overline P, H^e)$ in 
  $\mathcal D(H^e)$. Here, ${\mathrm{Hom}}_{H^e}(\overline P,H^e)$ is a complex
  of $H$-bimodules in the following sense: $(hfk)(-)=f(-)\times
  (k\otimes h)\in H^e$ if $f\in {\mathrm{Hom}}_{H^e}(\overline P,H^e)$, $h,k\in
  H$.


  \textbf{Step 2 - } The following mapping 
  \[
  \begin{array}{rcl}
    {\mathrm{Hom}}_{H^e}(\overline P, H^e) & \to & {\mathrm{Hom}}_{H^{\mathrm{op}}}(P,H\otimes
                                         H) \\
    f   & \mapsto & f(1\otimes -)
  \end{array}
  \]
  is well-defined provided that $H\otimes H$ is considered as a right
  $H$-module for the action such that $(h\otimes k)\leftharpoonup \ell
  = S(\ell_1) h\otimes k\ell_2$. Indeed, this follows from the
  identity $(1\otimes p\leftharpoonup h) = S(h_1) (1\otimes p)h_2$ in $\overline
  P$. This mapping  is moreover  bijective with
  inverse the mapping ${\mathrm{Hom}}_{H^{\mathrm{op}}}(P,H\otimes H)\to
  {\mathrm{Hom}}_{H^e}(\overline P,H^e)$ which assigns to any $g$ the morphism of
  $H$-bimodules $\overline P\to H^e$ defined by $\ell \otimes p\mapsto
  \ell g(p)$. Finally, it is an isomorphism of $H$-bimodules if
  ${\mathrm{Hom}}_{H^{\mathrm{op}}}(P,H\otimes H)$ is endowed with the action of
  $H^e$ such that $(hfk)(-)=f(-)\times (k\otimes h)\in H^e$ (hence
  $H\otimes H$ is a right $H\otimes H^e$-module). Thus,
  ${\mathrm{RHom}}_{H^e}(H,H^e)\simeq {\mathrm{Hom}}_{H^{\mathrm{op}}}(P,H\otimes H)$ in
  $\mathcal D(H^e)$.


  \textbf{Step 3 - } In order to get the announced description of ${\mathrm{RHom}}_{H^e}(H,H^e)$,
  it is necessary to transform ${\mathrm{Hom}}_{H^{\mathrm{op}}}(P,H\otimes H)$. Note
  that ${\mathrm{Hom}}_{H^{\mathrm{op}}}(P,H)$ is a complex of left $H$-modules in a
  natural way. This defines ${\mathrm{Hom}}_{H^{\mathrm{op}}}(P,H)\!\!\uparrow^{H^e}$. The following mapping is
  well-defined
  \[
  \begin{array}{rcl}
    {\mathrm{Hom}}_{H^{\mathrm{op}}}(P,H)\!\!\uparrow^{H^e} & \to & {\mathrm{Hom}}_{H^{\mathrm{op}}}(P,H\otimes
                                                     H) \\
    \ell \otimes \theta & \mapsto &
                                   \left(
                                   p\mapsto S(\theta(p)_1) \ell
                                   \otimes \theta(p)_2
                                   \right)
                                                                                        
  \end{array}
  \]
  It is moreover a morphism of complexes of $H$-bimodules. When $P$ is replaced by
  $H$, then it identifies with the mapping $H\otimes H\to H\otimes H$
  defined by $\ell \otimes \theta \mapsto S(\theta_1)\ell \otimes
  \theta_2$; This is an isomorphism with inverse given by $h\otimes k
  \mapsto S^2(k_1) h \otimes k_2$. Since $P$ consists of finitely
  generated projective $H$-modules, it follows that the complexes of
  $H$-bimodules ${\mathrm{Hom}}_{H^{\mathrm{op}}}(P,H)\!\!\uparrow^{H^e}$ and
  ${\mathrm{Hom}}_{H^{\mathrm{op}}}(P,H\otimes H)$ are isomorphic. Thus,
  ${\mathrm{RHom}}_{H^e}(H,H^e)\simeq {\mathrm{Hom}}_{H^{\mathrm{op}}}(P,H)\!\!\uparrow^{H^e}$.
\end{proof}

Keep the setting of the previous result. Taking cohomology shows that,
for every $n\in\mathbb N$, there is an isomorphism of $H$-bimodules
\[
{\mathrm{Ext}}^n_{H^e}(H,H^e)\simeq H\otimes {\mathrm{Ext}}_{H^{\mathrm{op}}}^n(\k_H,H)
\]
where the right hand-side term is endowed with the structure of
$H$-bimodule such that
$h (\ell \otimes e) k = S^2(h_1)\ell k \otimes h_2\rightharpoonup
e$.
See \cite[Corollary 2.2]{MR2678828} (and \cite[Section
4.5]{MR2437632}) for a previous similar description when $H$ is
Artin-Schelter Gorenstein (and with invertible antipode, respectively).

\subsection{Homological smoothness of Hopf algebras}
\label{sec:23}

The following result seems to be well-known at least when $H$ is
noetherian. See \cite[Lemma 2.4]{MR3341818} for a proof using that $S$
is invertible.  In the present situation, where $S$ need not be
invertible, it follows from the properties of the functor
$\overline{?} \colon \mathrm{mod}(H^{\mathrm{op}}) \to
\mathrm{mod}(H^e)$
considered in the proof in \ref{sec:24} and from the corresponding
ones of $\k\otimes_H- \colon \mathrm{mod}(H^e) \to \mathrm{mod}(H^{\mathrm{op}})$.
\begin{lem}
  The Hopf algebra $H$ is homologically smooth if and only if
  $\k_H\in {\mathrm{per}}(H^{\mathrm{op}})$. If it is so, then
  ${\mathrm{pd}}_{H^e}(H)={\mathrm{pd}}_{H^{\mathrm{op}}}(\k_H)$.
\end{lem}

\subsection{Invertibility of the antipode}
\label{sec:invert-antip-van}

\subsubsection{A sufficient condition for the right Artin-Schelter
  property}
\label{sec:suff-cond-right}

The following lemma is a key-step in the proof of
Proposition~\ref{prop2}, it is very similar to the lemma in
\cite[Section 1.2]{MR2437632} except that the condition that $S$ is
invertible is dropped here. The proof here is adapted from the proof
there.

\begin{lem}
  Let $H$ be a Hopf algebra. Assume the following:
  \begin{enumerate}[(a)]
  \item ${\mathrm{id}}(\,_HH)<\infty$,
  \item $\k_H$ has a resolution in ${\mathrm{mod}}(H^{\mathrm{op}})$ by finitely
    generated projective modules,
  \item there exists $d\in \mathbb N$ such that ${\mathrm{Ext}}_{H^{\mathrm{op}}}^i(\k_H,H)$ is finite dimensional if $i=d$ and zero
    otherwise. 
  \end{enumerate}
  Then, ${\mathrm{dim}}_\k {\mathrm{Ext}}^d_H(\,_H\k,H) =
  {\mathrm{dim}}_\k {\mathrm{Ext}}^d_{H^{\mathrm{op}}}(\k_H,H)=1$.
\end{lem}
\begin{proof}
  According to \cite[Section 3.2]{MR2437632}, conditions (a) and (b)
  grant the existence of Ischebeck's spectral sequence
  (\cite[1.8]{MR0237613})
  \[
  E_2^{p,q}:={\mathrm{Ext}}_H^p({\mathrm{Ext}}_{H^{\mathrm{op}}}^{-q}(\k_H,H),H)
  \Rightarrow \mathrm{Tor}^H_{-p-q}(\k_H, H) =
\left\{
    \begin{array}{ll}
      \k & \text{if $p+q=0$} \\
      0 & \text{otherwise.}
    \end{array}\right.
  \]
  Because of condition (c), the spectral sequence degenerates at
  $E_2$. In particular,
  \begin{equation}
    \label{eq:40}
    {\mathrm{dim}}_{\k}{\mathrm{Ext}}_H^d({\mathrm{Ext}}_{H^{\mathrm{op}}}^d(\k_H,H),H) =
    1\,.
  \end{equation}
Following the ideas in the proof of \cite[Lemma 1.11]{MR1482982},
denote by $V$ the finite dimensional left $H$-module ${\mathrm{Ext}}_{H^{\mathrm{op}}}^d(\k_H,H)$. 
Endow ${\mathrm{Hom}}_\k(V,H)$ with its usual structure of left $H$-module.
The canonical bijection  ${\mathrm{Hom}}_\k(V,H)\xrightarrow{\sim}
{\mathrm{Hom}}_\k(\k, {\mathrm{Hom}}_\k(V,H))$ restricts to an isomorphism ${\mathrm{Hom}}_H(V,H)
\xrightarrow{\sim} {\mathrm{Hom}}_H(\,_H\k,{\mathrm{Hom}}_\k(V,H))$. Deriving this
isomorphism yields that ${\mathrm{Ext}}^d_H(V,H)\simeq {\mathrm{Ext}}^d_H(\,_H\k,{\mathrm{Hom}}_\k(V,H))$ (see  \cite[Proposition
1.3]{MR1482982}). Besides, consider $H\otimes V^*$ as a left $H$-module in
the usual way also. The canonical mapping $H\otimes V^*\to
{\mathrm{Hom}}_\k(V,H)$ 
is $H$-linear. And it is bijective since ${\mathrm{dim}}_\k
V<\infty$. Therefore, ${\mathrm{Ext}}^d_H(\,_H\k,{\mathrm{Hom}}_\k(V,H))\simeq {\mathrm{Ext}}^d_H(\,_H\k,H\otimes V^*)$.
Now, the left $H$-module $H\otimes V^*$ (with action given by
$h\rightharpoonup (\ell \otimes \varphi) = h_1\rightharpoonup \ell
\otimes h_2\rightharpoonup \varphi$) is isomorphic to the free of rank
${\mathrm{dim}}_\k\,V$ left $H$-module $H\otimes V$ with action 
given by $h\rightharpoonup (\ell \otimes \varphi) = h\ell \otimes
\varphi$. Indeed, the mapping from the former to the latter defined by
$\ell \otimes \varphi \mapsto \ell_1\otimes S(\ell_2)\rightharpoonup
\varphi$ is an isomorphism in ${\mathrm{mod}}(H)$ with inverse given by
$\ell \otimes \varphi \mapsto \ell_1\otimes \ell_2\rightharpoonup
\varphi$. Thus, 
\begin{equation}
  \label{eq:41}
  {\mathrm{Ext}}_H^d(V,H) \simeq {\mathrm{Ext}}^d_H(\,_H\k,H)^{{\mathrm{dim}}\,V}\,.
\end{equation}
Combining \eqref{eq:40} and \eqref{eq:41} yields the announced equalities.
\end{proof}

\subsubsection{}
\label{sec:38}
The following result entails Proposition~\ref{prop2}. Part (2) is
proved in \cite[Section 4.4]{MR2437632} when $H$ is Artin-Schelter
Gorenstein and noetherian and $S$ is invertible.
\begin{prop}
 If a Hopf algebra $H$ has Van den Bergh duality in dimension $d$, then
\begin{enumerate}
\item $H$ is right Artin-Schelter regular in dimension $d$, and
\item  ${\mathrm{Ext}}^d_{H^e}(H,H^e)$ is isomorphic to $^{S^2\circ
    \Xi_{\int_r}^r}H$ as an $H$-bimodule.
\end{enumerate}
As a consequence, the antipode $S$ is invertible.
\end{prop}
\begin{proof}
Note that $H$ has finite global dimension equal to $d$ because $H\in {\mathrm{per}}(H^e)$ and that $d={\mathrm{pd}}_{H^e}(H)$. 
  There exists a projective resolution $P\to \k_H$ in ${\mathrm{mod}}(H^{\mathrm{op}})$ having length $d$ and consisting of finitely
  generated projective 
  right $H$-modules (see \ref{sec:23}). Consider the resulting
  complex of $H$-bimodules 
  $\overline P$ as introduced in the proof in \ref{sec:24}.
  In particular, $\overline P$ is a
  projective resolution with length $d$ of $H$ in ${\mathrm{mod}}(H^e)$. 
  Note that ${\mathrm{Hom}}_{H^{\mathrm{op}}}(P,H)$ is a complex of projective left
  $H$-modules concentrated in degrees $0,\ldots,d$. Similarly,
  ${\mathrm{Hom}}_{H^{\mathrm{op}}}(P,H)\!\!\uparrow^{H^e}$ is a complex of
  projective right $H^e$-modules whose cohomology is ${\mathrm{Ext}}_{H^e}^*(H,H^e)$ (see \ref{sec:24}).

  
  (1) 
  Denote by $U$ the right $H^e$-bimodule ${\mathrm{Ext}}^d_{H^e}(H,H^e)$.  Because of the assumption on the length of
  $P$, and since $H$ has Van den Bergh duality, there is a
  quasi-isomorphism in $\mathcal C(H^e)$ 
  \[
{\mathrm{Hom}}_{H^{\mathrm{op}}}(P,H)\!\!\uparrow^{H^e}\to U[-d]\,.
\]
This is a homotopy equivalence in ${\mathrm{mod}}(H^{\mathrm{op}})$ because both sides
are bounded complexes of projective right $H$-modules. Applying
  $-\underset{H}{\otimes}\,_H\k$ therefore yields a quasi-isomorphism
  in $\mathcal C(H)$
  \[
  \left({\mathrm{Hom}}_{H^{\mathrm{op}}}(P,H)\!\!\uparrow^{H^e}\right)\underset{H}{\otimes}\,_H\k \to
  U[-d]\underset{H}{\otimes}\,_H\k\,.
  \]
A direct computation shows that
  $\left({\mathrm{Hom}}_{H^{\mathrm{op}}}(P,H)\!\!\uparrow^{H^e}\right)\underset{H}{\otimes}
  \,_H\k\simeq {\mathrm{Hom}}_{H^{\mathrm{op}}}(P,H)$ in $\mathcal C(H)$.
  Consequently,
  \[
  {\mathrm{Ext}}^i_{H^{\mathrm{op}}}(\k_H,H) \simeq \left\{
    \begin{array}{ll}
      U\underset H\otimes \,_H\k & \text{if $i=d$} \\
      0 & \text{otherwise.}
    \end{array}
    \right.
    \]
    Now, as any invertible $H$-bimodule, $U$ is finitely generated in ${\mathrm{mod}}(H^{\mathrm{op}})$. Applying $-\underset{H}\otimes \,_H\k$ to a free 
    of finite rank cover of $U$ in ${\mathrm{mod}}(H^{\mathrm{op}})$ yields that
    $U\underset H\otimes \,_H\k$ is finite dimensional. Therefore,
    \ref{sec:suff-cond-right} applies here. Thus, $H$ is
right Artin Schelter regular in dimension $d$.


(2)  As explained earlier, $U\simeq H^d({\mathrm{Hom}}_{H^{\mathrm{op}}}(P,H)\!\!\uparrow^{H^e})$. The functor ${\mathrm{mod}}(H)\to
{\mathrm{mod}}(H^e)$ defined by $N\mapsto N\!\!\uparrow^{H^e}$ is
exact. Therefore, $H^d({\mathrm{Hom}}_{H^{\mathrm{op}}}(P,H)\!\!\uparrow^{H^e})\simeq {\mathrm{Ext}}^d_{H^{\mathrm{op}}}(\k_H,H)\!\!\uparrow^{H^e}$ in ${\mathrm{mod}}(H^e)$. Moreover, the construction in \ref{sec:24} and the
definition of $\int_r\colon H\to \k$ yield that ${\mathrm{Ext}}^d_{H^{\mathrm{op}}}(\k_H,H)\!\!\uparrow^{H^e}\simeq \,^{S^2\circ \Xi_{\int_r}^r}H$ in
${\mathrm{mod}}(H^e)$. Thus, $U\simeq \,^{S^2\circ \Xi_{\int_r}^r}H$. This
proves (2).


Since $U$ is invertible as an $H$-bimodule and $U\simeq \,^{S^2\circ
  \Xi_{\int_r}^r}H$ in ${\mathrm{mod}}(H^e)$, then $S^2\circ
\Xi_{\int_r}^r$ is an automorphism of $H$. And hence so is $S^2$. Thus,
$S$ is invertible.
\end{proof}

\subsection{On Van den Bergh duality of Hopf algebras}
\label{sec:van-den-bergh}

For later purposes, here are some consequences of \ref{sec:38} many of
which were proved in \cite{MR2437632}, \cite{MR2678828}, and
\cite{MR3250287} assuming that $S$ is invertible and/or that
$H$ is noetherian.

\subsubsection{Summary on Van den Bergh duality}
\label{sec:25}
The following characterisation is obtained from \ref{sec:invert-antip-van}
and from the main results in \cite{MR2437632}. When $S$ is assumed to
be invertible and $H$ to be noetherian, the same result is already
proved in \cite[Lemma
1.3]{MR3250287}.
\begin{thm}
  \label{thm1}
  Let $H$ be a Hopf algebra.  The
  following conditions are equivalent
  \begin{enumerate}[(i)]
  \item $H$ has Van den Bergh duality,
  \item $\k_H\in
    {\mathrm{per}}(H^{\mathrm{op}})$, $S$ is invertible and the graded
    $\k$-vector space ${\mathrm{Ext}}^*_{H^{\mathrm{op}}}(\k_H,H)$ is
    finite-dimensional and concentrated in one degree,
  \item $H$ is skew Calabi-Yau.
  \end{enumerate}
  Under any of these conditions, the homological dimensions involved
  in $(i)$, $(ii)$ and $(iii)$ coincide and
  $S^{-2}\circ \Xi_{\int_r\circ S}^r$ is a Nakayama automorphism of
  $H$.
\end{thm}
\begin{proof}
  
The implication $(i)\Rightarrow (ii)$ is proved in
  \ref{sec:38} and \ref{sec:23}.

  The implication $(ii)\Rightarrow (iii)$ is proved in
  \cite{MR2437632} when $H$ is noetherian and may be adapted to the
  present situation. Here is a proof for the convenience of the
  reader. Assume $(ii)$. Then, $H$ is homologically smooth (see
  \ref{sec:23}). Let $d\in \mathbb N$ be such that ${\mathrm{dim}}_\k\,{\mathrm{Ext}}^i_{H^{\mathrm{op}}}(\k_H,H)$ is finite if $i=d$ and $0$
  otherwise.  Therefore, ${\mathrm{dim}}_\k{\mathrm{Ext}}^d_{H^{\mathrm{op}}}(\k_H,H)=1$ (see~\ref{sec:suff-cond-right}). It
  follows from \ref{sec:24} that ${\mathrm{Ext}}^i_{H^e}(H,H^e)$ is isomorphic to $^{S^2\circ \Xi_{\int_r}^r}H
  \simeq H^{S^{-2}\circ \Xi_{\int_r\circ S}^r}$ in ${\mathrm{mod}}(H^e)$
  when $i=d$ and is $0$ otherwise. This proves that $(ii)\Rightarrow (iii)$.

The implication
  $(iii)\Rightarrow (i)$ follows from the definition.
\end{proof}

\subsubsection{Nakayama automorphisms}
\label{sec:nak}

Using \ref{sec:25} yields the following relationship between the right
homological integral and the Nakayama automorphisms of $H$. Part (1)
was proved in \cite[Section 0.3]{MR2437632} assuming that $S$ is
invertible and $H$ is noetherian.
\begin{prop}
  If a Hopf algebra  $H$ has Van den Bergh duality in
  dimension $d$, then the antipode is invertible and
  \begin{enumerate}
  \item $\left(S^2\circ \Xi_{\int_r}^r\right)^{-1}=S^{-2}\circ
    \Xi_{\int_r\circ S}^r$ is a Nakayama automorphism of $H$.
  \item If $\mu\in {\mathrm{Aut}}_{\k-{\mathrm{alg}}}(H)$ is any Nakayama
    automorphism of $H$ then $\epsilon \circ \mu^{-1}$ is the right
    homological integral of $H$,
  \end{enumerate}
\end{prop}
\begin{proof}
  (1) follows from \ref{sec:25}. (2) follows from (1) and from the
  fact that   Nakayama automorphisms differ from one another by an
  inner automorphism. 
\end{proof}

\subsubsection{Calabi-Yau duality}
\label{sec:char-calabi-yau}

Combining \ref{sec:25} and \ref{sec:nak} yields the following.
\begin{cor}
  Let $H$ be a Hopf algebra. Then $H$ is Calabi-Yau in dimension $d$
  if and only if the following conditions hold
  \begin{enumerate}[(a)]
  \item $\k_H\in {\mathrm{per}}(H^{\mathrm{op}})$,
  \item $S^2$ is an inner automorphism of $H$,
  \item $H$ is satisfies the right Artin-Schelter condition and its right
    homological integral is trivial.
  \end{enumerate}
\end{cor}

\section{Equivariant modules}
\label{sec:equivariant-modules}

This section develops properties and constructions based on dg
$A$-bimodules which are $H_{S^{2i}}$-equivariant ($i\in \mathbb Z$)
and which are needed to describe
$\mathrm{RHom}_{\Lambda}(\Lambda,\Lambda^e)$
(Section~\ref{sec:proof-main-result}), to describe the deformed
Calabi-Yau completions of $\Lambda$
(Section~\ref{sec:appl-constr-calabi}), and to describe the Nakayama
automorphisms of $\Lambda$
(Section~\ref{sec:appl-dual-non}). Section~\ref{sec:algebras-delta_i}
introduces a dg algebra $\Delta_i$ such that $\mathcal C(\Delta_i)$
consists of those equivariant dg
bimodules. Section~\ref{sec:asharp-h-bimodules} gives details on
certain extension-of-scalars functors $D \mapsto D\sharp^\sigma H$ from
$\mathcal C(\Delta_i)$ to $\mathcal C(\Lambda^e)$.
Section~\ref{sec:AdjunctionDelta} explains how $H$ acts on morphism
spaces between $H_{S^{2i}}$-equivariant dg
$A$-bimodules. Sections~\ref{sec:equiv-acti-duals} and
\ref{sec:tens-prod-equiv} detail the behaviour of $H_{S^{2i}}$-equivariant dg
$A$-bimodules under ${\mathrm{Hom}}_A(-,A)$ and under tensor products over $A$,
respectively. Finally, Section~\ref{sec:invertible-dg-lambda} proves
that such a bimodule $D$ is invertible over $A$ if and only if
$D\sharp^\sigma H$ is invertible over $\Lambda$.

$\Delta_0$ was introduced in \cite[Definition 3.1]{MR2366958} and
applied to smash products whether to investigate Calabi-Yau duality
(see \cite{MR2905560,MR3250287}) or Hochschild cohomology (see
\cite{MR3341818}).

\subsection{Dg algebras which dg modules are equivariant $A$-bimodules}
\label{sec:algebras-delta_i}\label{sec:running-example}

\begin{defn}
  Let $i\in \mathbb Z$ and assume that $S$ is invertible when
  $i<0$. Define $\Delta_i$ to be the dg algebra whose
  underlying complex of vector spaces is $A^e\otimes H$ and whose
  (associative) product is given by the rule
  \[
    \begin{array}{c}
      (a\otimes b\otimes h) \times (a'\otimes b'\otimes k) = \\
      (-1)^{\deg(b)\cdot (\deg(a')+\deg(b'))} \times
  (a (h_1\rightharpoonup a')) \otimes ((S^{2i}(h_3)\rightharpoonup
      b')b) \otimes h_2k\,.
      \end{array}
  \]
\end{defn}
It is elementary although tedious to check that the product is indeed
associative. Note that the product in $\Delta_i$ is determined by the
three following properties
  \begin{itemize}
  \item the product is associative,
  \item $A^e$ is a dg subalgebra of $\Delta_i$ \emph{via} the mapping $A^e\to A^e\otimes H,\, a\otimes b \mapsto
    a\otimes b\otimes 1$,
  \item $H$ is dg subalgebra of $\Delta_i$ \emph{via} the mapping $H\to A^e\otimes H,\, h\mapsto 1\otimes 1\otimes
    h$,
  \item the product satisfies the following identity in $\Delta_i$
    \begin{equation}
      \label{eq:21}
      h\times (a\otimes b) =
    (h_1\rightharpoonup a \otimes S^{2i}(h_3)\rightharpoonup
    b) \times h_2\,.
    \end{equation}
  \end{itemize}
  When $S$ is invertible, $\Delta_i$
  features the following useful identity
\begin{equation}
  \label{eq:32}
  (a\otimes b) \times h = h_2 \times (S^{-1}(h_1)\rightharpoonup a
  \otimes S^{2i+1}(h_3)\rightharpoonup b)\,.
\end{equation}

In general, the following mapping is a homomorphism of dg algebras
that makes of $\Lambda^e$ a left (and right) dg $\Delta_i$-module
\begin{equation}
  \label{eq:31}
  \begin{array}{rcl}
    \Delta_i & \to & \Lambda^e\\
    (a\otimes b)\times h & \mapsto & (a\otimes b) \times
                                     (h_1\otimes S^{2i+1}(h_2))\,.
  \end{array}
\end{equation}
The image of this mapping is the dg subalgebra of $\Lambda^e$
generated by $A^e\cup \{h_1\otimes S^{2i+1}(h_2)\ |\ h\in H\}$ because
of the following identity in $\Lambda^e$,
\[
(h_1\otimes S^{2i+1}(h_2))\times (a\otimes b) =
(h_1\rightharpoonup a \otimes S^{2i}(h_4)\rightharpoonup b)
\times (h_2\otimes S^{2i+1}(h_3))\,.
\]
When $S$ is invertible, this mapping is injective and has a retraction
given by
$ah \otimes bk \mapsto a\otimes S^{-1}(k)\rightharpoonup b \otimes h$.

If $H$ is cocommutative, then $\Delta_i$ does not depend on $i$ and is
isomorphic to a smash product of $A^e\sharp H$ (see \cite[Remark
1.7]{MR2905560}). This does not hold true in general.


The left dg $\Delta_i$-modules are the $H_{S^{2i}}$-equivariant left
dg $A$-bimodules defined in \cite[Definition 2.2]{MR3250287}. The
latter are the left dg $A$-bimodules $M$ endowed with a structure of
left $H$-module (which preserves the degree and is compatible with the
differential of the complex) in such a way that the following identity
holds in $M$
\begin{equation}
  \label{eq:33}
h\rightharpoonup (amb)=(h_1\rightharpoonup a) (h_2\rightharpoonup m)
(S^{2i}(h_3)\rightharpoonup b)\,.
\end{equation}
In particular, $A\in \mathcal C(\Delta_0)$.

\begin{ex}
  Let $A=\k[x_1,\ldots,x_n]$. Let $H=\mathcal U(\mathfrak g)$ as in
  the example of \ref{sec:homol-integr-wind}. Let
  $\mathfrak g\to {\mathrm{Der}}_\k(A),\,X\mapsto \partial_X$ be a
  homomorphism of Lie algebras.  Hence, $A$ is an $H$-module
  algebra. The structure of $\Delta_0$-module on $A^e$ is given
  by
\[
X\rightharpoonup (a\otimes b) = \partial_X(a)\otimes b +
a\otimes \partial_X(b)\,.
\]

The sequence $(x_i\otimes 1 - 1\otimes x_i)_{1\leqslant i\leqslant n}$
of the commutative ring $A^e$ is regular and the quotient of $A^e$ by
the ideal generated by this sequence is isomorphic to $A$. Recall that
the Koszul resolution $(K^\bullet,d_K)$ of $_AA_A$ is as follows.
Denote $\oplus_{i=1}^n\k\cdot x_i$ by $V$. Let $K^\bullet$ be the
graded vector space $\Lambda^{-\bullet}_{A^e}(A\otimes V \otimes A)$
concentrated in degrees $-n,-n+1,\ldots,-1,0$. This is a
graded-commutative algebra over the commutative ring $A^e$ in the
usual way (the product is denoted by $\wedge$). By a
\emph{skew derivation} of degree $\ell$ of $K^\bullet$ is meant a
homogeneous $\k$-linear mapping
$d \colon K^\bullet \to K^{\bullet+\ell}$ such that, for all
homogeneous $\omega_1,\omega_2\in K^\bullet$,
\[
d(\omega_1\wedge \omega_2) = d(\omega_1)\wedge \omega_2
+(-1)^{\ell\mathrm{deg}(\omega_1)} \omega_1\wedge d(\omega_2)\,.
\]
The skew derivations of degree $0$ are usual $\k$-linear algebra
derivations.  Denote by $d_K$ the unique skew derivation of degree
$+1$ of $K^\bullet$ such that
$d_K(1\otimes v\otimes 1) = v\otimes 1-1\otimes v\in A^e$ for all
$v\in V$.  Hence, $d_K$ is $A^e$-linear and squares to zero.  By
construction, $(K^\bullet,d_K)$ is the Koszul complex of the sequence
$(x_i\otimes 1-1\otimes x_i)_{1\leqslant i\leqslant n}$ of the
commutative ring $A^e$. This is a projective resolution of $A$ as an
$A$-bimodule.

Here is how to endow $(K^\bullet,d_K)$ with an action of $H$ for which
$K^\bullet$ lies in $\mathcal C(\Delta_0)$.  Given any $m\in A$
written as a linear combination of monomials
\[
m = \lambda_0 + \sum_{r\geqslant 1}\ \sum_{1\leqslant j_1\leqslant \cdots
  \leqslant j_r\leqslant n}
\lambda_{j_1,\ldots,j_r}\ x_{j_1}x_{j_2}\cdots x_{j_r}\,,
\]
where $\lambda_0,\lambda_{j_1,\ldots,j_r}\in \k$, use the symbolic piece
of notation
\[
\sum_im'_i\otimes x_i\otimes m''_i
\]
to denote
\[
\sum_{r\geqslant 1}\ \sum_{1\leqslant j_1\leqslant \cdots \leqslant
  j_r\leqslant n} \lambda_{j_1,\ldots,j_r} \sum_{t=1}^r x_{j_1} \cdots
x_{j_{t-1}}\otimes x_{j_t} \otimes x_{j_{t+1}} \cdots x_{j_r}\in
A\otimes V\otimes A\,.
\]
For every $X\in \mathfrak g$, there exists a unique skew derivation of
degree $0$ of $K^\bullet$ denoted by $\partial_X$ and such that
\begin{enumerate}[(a)]
\item $\partial_X(a\otimes b) = X\rightharpoonup (a\otimes b)$ for all
  $a\otimes b\in A\otimes A$,
\item and $\partial_X(1\otimes v\otimes 1) = \sum_i \partial_X(v)_i' \otimes
x_i\otimes \partial_X(v)_i''$ for all $v\in V$.
\end{enumerate}
Then, $d_K\circ \partial_X - \partial_X \circ d_K$ is a
skew derivation of degree $+1$ of $K^\bullet$. Therefore, it vanishes
on $K^0$. In view of (b), it vanishes on $1\otimes V\otimes 1$. Thus,
it is zero, that is, $\partial_X \colon K^\bullet \to K^\bullet$ is a
morphism of complexes of vector spaces. The family
$(\partial_X)_{X\in \mathfrak g}$ hence yields an action of $H$ on
$K^\bullet$. In view of (a), $K^\bullet$ is a dg $\Delta_0$-module.
\end{ex}

\subsection{$A\sharp H$-bimodules arising from equivariant $A$-bimodules}
\label{sec:asharp-h-bimodules}

Let $i\in \mathbb Z$ and assume that $S$ is invertible whenever
$i<0$. Let $D$ be an $H_{S^{2i}}$-equivariant 
dg $A$-bimodule (equivalently, a left dg  $\Delta_i$-module). Let $\sigma\in {\mathrm{Aut}}_{\k-{\mathrm{alg}}}(H)$ be such that the following identity holds in $H$
\begin{equation}
  \label{eq:43}
\sigma(h)_1 \otimes \sigma(h)_2 = S^{2i}(h_1)\otimes \sigma(h_2)\,.  
\end{equation}
(for instance, $\sigma=S^{2i}\circ
\Xi_\pi^r$ for some algebra homomorphism $\pi\colon H\to
\k$). When $S$ is invertible, this condition is equivalent to the
following identity in $H$
\begin{equation}
  \label{eq:42}
  \sigma^{-1}(h)_1\otimes \sigma^{-1}(h)_2 =
  S^{-2i}(h_1)\otimes \sigma^{-1}(h_2)\,.
\end{equation}

\subsubsection{}
\label{sec:44}
According to \cite[Lemma 2.5]{MR3250287}, the following actions endow
$D\otimes H$ with a structure of dg  $\Lambda$-bimodule (see \eqref{eq:33})
\begin{equation}
  \label{eq:35}
  \left\{
    \begin{array}{l}
      a  (d\otimes \ell) b= ad(\ell_1\rightharpoonup b)
      \otimes \ell_2 \\
      h(d\otimes \ell)k= h_1\rightharpoonup d
      \otimes \sigma(h_2)\ell k\,.
    \end{array}
  \right.
\end{equation}
In the sequel, it is denoted by $D\sharp \,^\sigma H$. When $S$ is
invertible, $D\sharp \,^\sigma H$ features the following identity
\begin{equation}
  \label{eq:44}
  d\otimes \ell = \sigma^{-1}(\ell_2) ( S^{-1-2i}(\ell_1)\rightharpoonup d
  \otimes 1)\,.
\end{equation}

In \cite[Definition 3.2]{MR3341818}, the construction $D\sharp H$ is
extended to $D\sharp X$, where $X$ is any complex of Hopf bimodules
over $H$.

\subsubsection{}
\label{sec:43}
Assume that $S$ is invertible.  The following lemma is a functorial
interpretation of the previous construction. The mapping
$\varphi\colon \Lambda\to \Lambda$ defined by
$\varphi(ah)=a\,\sigma(S^{-2i}(h))$ is an automorphism of the dg
algebra $\Lambda$. Hence, $\Lambda\otimes \,^\varphi\Lambda$ is a
$\Lambda^e-\Delta_i$-bimodule. As a left dg $\Lambda^e$-module it
equals $_\Lambda\Lambda \otimes \Lambda_\Lambda$. Its structure of
right dg $\Delta_i$-module is given by
$(\lambda\otimes \lambda')\cdot (a\otimes b \otimes h) = \pm\lambda a
h_1 \otimes \varphi(S^{2i+1}(h_2) b)\lambda'$
(where $\lambda,\lambda'\in \Lambda$). In particular,
$(\Lambda\otimes\,^\varphi\Lambda)\underset{\Delta_i}\otimes D$
inherits of a structure of dg $\Lambda$-bimodule.
\begin{lem}
  Keep the setting introduced previously. Then, $D\sharp \,^\sigma
  H\simeq \left(\Lambda\otimes\,^{\varphi}\Lambda\right)
  \underset{\Delta_i}{\otimes}D$ in $\mathcal C(\Lambda^e)$.
\end{lem}
\begin{proof}
 Denote $\left(\Lambda\otimes\,^{\varphi}\Lambda\right)
  \underset{\Delta_i}{\otimes}D$ by $\overline D$. Note the
  identity in $\overline D$
  \begin{equation}
    \label{eq:34}
          (h\otimes k)\underset{\Delta_i}{\otimes} d
          =
              (1\otimes \sigma 
    (h_2) k) \underset{\Delta_i}{\otimes} (h_1\rightharpoonup
    d)\,.
  \end{equation}
  Indeed, the right hand-side term equals $\left( h_1\otimes
    \varphi(S^{2i+1}(h_2)) \sigma(h_3)k \right)
  \underset{\Delta_i}\otimes d$; This is equal to $(h\otimes
  k)\underset{\Delta_i}{\otimes} d$.
Therefore, the linear mapping 
from   $(\Lambda\otimes \Lambda)\otimes D$  to $D\otimes H$ defined by
$(ha \otimes bk) \otimes d \mapsto \pm(h_1\rightharpoonup (adb))
\otimes \sigma (h_2) 
k$ induces a linear mapping
\[
\begin{array}{crcl}
  \Phi \colon & \overline D & \to & D\otimes H \\
  & (ha \otimes bk) \underset{\Delta_i}{\otimes}d
                            & \mapsto &
                                        \pm (h_1\rightharpoonup
                                        (adb))
                                        \otimes
                                        \sigma(h_2)
                                        k \,. 
\end{array}
\]
Here, $\pm$ is the sign $(-1)^{{\mathrm{deg}}(b)\cdot {\mathrm{deg}}(d)}$.
In view of \eqref{eq:34}, it is bijective and its inverse is the
mapping defined by  $d\otimes \ell \mapsto (1\otimes
\ell)\underset{\Delta_i}{\otimes} d$.
In order to prove the assertion of the lemma, it therefore suffices to
prove that the latter mapping is a morphism of dg
$\Lambda$-bimodules from $D\sharp\,^\sigma H$ to $\overline D$. This
is done in the computations below made in 
$\overline D$: 
\[
\begin{array}{rcl}
  a((1\otimes \ell)\underset{\Delta_i}{\otimes} d)b
  & = &
        (-1)^{{\mathrm{deg}}(b){\mathrm{deg}}(d)}(a\otimes \ell b) \underset{\Delta_i}{\otimes} d\\
  & = &
        (1\otimes \ell_2) \underset{\Delta_i}{\otimes}
        ad(\ell_1\rightharpoonup b)\,.
\end{array}
\]
and
\[
\begin{array}{rcl}
  h ((1\otimes \ell) \underset{\Delta_i}{\otimes} d)k
  & = &
        (h\otimes \ell k) \underset{\Delta_i}{\otimes} d \\
  & = &
        (1 \otimes \sigma (h_2) \ell k)
        \underset{\Delta_i}{\otimes} (h_1\rightharpoonup d)\,.
\end{array}
\]
\end{proof}

\subsubsection{}
\label{sec:45}
Assume that $S$ is invertible and define a dg $\Lambda$-bimodule
$H^{\sigma^{-1}} \sharp D$ as follows. Its underlying complex of
vector spaces is $H\otimes D$ and the actions of $\Lambda$ are given
by
\begin{equation}
  \label{eq:45}
  \left\{
    \begin{array}{rcl}
      a(\ell \otimes d) b & = & \ell_2\otimes
                                (S^{-1}(\ell_1)\rightharpoonup a)
                                d b \\
      h (\ell \otimes d) k & = &
                                 h\ell \sigma^{-1}(k_2) \otimes
                                 S^{-2i-1}(k_1)\rightharpoonup d\,.
    \end{array}
    \right.
  \end{equation}

  \begin{lem}
    \label{sec:40}
    Keep the setting introduced previously. The following mapping is
    an isomorphism in $\mathcal C(\Lambda^e)$
    \[
    \begin{array}{rcl}
      D\sharp \,^\sigma H & \to & H^{\sigma^{-1}}\sharp D \\
      d\otimes \ell & \mapsto & \sigma^{-1}(\ell_2) \otimes
                                S^{-2i-1}(\ell_1)\rightharpoonup d\,.
    \end{array}
    \]
  \end{lem}
  \begin{proof}
    Denote this mapping by $\varphi$. Consider the mapping
    $H^{\sigma^{-1}}\sharp D \to D\sharp \,^\sigma H$ defined by
    $\ell\otimes d \mapsto \ell_1\rightharpoonup d \otimes \sigma(\ell_2)$. In
    view of \eqref{eq:43}, \eqref{eq:42} and \eqref{eq:44}, it is an inverse of
    $\varphi$.
    That $\varphi$ is a morphism in $\mathcal C(\Lambda^e)$ follows
    from the computations below
    \[
    \begin{array}{rcl}
      \varphi(h(d\otimes \ell) k)
      & = &
            \varphi(h_1\rightharpoonup d \otimes \sigma(h_2)\ell k) \\
      & = &
            h_3 \sigma^{-1}(\ell_2) \sigma^{-1}(k_2) \otimes
            (S^{-2i-1}(k_1) S^{-2i-1}(\ell_1) S^{-1}(h_2) h_1)
            \rightharpoonup d \\
      & = &
            h\varphi(d\otimes \ell) k
    \end{array}
    \]
    and
    \[
    \begin{array}{rcl}
      \varphi(a (d\otimes \ell) b)
      & = &
            \varphi(ad(\ell_1\rightharpoonup b) \otimes \ell_2) \\
      & = &
            \sigma^{-1}(\ell_3) \otimes S^{-2i-1}(\ell_2)
            \rightharpoonup (a d (\ell_1\rightharpoonup b))
      \\
      & = &
            \sigma^{-1}(\ell_5) \otimes \\
      &  &
            (S^{-2i-1}(\ell_4)\rightharpoonup a)
            (S^{-2i-1}(\ell_3) \rightharpoonup d)
            ((S^{-1}(\ell_2)  \ell_1)\rightharpoonup b)
      \\
      & = &
            a (\sigma^{-1}(\ell_2) \otimes
            S^{-2i-1}(\ell_1)\rightharpoonup d) b\,.
    \end{array}
    \]
  \end{proof}

\subsection{Morphisms defined on equivariant dg bimodules} 
\label{sec:AdjunctionDelta}

The following result is used in the description of
${\mathrm{RHom}}_{\Lambda^e}(\Lambda,\Lambda^e)$ in
Section~\ref{sec:proof-main-result} and in the description of
(deformed) Calabi-Yau completions of $A\sharp H$ in
Section~\ref{sec:appl-constr-calabi}.  Note that,
\begin{itemize}
\item when $i=1$, part (1) is being considered
  in \cite[(1.3)]{MR2905560},
\item when $i=0$, the definition of $h\rightharpoonup f$ in part (2)
  coincides with the one of $S^{-2}(h)\rightharpoonup f$ in
  \cite[(1.2)]{MR2905560} and with the one of $fS^{-1}(h)$ in
  \cite[Definition 4.1]{MR3341818},
\end{itemize}
\begin{lem} Let $i\in \mathbb Z$ and assume that $S$ is invertible.
  Let $X\in \mathcal C(\Delta_i)$, $U\in \mathcal C(\Delta_{1-i})$,
  and $M\in \mathcal C(\Lambda^e)$. 
  \begin{enumerate}
  \item There exists a structure of left dg $H$-module on
    $U\underset{A^e}{\otimes} M$ such that 
    \[
h\rightharpoonup (u\otimes m) = (h_2\rightharpoonup u)\otimes
S^{2-2i}(h_3)mS(h_1)\,.
    \]
  \item There is a structure of left dg
    $H$-module on ${\mathrm{Hom}}_{A^e}(X,M)$ such that
    \begin{equation}
      \label{eq:30}
          (h\rightharpoonup f)(x)=S^{2-2i}(h_3)
    f(S^{1-2i}(h_2)\rightharpoonup x)S(h_1)\,.
    \end{equation}
  \item  Given
    $f\in {\mathrm{Hom}}_{A^e}(X,M)$, the following are equivalent
    \begin{enumerate}[(i)]
    \item $f\in {\mathrm{Hom}}_{\Delta_i}(X,M)$,
    \item $(\forall h\in H)\ \ h\rightharpoonup f= \epsilon(h) f$,
    \item $(\forall h\in H)\ (\forall x\in X)\ \ f(h\rightharpoonup x) =
      h_1f(x)S^{1+2i}(h_2)$.
    \end{enumerate}
  \item There is a functorial isomorphism
    \[
    \begin{array}{rcl}
      {\mathrm{Hom}}_{H}(\,_H\k,{\mathrm{Hom}}_{A^e}(X,M))
      & \xrightarrow{\simeq} &
                               {\mathrm{Hom}}_{\Delta_i}(X,M) \\
      \lambda & \mapsto & \lambda(1)\,.
    \end{array}
    \]
  \item Assume, here, that $M$ has an additional structure of right dg
    $A^e$-module which is compatible with the structure of left dg
    $\Lambda^e$-module in the sense of (\ref{eq:compat}) (\emph{e.g.}
    $M=A^e$). Consider ${\mathrm{Hom}}_{A^e}(X,M)$ as dg $A$-bimodule using this
    additional action of $A^e$ on $M$. Then (2) makes of
    ${\mathrm{Hom}}_{A^e}(X,M)$ an $H_{S^{2-2i}}$-equivariant dg $A$-bimodule:
    ${\mathrm{Hom}}_{A^e}(X,M)\in \mathcal C(\Delta_{1-i})$.
\item The canonical mapping ${\mathrm{Hom}}_{A^e}(X,A^e)\underset{A^e}{\otimes} M\to {\mathrm{Hom}}_{A^e}(X,M)$ is $H$-linear.
\end{enumerate}
\end{lem}
\begin{proof}
(1) Here, $U$ is considered as a right dg $A^e$-module in a natural
way: $u\leftharpoonup (a\otimes b) = (-1)^{{\mathrm{deg}}(b) ({\mathrm{deg}}(au))} bua$. The identity given in the statement  endows
$U\otimes M$ with a structure of left dg
$H$-module. Using that $U$ is $H_{S^{2-2i}}$-equivariant,
a simple computation shows that this structure factors through 
$U\otimes M\to U\underset{A^e}{\otimes}U$.


(2) The given action makes of ${\mathrm{Hom}}_\k(X,M)$ a left dg $H$-module. The
following computation where $f\in {\mathrm{Hom}}_{A^e}(X,M)$ proves that
${\mathrm{Hom}}_{A^e}(X,M)$ is a dg $H$-submodule of ${\mathrm{Hom}}_\k(X,M)$
\[
\begin{array}{rcl}
  (h\rightharpoonup f) (axb)
  & = &
        S^{2-2i}(h_3) f(S^{1-2i}(h_2)\rightharpoonup (axb)) S(h_1) \\
  & = &
        \pm S^{2-2i}(h_5) S^{1-2i}(h_4)\rightharpoonup a
        f(S^{1-2i}(h_3)\rightharpoonup x)
        S(h_2)\rightharpoonup b S(h_1) \\
  & = &
        \pm a(h\rightharpoonup f)(x) b\,,
\end{array}
\]
where $\pm$ is the sign $(-1)^{{\mathrm{deg}}(f)\cdot {\mathrm{deg}}(a)}$.


(3) The equivalence ``$(i)\Leftrightarrow (iii)$'' follows from the
definition of $\Delta_i$ (see \eqref{eq:31}). The implication ``$(iii)\Rightarrow (ii)$''
follows from the definition of the action of $H$ on
${\mathrm{Hom}}_{A^e}(X,M)$. If $(ii)$ holds true, then
\[
\begin{array}{rcl}
  f(h\rightharpoonup x)
  & = &
        \epsilon(h_1) f(h_2\rightharpoonup x) \epsilon(h_3) \\
  & = &
        h_1 S(h_2) f(h_3\rightharpoonup x) S^{2i}(h_4) S^{1+2i}(h_5) \\
  & = &
        h_1 (S^{2i-1}(h_2) \rightharpoonup f)(x) S^{1+2i}(h_3) \\
  & = &
        h_1 f(x) S^{1+2i}(h_2)\,.
\end{array}
\]
This proves that $(ii)\Rightarrow (iii)$.


(4) In view of (3), the given mapping is well-defined and
surjective. The mapping is injective, and hence it 
is an isomorphism.


(5) This follows from the computation below where $f\in {\mathrm{Hom}}_{A^e}(X,M)$.
To avoid any confusion, note that the computation does not involve the action of $A^e$ on $M$
arising from the action of $\Lambda^e$, and that the third equality is due to \eqref{eq:compat}.
\[
\begin{array}{ll}
  (h\rightharpoonup (f\leftharpoonup (a\otimes b)))(x) &= \\

  (S^{2-2i}(h_3)\otimes S(h_1))\rightharpoonup
        ((f\leftharpoonup (a\otimes b))(S^{1-2i}(h_2)\rightharpoonup x))&= \\
  
  \pm (S^{2-2i}(h_3)\otimes S(h_1))\rightharpoonup
  (f(S^{1-2i}(h_2)\rightharpoonup x)\leftharpoonup (a\otimes b))
  &=\\
  
  \pm ((S^{2-2i}(h_4)\otimes S(h_2)) \rightharpoonup (f(S^{1-2i}(h_3)
  \rightharpoonup x)))
  \leftharpoonup \\
  \ \ \ \ ((S^{2-2i}(h_5) \otimes S(h_1))\rightharpoonup (a\otimes b)) &=
  \\
  
  \pm((S^{2-2i}(h_4)\otimes S(h_2))\rightharpoonup
  f(S^{1-2i}(h_3)\rightharpoonup x))  \\
  
  \ \ \ \ \leftharpoonup (S^{2-2i}(h_5)\rightharpoonup a\otimes
  h_1\rightharpoonup b) &= \\
  
  ((h_2\rightharpoonup f) \leftharpoonup (S^{2-2i}(h_3)\rightharpoonup
  a \otimes h_1\rightharpoonup b))(x)\,,
\end{array}
\]
where $\pm$ is the sign $(-1)^{\deg(x)\cdot (\deg(a)+\deg(b))}$. In
other words, using bimodule notation, then
\[
  h\rightharpoonup (afb) = (h_1\rightharpoonup a ) (h_2\rightharpoonup
  f) (S^{2-2i}(h_3) \rightharpoonup b)\,.
  \]


(6) Following (5), ${\mathrm{Hom}}_{A^e}(X,A^e)$ is a left dg
$\Delta_{1-i}$-module
for which the corresponding structure of right dg $A^e$-module is given by
$(f\leftharpoonup (a\otimes b))(x) = f(x)\times (a\otimes b)$. Hence,
(1) provides ${\mathrm{Hom}}_{A^e}(X,A^e)\underset{A^e}{\otimes } M$ with a structure of  left dg $H$-module.

Let $\varphi\in {\mathrm{Hom}}_{A^e}(X,A^e)$, $m\in M$ and $h\in H$;
Denote by $f$ the image of $\varphi \otimes m$ in
${\mathrm{Hom}}_{A^e}(X,M)$; It is defined by
$f(x)= (-1)^{\deg(m)\cdot \deg(x)}\varphi(x)\rightharpoonup m$ (where
$\rightharpoonup$ is used to denote the action of $\Lambda^e$ on $M$);
Denote by $g$ the image of $h\rightharpoonup (\varphi \otimes m)$ in
${\mathrm{Hom}}_{A^e}(X,M)$; It is defined by
$g(x) = (-1)^{\deg(m)\cdot \deg(x)} (h_2\rightharpoonup \varphi)(x)
\rightharpoonup (S^{2-2i}(h_3) m S(h_1))$. Then
$g= h\rightharpoonup f$ according to the following computation where
$M$ is considered alternatively as a left dg $\Lambda^e$-module and as
a dg $\Lambda$-bimodule
\[
\begin{array}{rcl}
  (h\rightharpoonup f) (x)
  & = &
        (S^{2-2i}(h_3)\otimes S(h_1)) \rightharpoonup
        f(S^{1-2i}(h_2)\rightharpoonup x) \\
  & = &
        \pm (S^{2-2i}(h_3) \otimes S(h_1))\rightharpoonup
        (\varphi(S^{1-2i}(h_2)\rightharpoonup x) \rightharpoonup m) \\
  & \underset{\eqref{eq:15}}{=} &
                                  \pm ((S^{2-2i}(h_4)\otimes S(h_2))
                                  \rightharpoonup
                                  \varphi(S^{1-2i}(h_3)\rightharpoonup x)) \\
  &   &
                                  \rightharpoonup
                                  ((S^{2-2i}(h_5)\otimes
                                  S(h_1))\rightharpoonup m) \\
  & = &
        \pm (h_2\rightharpoonup \varphi)(x) \rightharpoonup
        (S^{2-2i}(h_3)mS(h_1)) \\
  & = &
        g(x)\,.        
\end{array}
\]
\end{proof}

\subsection{Equivariant actions on duals over $A$}
\label{sec:equiv-acti-duals}

In this section, $S$ is assumed to be invertible and $A$ is assumed to be a $\k$-algebra.
Let $i\in\mathbb Z$, let $D$ be an $H_{S^{2i}}$-equivariant  $A$-bimodule.
Let $\sigma\in {\mathrm{Aut}}_{\k-{\mathrm{alg}}}(H)$ be such that the identity
$\sigma(h)_1\otimes \sigma(h)_2 = S^{2i}(h_1) \otimes \sigma(h_2)$
holds in $H$. 
When $D$ is invertible as an $A$-bimodule, then ${\mathrm{Hom}}_A(D,A)\simeq {\mathrm{Hom}}_{A^{\mathrm{op}}}(D,A)$ is the inverse of
$D$. 
Hence,  
this section considers $\Lambda$-bimodules arising from
${\mathrm{Hom}}_A(D,A)$ and ${\mathrm{Hom}}_{A^{\mathrm{op}}}(D,A)$.

\subsubsection{}
\label{sec:equiv-acti-duals-1}
The following result shows that ${\mathrm{Hom}}_A(D,A)$ and ${\mathrm{Hom}}_{A^{\mathrm{op}}}(D,A)$ both have a structure of
$H_{S^{-2i}}$-equivariant $A$-bimodules.
\begin{lem}
Keep the setting stated previously.
  \begin{enumerate}
  \item There exists a structure of left $H$-module on
    ${\mathrm{Hom}}_A(D,A)$ such that $(h\rightharpoonup f)(d) =
    S^{-2i}(h_2) \rightharpoonup f(S^{-1-2i}(h_1)\rightharpoonup
    d)$ for every $f\in {\mathrm{Hom}}_A(D,A)$. For this structure, ${\mathrm{Hom}}_A(D,A)$ is an
    $H_{S^{-2i}}$-equivariant  $A$-bimodule.
  \item There exists a structure of left  $H$-module on
    ${\mathrm{Hom}}_{A^{\mathrm{op}}}(D,A)$ such that $(h\rightharpoonup f)(d) =
    h_1 \rightharpoonup f(S^{1-2i}(h_2)\rightharpoonup
    d)$ for every $f\in {\mathrm{Hom}}_{A^{\mathrm{op}}}(D,A)$. For this structure,
    ${\mathrm{Hom}}_{A^{\mathrm{op}}}(D,A)$ is an 
    $H_{S^{-2i}}$-equivariant  $A$-bimodule.
  \end{enumerate}
\end{lem}
\begin{proof}
  (1) If $f\in {\mathrm{Hom}}_A(D,A)$, then the mapping $h\rightharpoonup f\colon
  D\to A$ introduced in the statement is $A$-linear, indeed
  \[
  \begin{array}{rcl}
    (h\rightharpoonup f)(ad)
    & = &
          S^{-2i}(h_2)\rightharpoonup f(S^{-1-2i}(h_1)\rightharpoonup
          (ad)) \\
    & = &
          S^{-2i}(h_3) \rightharpoonup \left(
          (S^{-1-2i}(h_2)\rightharpoonup a)
          f(S^{-1-2i}(h_1)\rightharpoonup d)
          \right) \\
    & = &
          a(h\rightharpoonup f)(d)\,.          
  \end{array}
  \]
  Note that the $A$-bimodule structure of ${\mathrm{Hom}}_A(D,A)$ is such that
  $(afb)(d) = f(da)b$.
The following computation shows the second assertion of (1)
\[
\begin{array}{rcl}
  (h\rightharpoonup (afb))(d)
  & = &
        S^{-2i}(h_2)\rightharpoonup
        (f((S^{-1-2i}(h_1)\rightharpoonup d)a)b) \\
  & \underset{~\eqref{eq:32}}{=} &
                                   S^{-2i}(h_3)\rightharpoonup \left(
                                   f(S^{-1-2i}(h_2)\rightharpoonup
                                   (d(h_1\rightharpoonup a))) b
                                   \right) \\
  & = &
        (h_2\rightharpoonup f)(d(h_1\rightharpoonup a))
        (S^{-2i}(h_3)\rightharpoonup b) \\
  & = &
        \left(
        (h_1\rightharpoonup a) (h_2\rightharpoonup f)
        (S^{-2i}(h_3)\rightharpoonup b)
        \right) (d)\,.        
\end{array}
\]
(2) is proved similarly.
\end{proof}

\subsubsection{}
\label{sec:equiv-acti-duals-2}
The following technical result is used later to prove that $D$ is
invertible as an $A$-bimodule if and only if $D\sharp\,^\sigma H$ is
invertible as a $\Lambda$-bimodule.
\begin{lem}
  Keep the setting stated previously. Let $M\in {\mathrm{mod}}(\Lambda)$. Let $N\in {\mathrm{mod}}(\Lambda^{\mathrm{op}})$.
  \begin{enumerate}
  \item The following endows ${\mathrm{Hom}}_A(D,M)$ with a structure of left
    $\Lambda$-module (here $g\in {\mathrm{Hom}}_A(D,M)$)
    \[
    \begin{array}{crcl}
      ag \colon & d &\mapsto & g(da) \\
      hg \colon & d & \mapsto & \sigma^{-1}(h_2)
                                g(S^{-2i-1}(h_1)\rightharpoonup d)\,.
    \end{array}
    \]
    For this structure, the following mapping is an
    isomorphism in ${\mathrm{mod}}(\Lambda)$
    \[
    \begin{array}{rcl}
      {\mathrm{Hom}}_\Lambda(D\sharp\,^\sigma H, M) & \to & {\mathrm{Hom}}_A(D,M) \\
      f & \mapsto & f(-\otimes 1)\,.
    \end{array}
    \]
  \item Assume that $M=\Lambda$. Then, the canonical mapping
    from ${\mathrm{Hom}}_A(D,A)\otimes H$ to ${\mathrm{Hom}}_A(D,A\otimes H)$ is a morphism in ${\mathrm{mod}}(\Lambda^e)$ from ${\mathrm{Hom}}_A(D,A)\sharp \,^{\sigma^{-1}} H$ to
    ${\mathrm{Hom}}_A(D,\Lambda)$. In particular, if $D$ is finitely presented
    as a left $A$-module or if ${\mathrm{dim}}_{\k}\,H<\infty$, then
    ${\mathrm{Hom}}_\Lambda(D\sharp\,^\sigma H,\Lambda) \simeq
    {\mathrm{Hom}}_A(D,A)\sharp\,^{\sigma^{-1}} H$ in ${\mathrm{mod}}(\Lambda^e)$.
  \item The following actions endow ${\mathrm{Hom}}_{A^{\mathrm{op}}}(D,N)$ with a
    structure of right $\Lambda$-module (here $g\in {\mathrm{Hom}}_{A^{\mathrm{op}}}(D,N)$)
    \[
    \begin{array}{crcl}
      ga \colon & d & \mapsto & g(ad) \\
      gh \colon & d & \mapsto & g(h_1\rightharpoonup d) \sigma(h_2)\,.
    \end{array}
    \]
    For this structure, the following mapping is an isomorphism in
    ${\mathrm{mod}}(\Lambda^{\mathrm{op}})$
    \[
    \begin{array}{rcl}
      {\mathrm{Hom}}_{\Lambda^{\mathrm{op}}}(D\sharp\,^\sigma H, N) & \to &
                                                             {\mathrm{Hom}}_{A^{\mathrm{op}}}(D,N)
      \\
      f & \mapsto & f(- \otimes 1) \,.
    \end{array}
    \]
  \item Assume that $N=\Lambda$. Then, there is a morphism in ${\mathrm{mod}}(\Lambda^e)$
    \[
    {\mathrm{Hom}}_{A^{\mathrm{op}}}(D,A) \sharp\,^{\sigma^{-1}}H \to {\mathrm{Hom}}_{A^{\mathrm{op}}}(D,\Lambda)
    \]
    such that the image of a tensor $\varphi \otimes \ell$ is the
    morphism $D\to \Lambda$ given by $d\mapsto
    \varphi(\ell_1\rightharpoonup d) \sigma(\ell_2)$. If $D$ is
    finitely presented as a right $A$-module or if ${\mathrm{dim}}_{\k}\,H<\infty$, then it is an isomorphism.
  \end{enumerate}
\end{lem}
\begin{proof}
  (1) Given $f\in {\mathrm{Hom}}_\Lambda(D\sharp\,^\sigma H, M)$, consider
  $g\colon D\to M$ defined by $g(d) = f(d\otimes 1)$. Then,
  \begin{itemize}
  \item $g\in {\mathrm{Hom}}_A(D,M)$,
  \item $f$ is given by $f(d\otimes \ell) = \sigma^{-1}(\ell_2) g(
    S^{-1-2i}(\ell_1) \rightharpoonup d)$ (see \eqref{eq:44}).
  \end{itemize}
  Under the above construction, given $a\in A$ and $h\in H$, the
  morphisms $af,hf\in {\mathrm{Hom}}_\Lambda(D\sharp\,^\sigma H,M)$ are mapped
  onto $ag$ and $hg$, respectively.

  Conversely, let $g\in {\mathrm{Hom}}_A(D,M)$ and define $f\colon D\otimes H
  \to M$ by $f(d\otimes \ell) = \sigma^{-1}(\ell_2)
  g(S^{-1-2i}(\ell_1)\rightharpoonup d)$. The following computations
  made in $M$ show that $f\in {\mathrm{Hom}}_\Lambda(D\sharp\,^\sigma H,M)$
  \[
  \begin{array}{rcl}
    f( a ( d\otimes \ell))
    & = &
          f(ad\otimes \ell) \\
    & = &
          \sigma^{-1}(\ell_2) g(S^{-1-2i}(\ell_1) \rightharpoonup (ad)
          ) \\
    & = &
          \sigma^{-1}(\ell_3) (S^{-1-2i}(\ell_2)\rightharpoonup a)
          g(S^{-1-2i}(\ell_1)\rightharpoonup d) \\
    & = &
          \sigma^{-1}(\ell_2)_2 (S^{-1}(\sigma^{-1}(\ell_2)_1)\rightharpoonup a)
          g(S^{-1-2i}(\ell_1)\rightharpoonup d) \\
    & = &
          a\sigma^{-1}(\ell_2) g( S^{-1-2i}(\ell_1)\rightharpoonup d)
    \\
    & = &
          af(d\otimes \ell)
  \end{array}
  \]
  and
  \[
  \begin{array}{rcl}
    f(h(d\otimes \ell))
    & = &
          f(h_1\rightharpoonup d \otimes \sigma(h_2) \ell) \\
    & = &
          \sigma^{-1}(\sigma(h_3)\ell_2)
          g(S^{-1-2i}(S^{2i}(h_2)\ell_1)\rightharpoonup
          (h_1\rightharpoonup d)) \\
    & = &
          h \sigma^{-1}(\ell_2) g( S^{-1-2i}(\ell_1)\rightharpoonup d)
    \\
    & = &
          h f(d\otimes \ell)\,.
  \end{array}
  \]
  The previous considerations prove that the mapping from ${\mathrm{Hom}}_\Lambda
  (D\sharp\,^\sigma H , M)$ to ${\mathrm{Hom}}_A(D,M)$ given in the statement of
  the lemma is well-defined and bijective. They also prove that the
  actions of $A$ and $H$ on ${\mathrm{Hom}}_A(D,M)$ given in the statement form
  a structure of $\Lambda$-module such that the mapping mentioned
  previously if $\Lambda$-linear.


  (2) It suffices to prove the first statement. The structure of left
  $H$-module on ${\mathrm{Hom}}_A(D,M)$ is functorial in $M$, and hence,
  ${\mathrm{Hom}}_A(D,\Lambda)$ is a $\Lambda$-bimodule with structure of right
  $\Lambda$-module inherited from the one of $\Lambda$. The given canonical
  mapping is a morphism in ${\mathrm{mod}}(A)$ and ${\mathrm{mod}}(H^{\mathrm{op}})$. There remains to prove that it is so in ${\mathrm{mod}}(A^{\mathrm{op}})$ and in ${\mathrm{mod}}(H)$.
  Let $\varphi \otimes \ell \in {\mathrm{Hom}}_A(D,A) \otimes H$, and denote by
  $f\colon d \mapsto \varphi(d) \ell$ its image in
  ${\mathrm{Hom}}_A(D,\Lambda)$.

  Let $a\in A$, and denote by $f'$ the image of
  $(\varphi \otimes \ell) a = (\varphi (\ell_1\rightharpoonup a)
  \otimes \ell_2)$ in ${\mathrm{Hom}}_A(D,\Lambda)$. The computation below made
  in $\Lambda$ proves that $f'=fa$
  \[
  \begin{array}{rcl}
    (fa)(d)
    & = &
          f(d) a \\
    & = &
          \varphi(d) \ell a \\
    & = &
          \varphi(d) (\ell_1\rightharpoonup a) \ell_2 \\
    & = &
          f'(d)\,.
  \end{array}
  \]
  Let $h\in H$, and denote by $f''$ the image of $h(\varphi \otimes
  \ell) = h_1\rightharpoonup \varphi \otimes \sigma^{-1}(h_2) \ell$ in
  ${\mathrm{Hom}}_\Lambda(D,\Lambda)$. The computation below made in $\Lambda$
  proves that $f''=hf$
  \[
  \begin{array}{rcl}
    (hf)(d)
    & \underset{\text{see (1)}}= &
                                    \sigma^{-1}(h_2) f(S^{-1-2i}(h_1)
                                    \rightharpoonup d ) \\
    & = &
          \sigma^{-1}(h_2) \varphi(S^{-1-2i}(h_1)\rightharpoonup d)
          \ell \\
    & = &
          S^{-2i}(h_2) \rightharpoonup
          \varphi(S^{-1-2i}(h_1)\rightharpoonup d) \sigma^{-1}(h_3)
          \ell \\
    & = &
          (h_1\rightharpoonup \varphi)(d) \sigma^{-1}(h_2)\ell \\
    & = &
          f''(d)\,.          
  \end{array}
  \]
  This proves (2).


  (3) Let $f\in {\mathrm{Hom}}_{\Lambda^{\mathrm{op}}}(\Lambda,N)$ and define $g\in
  {\mathrm{Hom}}_\k(D,N)$ by $g(d)=f(d\otimes 1)$. Then, $g\in {\mathrm{Hom}}_{A^{\mathrm{op}}}(D,N)$. This construction defines a mapping from ${\mathrm{Hom}}_{\Lambda^{\mathrm{op}}}(D\sharp \,^\sigma H,N)$ to ${\mathrm{Hom}}_{A^{\mathrm{op}}}(D,N)$.
  Keep $f$ and $g$ as above. Given $a\in A$, then $ga$ corresponds to
  $fa$ under the same construction. Also, given $h\in H$, then $gh$
  corresponds to $fh$ under this construction as proved by the
  following computation
  \[
  \begin{array}{rcl}
    (fh) (d\otimes 1)
    & = &
          f(h(d\otimes 1)) \\
    & = &
          f(h_1\rightharpoonup d \otimes \sigma(h_2)) \\
    & = &
          g(h_1\rightharpoonup d) \sigma(h_2)\,.
  \end{array}
  \]
  Now, let $g\in {\mathrm{Hom}}_{A^{\mathrm{op}}}(D,N)$ and define $f\colon D\sharp
  \,^\sigma H \to N$ by $f(d\otimes \ell) =g(d) \ell$. Then, $f$
  lies in ${\mathrm{Hom}}_{H^{\mathrm{op}}}(D\sharp\, ^\sigma H,N)$. And so does it
  in ${\mathrm{Hom}}_{A^{\mathrm{op}}}(D\sharp\,^\sigma H, N)$ as proved by the
  following computation
  \[
  \begin{array}{rcl}
    f((d\otimes \ell) a)
    & = &
          f( d (\ell_1\rightharpoonup a) \otimes \ell_2) \\
    & = &
          g(d (\ell_1\rightharpoonup a)) \ell_2 \\
    & = &
          g(d) (\ell_1\rightharpoonup a) \ell_2 \\
    & = &
          g(d) \ell a \\
    & = &
          f(d\otimes \ell) a\,.
  \end{array}
  \]
  These considerations prove that there is a well-defined mapping from
  ${\mathrm{Hom}}_{A^{\mathrm{op}}}(D,N)$ to
  ${\mathrm{Hom}}_{\Lambda^{\mathrm{op}}}(D\sharp\,^\sigma H,N)$ which is inverse to
  the mapping from ${\mathrm{Hom}}_{\Lambda^{\mathrm{op}}}(D\sharp\,^\sigma H,N)$ to
  ${\mathrm{Hom}}_{A^{\mathrm{op}}}(D,N)$ introduced in the statement of (3). This
  proves (3).


  (4) Like in (2), ${\mathrm{Hom}}_{A^{\mathrm{op}}}(D,\Lambda)$ is a
  $\Lambda$-bimodule with structure of left $\Lambda$-module inherited
  from the one of $\Lambda$.

  First, the given mapping from ${\mathrm{Hom}}_{A^{\mathrm{op}}}(D,A)\sharp\,^{\sigma^{-1}}H$ to ${\mathrm{Hom}}_{A^{\mathrm{op}}}(D,\Lambda)$ is well-defined. Indeed, let $\varphi \otimes
  \ell\in {\mathrm{Hom}}_{A^{\mathrm{op}}}(D,A)\sharp\,^{\sigma^{-1}}H$. Define a mapping $g\colon D\to
  \Lambda$ by $g(d) =\varphi(\ell_1\rightharpoonup
  d) \sigma(\ell_2)$. In view of the following equalities in
  $\Lambda$, this mapping is a morphism in ${\mathrm{mod}}(A^{\mathrm{op}})$
  \[
  \begin{array}{rcl}
    g(da)
    & = &
          \varphi(\ell_1\rightharpoonup (da)) \sigma(\ell_2) \\
    & = &
          \varphi(\ell_1\rightharpoonup d)
          (S^{2i}(\ell_2)\rightharpoonup a) \sigma(\ell_3) \\
    & = &
          \varphi(\ell_1\rightharpoonup d) \sigma(\ell_2) a\,.
  \end{array}
  \]
  Denote by $\theta$ the resulting mapping from ${\mathrm{Hom}}_{A^{\mathrm{op}}}(D,A)\sharp\,^{\sigma^{-1}}H$ to ${\mathrm{Hom}}_{A^{\mathrm{op}}}(D,\Lambda)$.

  Next, $\theta$ is a morphism in ${\mathrm{mod}}(\Lambda^e)$. Indeed, keep
  the notation ($\varphi,\ell,g$) introduced previously, then
  \[
  \begin{array}{rcl}
    (ag)(d)
    & = &
          ag(d) \\
    & = &
          (a\varphi)(\ell_1\rightharpoonup d) \sigma(\ell_2) \\
    \\
    (hg)(d)
    & = &
          hg(d) \\
    & = &
          h \varphi(\ell_1\rightharpoonup d) \sigma(\ell_2) \\
    & = &
          h_1\rightharpoonup \varphi(\ell_1\rightharpoonup d) h_2
          \sigma(\ell_2) \\
    & = &
          h_1\rightharpoonup \varphi((S^{1-2i}(h_2) S^{-2i}(h_3)
          \ell_1)
          \rightharpoonup d) h_4 \sigma(\ell_2) \\
    & = &
          (h_1\rightharpoonup \varphi) (
          (\sigma^{-1}(h_2)\ell)_1 \rightharpoonup d)
          \sigma((\sigma^{-1}(h_2)\ell)_2) \\
    \\
    (ga)(d)
    & = &
          g(ad) \\
    & = &
          \varphi(\ell_1\rightharpoonup (ad) ) \sigma(\ell_2) \\
    & = &
          (\varphi \cdot (\ell_1\rightharpoonup a))
          (\ell_2\rightharpoonup d) \sigma(\ell_3)  \\
    \\
    (gh)(d)
    & = &
          g(h_1\rightharpoonup d) \sigma(h_2) \\
    & = &
          \varphi((\ell_1h_1)\rightharpoonup d) \sigma(\ell_2 h_2)\,.
  \end{array}
  \]
  In order to prove (4), there only remains to prove that $\theta$ is
  bijective when $D$ is finitely presented as a right $A$-module or
  ${\mathrm{dim}}_\k\,H<\infty$. For this purpose, consider the composite
  mapping
  \[
  H^\sigma\sharp {\mathrm{Hom}}_{A^{\mathrm{op}}}(D,A) \underset{~\ref{sec:45}}
  {\xrightarrow{\sim}}
  {\mathrm{Hom}}_{A^{\mathrm{op}}}(D,A)\sharp \,^{\sigma^{-1}}H
  \xrightarrow{\theta}
  {\mathrm{Hom}}_{A^{\mathrm{op}}}(D,\Lambda)\,.
  \]
  Recall from the proof in \ref{sec:45} that the left hand-side
  mapping is defined by $\ell \otimes \varphi \mapsto
  \ell_1\rightharpoonup \varphi \otimes \sigma^{-1}(\ell_2)$. Hence,
  to any $\ell\otimes \varphi \in H\otimes {\mathrm{Hom}}_{A^{\mathrm{op}}}(D,A)$,
  the composite mapping 
  associates $\theta(\ell_1\rightharpoonup \varphi \otimes
  \sigma^{-1}(\ell_2))$. Denote this morphism by $g$. Then
  \[
  \begin{array}{rcl}
    g(d)
    & = &
          (\ell_1\rightharpoonup \varphi)
          (\sigma^{-1}(\ell_2)_1\rightharpoonup d)
          \sigma(\sigma^{-1}(\ell_2)_2) \\
    & = &
          (\ell_1\rightharpoonup \varphi)
          (S^{-2i}(\ell_2)\rightharpoonup d)\ell_3 \\
    & = &
          \ell_1\rightharpoonup \varphi((
          S^{1-2i}(\ell_2) S^{-2i}(\ell_3))\rightharpoonup d) \ell_4 \\
    & = &
          \ell \varphi (d)\,.
  \end{array}
  \]
  Hence, the above composite mapping is the canonical one from
  $H\otimes {\mathrm{Hom}}_{A^{\mathrm{op}}}(D,A)$ to ${\mathrm{Hom}}_{A^{\mathrm{op}}}(D,H\otimes
  A)$. It is hence bijective when $D$ is finitely presented in ${\mathrm{mod}}(A^{\mathrm{op}})$ or ${\mathrm{dim}}_\k\,H<\infty$. Accordingly,
  $\theta$ is bijective under the same assumption.
\end{proof}

\subsection{Tensor product of equivariant bimodules}
\label{sec:tens-prod-equiv}

Here, $A$ is a dg algebra.
\subsubsection{}
\label{sec:26}

The following result describes $(D\sharp\,^\sigma H)
\underset{\Lambda}{\otimes} (D'\sharp\,^\tau H)$ for equivariant 
dg $A$-bimodules $D$ and $D'$. Here is its setting
\begin{itemize}
\item $i,j\in \mathbb Z$ and $S$ is invertible as soon as $i<0$ or $j<0$,
\item $D$ is an $H_{S^{2i}}$-equivariant  dg $A$-bimodule,
\item $D'$ is an $H_{S^{2j}}$-equivariant  dg $A$-bimodule,
\item $\sigma,\tau$ are automorphisms of $H$ 
  which satisfy the identities
  $\sigma(h)_1\otimes \sigma(h)_2 = S^{2i}(h_1) \otimes \sigma(h_2)$
  and $\tau(h)_1\otimes \tau(h)_2 = S^{2j}(h_1)\otimes \tau(h_2)$
  in $H$, and such that   $\sigma$ and $S^2$ commute, and such that $\tau$
  and $S^2$ commute (for instance, 
  $\sigma=S^{+2i}\circ \Xi_\pi^r$ and $\mu = S^{+2j}\circ \Xi_{\pi'}^r$ for
algebra homomorphisms $\pi,\pi'\colon H\to \k$).
\end{itemize}
\begin{lem}
  Keep the setting stated previously.
  \begin{enumerate}
  \item There is a structure of $H_{S^{2(i+j)}}$-equivariant 
   dg $A$-bimodule on $D\underset{A}{\otimes} D'$ such that
   $h\rightharpoonup (d\otimes d')=h_1\rightharpoonup d\otimes
   S^{2i}(h_2)\rightharpoonup d'$,
 \item $(D\sharp\,^\sigma H)\underset{\Lambda}{\otimes} (D'\sharp\,
   ^\tau H) \simeq (D\underset{A}{\otimes} D')\sharp \,^{\tau\circ
     \sigma}H$ as  dg $\Lambda$-bimodules.
  \end{enumerate}
\end{lem}
\begin{proof}
  (1) The given action is indeed a structure of left  $H$-module on
  $D\underset{A}{\otimes} D'$ because $D$ is
  $H_{S^2i}$-equivariant. For this structure, $D\underset{A}{\otimes}D'$
  is $H_{S^{2(i+j)}}$-equivariant because $D$ is $H_{S^{2i}}$-equivariant
  and $D'$ is $H_{S^{2j}}$-equivariant.


  (2) Note the following identity in $H$
  \[
  \begin{array}{rcl}
  (\tau \circ \sigma)(h)_1\otimes (\tau \circ \sigma)(h)_2
    & = &
          S^{2j}(\sigma(h)_1) \otimes \tau(\sigma(h)_2) \\
    & = &
          S^{2i+2j}(h_1) \otimes \tau\circ \sigma(h_2)\,.
  \end{array}
  \]
  Hence, $(D\underset{A}{\otimes} D')\sharp \,^{\tau\circ \sigma}H$ is
  well-defined.

  Note the following identity in $(D\sharp\,^\sigma
  H)\underset\Lambda\otimes (D'\sharp \,^\tau H)$:
  \begin{equation}
    \label{eq:36}
    (d\otimes \ell) \underset{\Lambda}{\otimes} (d'\otimes \ell')
    = 
          (d\otimes 1) \underset{\Lambda}{\otimes}
          (\ell_1\rightharpoonup d' \otimes \tau(\ell_2)\ell')\,.
        \end{equation}
Therefore, there is a well-defined linear mapping
\[
\begin{array}{rcl}
  (D\sharp \,^\sigma H) \underset{\Lambda}{\otimes} (D'\sharp
    \,^\tau H)
  & \to
  & (D\underset{A}{\otimes} D')\otimes H \\
   (d\otimes \ell)\underset{\Lambda}{\otimes} (d'\otimes \ell')
  & \mapsto
  & (d\underset{A}{\otimes} \ell_1\rightharpoonup d') \otimes
    \tau(\ell_2) \ell'\,.
\end{array}
\]
There is also a well-defined linear mapping
\[
\begin{array}{rcl}
  (D\underset{A}{\otimes} D')\otimes H
  & \to
  & (D\sharp \,^\sigma H) \underset{\Lambda}{\otimes} (D'\sharp
    \,^\tau H) \\
  (d\underset{A}{\otimes} d') \otimes \ell'
  & \mapsto
  & (d\otimes 1) \underset{\Lambda}{\otimes} (d'\otimes\ell')
\end{array}
\]
In view of \eqref{eq:36}, these two linear mappings are inverse to each
other. In order to prove the lemma, it therefore suffices to check that the
latter mapping is a morphism of dg $\Lambda$-bimodules from
$(D\underset A \otimes D')\sharp \,^{\tau\circ \sigma} H$ to $(D\sharp
\,^\sigma H) \underset\Lambda\otimes (D'\sharp\,^{\tau} H)$. This
follows from the computations below made in $(D\sharp
\,^\sigma H) \underset\Lambda\otimes (D'\sharp\,^{\tau} H)$:
\[
\begin{array}{rcl}
  a((d\otimes 1) \underset{\Lambda}{\otimes} (d'\otimes \ell'))b
  & = &
        (a(d\otimes 1)) \underset\Lambda\otimes ((d'\otimes \ell')b)
  \\
  & = &
        (ad\otimes 1) \underset{\Lambda}{\otimes}
        (d'(\ell'_1\rightharpoonup b)\otimes \ell'_2) \\
\end{array}
\]
and
\[
\begin{array}{rcl}
  h((d\otimes 1) \underset{\Lambda}{\otimes} (d'\otimes \ell'))k
  & = &
        (h(d\otimes 1)) \underset\Lambda\otimes ((d'\otimes \ell')k)
  \\
  & = &
        (h_1\rightharpoonup d \otimes \sigma(h_2)) \underset\Lambda\otimes
        (d'\otimes \ell' k) \\
  & \underset{\eqref{eq:36}}= &
        (h_1\rightharpoonup d \otimes 1)
        \underset{\Lambda}{\otimes}
        (\sigma(h_2)_1\rightharpoonup d'\otimes \tau(\sigma(h_2)_2)\ell' k) \\
  & = &
        (h_1\rightharpoonup d \otimes 1) \underset{\Lambda}{\otimes}
        (S^{2i}(h_2)\rightharpoonup d' \otimes \tau(\sigma(h_3))\ell'k)\,.
\end{array}
\]
\end{proof}

\subsubsection{}
\label{sec:27}
In this paragraph, $S$ need not be invertible. 
Let $D$ be an $H_{S^2}$-equivariant dg $A$-bimodule and let $\sigma\in
{\mathrm{Aut}}_{\k-{\mathrm{alg}}}(H)$ satisfy the identity
$\sigma(h)_1\otimes \sigma(h)_2=S^2(h_1)\otimes \sigma(h_2)$ in $H$ and
commute with $S^2$. The
preceding result 
provides a description of $T_{A\sharp H}(D\sharp\,^\sigma H)$. This is
used later when discussing on Calabi-Yau completions.

For every $n\geqslant 1$, consider $D^{\underset{A}{\otimes}n}$ as an
$H_{S^{2n}}$-equivariant dg $A$-bimodule for the following action of $H$
(see \ref{sec:26})
\begin{equation}
  \label{eq:28}
  h\rightharpoonup (d_1\underset{A}{\otimes}\cdots
  \underset{A}{\otimes} d_n) = h_1\rightharpoonup d_1
  \underset{A}{\otimes} S^2(h_2) \rightharpoonup d_2
  \underset{A}{\otimes} \cdots \underset{A}{\otimes}
  S^{2(n-1)}(h_n)\rightharpoonup d_n\,.
\end{equation}
According to \ref{sec:26}, there exists an isomorphism of algebras
\[
T_A(D)\otimes H\xrightarrow \sim T_{A\sharp H}(D\sharp\,^\sigma H) 
\]
which extends the identity maps $A\to A$ and $H\to H$ as well as the
mapping $D\to D\sharp\,^{\sigma} H,\, d\mapsto d\otimes 1$, and where
$T_A(D)\otimes H$  is endowed with the product such that 
$T_A(D)$ and $H$ are subalgebras  in the natural way and
\begin{equation}
  \label{eq:29}
  h\times (d_1\underset{A}{\otimes}\cdots
  \underset{A}{\otimes} d_n) = h_1\rightharpoonup (d_1\underset{A}{\otimes}\cdots
  \underset{A}{\otimes} d_n) \otimes \sigma^n(h_2)\,.
\end{equation}

This dg algebra with underlying complex $T_A(D)\otimes
H$ is denoted by $T_A(D)\sharp\,^{\sigma^*}H$. If
$\sigma={\mathrm{Id}}_H$, then $T_A(D)$ is an $H$-module dg algebra and
$T_A(D)\sharp\,^{\sigma^*}H=T_A(D)\sharp H$.

\subsection{On invertible  $\Lambda$-bimodules}
\label{sec:invertible-dg-lambda}
Assume that $A$ is a $\k$-algebra and that $S$ is invertible.  The
following results relate equivariant $A$-bimodules which are invertible
as $A$-bimodules to invertible $\Lambda$-bimodules.  Recall that an
$A$-bimodule $D$ is called \emph{invertible} if there exists an
$A$-bimodule $D'$ such that
$D\underset{A}\otimes D'\simeq D'\underset{A}\otimes D\simeq A$ in
${\mathrm{mod}}(A^e)$. In view of the adjunctions
\[
\xymatrix{
  \mathrm{mod}(A) \ar@/_/[d]_{D\underset A \otimes -}
  &&
  \mathrm{mod}(A^{\mathrm{op}}) \ar@/_/[d]_{-\underset A \otimes D}
  \\
  \mathrm{mod}(A) \ar@/_/[u]_{{\mathrm{Hom}}_A(D,-)}
  &&
  \mathrm{mod}(A^{\mathrm{op}}) \ar@/_/[u]_{{\mathrm{Hom}}_{A^{\mathrm{op}}}(D,-)}
  }
\]
the bimodule $D$ is invertible if and only if the canonical morphisms
${\mathrm{Hom}}_A(D,A) \otimes_A D\to A$ 
and $D \otimes_A {\mathrm{Hom}}_{A^{\mathrm{op}}}(D,A) \to A$ are bijective.

\subsubsection{}
\label{sec:invertible-dg-lambda-1}
The following result gives a sufficient condition for
$D\sharp\,^\sigma H$ to be invertible.
\begin{prop}
  Let $H$ be a Hopf algebra with invertible antipode. Let $A$ be an
  $H$-module $\k$-algebra.
  Let $D$ be an $H_{S^{2i}}$-equivariant  $A$-bimodule (for some $i\in
  \mathbb Z$) which is invertible as an  $A$-bimodule. Let $\sigma$ be
  an automorphism of $H$ which commutes with $S^2$ and satisfies the
  identity $\sigma(h)_1\otimes \sigma(h)_2= S^{2i}(h_1)\otimes
  \sigma(h_2)$ in $H$. Then, $D\sharp\,^\sigma H$ is an invertible
  $\Lambda$-bimodule.
\end{prop}
\begin{proof} The conclusion is a consequence of following the
  assertion proved below: 
  ${\mathrm{Hom}}_A(D,A)\sharp 
  \,^{\sigma^{-1}} H$ and ${\mathrm{Hom}}_{A^{\mathrm{op}}}(D,A)\sharp
  \,^{\sigma^{-1}}H$ are right and left inverses of
  $D\sharp\,^{\sigma} H$,
  respectively. Note that ${\mathrm{Hom}}_A(D,A)$ and
  ${\mathrm{Hom}}_{A^{\mathrm{op}}}(D,A)$ are $H_{S^{-2i}}$-equivariant 
  $A$-bimodules (see \ref{sec:equiv-acti-duals-1}) and 
  $\sigma^{-1}$ satisfies the identity $\sigma^{-1}(h)_1\otimes
  \sigma^{-1}(h)_2= S^{-2i}(h_1)\otimes \sigma^{-1}(h_2)$.

  Using part (2) of the lemma in \ref{sec:26} yields isomorphisms of 
  $\Lambda$-bimodules
  \[
  \begin{array}{l}
  (D\sharp \,^\sigma H)\underset{\Lambda}{\otimes}
    ({\mathrm{Hom}}_A(D,A)\sharp \,^{\sigma^{-1}}H) \simeq
    (D\underset{A}{\otimes}{\mathrm{Hom}}_A(D,A))\sharp H\\
    ({\mathrm{Hom}}_{A^{\mathrm{op}}}(D,A)\sharp\,^{\sigma^{-1}} H)
    \underset{\Lambda}{\otimes} (D\sharp \,^{\sigma} H) \simeq
    ({\mathrm{Hom}}_{A^{\mathrm{op}}}(D,A)\underset{A}{\otimes} D)\sharp H\,.
    \end{array}
    \]
    Note that the $D\underset A \otimes {\mathrm{Hom}}_A(D,A)$ and
    ${\mathrm{Hom}}_{A^{\mathrm{op}}}(D,A)\underset A\otimes D$ are $H$-equivariant $A$-bimodules
    (see \ref{sec:26}).
  Moreover, the following mappings are isomorphisms in ${\mathrm{mod}}(A^e)$
  \[
  \begin{array}{rclccrcl}
    D\underset{A}{\otimes} {\mathrm{Hom}}_A(D,A)& \to & A
    &&&
        {\mathrm{Hom}}_{A^{\mathrm{op}}}(D,A)\underset{A}{\otimes} D & \to & A \\
    d\otimes f & \mapsto & f(d) &&& f \otimes d & \mapsto &f(d)\,.
  \end{array}
  \]
  These are actually
  $H$-linear as proved by the two following
  computations where $f$ lies in ${\mathrm{Hom}}_A(D,A)$ and ${\mathrm{Hom}}_{A^{\mathrm{op}}}(D,A)$, respectively
  \[
  \begin{array}{rcl}
    (S^{2i}(h_2)\rightharpoonup f)(h_1\rightharpoonup d)
    & = &
          h_3\rightharpoonup f(S^{-1}(h_2)\rightharpoonup
          (h_1\rightharpoonup d)) \\
    & = &
          h\rightharpoonup f(d) \\
  \end{array}
  \]
  and
  \[
  \begin{array}{rcl}
    (h_1\rightharpoonup f)(S^{-2i}(h_2)\rightharpoonup d)
    & = &
          h_1 \rightharpoonup f(S^{1-2i}(h_2)\rightharpoonup
          (S^{-2i}(h_3)\rightharpoonup d)) \\
    & = &
          h \rightharpoonup f(d)\,.
  \end{array}
  \]
  Therefore, $D\underset A\otimes {\mathrm{Hom}}_A(D,A)\simeq A$ and
  ${\mathrm{Hom}}_{A^{\mathrm{op}}}(D,A)\underset A\otimes D\simeq A$ as
  $H$-equivariant $A$-bimodules. 
  Thus, there are isomorphisms of  $\Lambda$-bimodules
  \[
  \begin{array}{l}
  (D\underset{A}{\otimes}{\mathrm{Hom}}_A(D,A))\sharp H \simeq \Lambda \\
  ({\mathrm{Hom}}_{A^{\mathrm{op}}}(D,A)\underset{A}{\otimes} D)\sharp H \simeq
    \Lambda\,.
    \end{array}
  \]
This proves that ${\mathrm{Hom}}_A(D,A)\sharp\,^{\sigma^{-1}} H$  and ${\mathrm{Hom}}_{A^{\mathrm{op}}}(D,A)\sharp \,
^{\sigma^{-1}} H$ are right inverse and left inverse to
$D\sharp\,^\sigma H$, respectively.
\end{proof}

\subsubsection{}
\label{sec:invertible-dg-lambda-2}
The following result gives a necessary condition for $D\sharp\,^\sigma
H$ to be invertible.
\begin{prop}
  Let $H$ be a Hopf algebra with invertible antipode. Let $A$ be an
  $H$-module $\k$-algebra.
  Let $D$ be an $H_{S^{2i}}$-equivariant $A$-bimodule. Let $\sigma\in
  {\mathrm{Aut}}_{\k-{\mathrm{alg}}}(H)$ be such that $\sigma\circ S^2=S^2\circ
  \sigma$ and such that the identity $\sigma(h)_1\otimes \sigma(h)_2 =
  S^{2i}(h_1) \otimes \sigma(h_2)$ holds in $H$. Assume that
  $D\sharp\,^\sigma H$ is invertible as a $\Lambda$-bimodule. Then,
  $D$ is invertible as an $A$-bimodule.
\end{prop}
\begin{proof}
  Since $(D\sharp \,^\sigma H)\underset \Lambda \otimes A\simeq D$ in
  ${\mathrm{mod}}(A)$, and since $D\sharp \,^\sigma H$ is invertible as a
  $\Lambda$-bimodule, then $D$ is finitely generated and projective in
  ${\mathrm{mod}}(A)$. Similarly, since $A\underset \Lambda \otimes
  (D\sharp \,^\sigma H) \underset{~\ref{sec:45}} \simeq A\underset
  \Lambda \otimes (H^{\sigma^{-1}}\sharp D) \simeq D$ in ${\mathrm{mod}}(A^{\mathrm{op}})$, then $D$ is also finitely generated and
  projective in ${\mathrm{mod}}(A^{\mathrm{op}})$.

  In order to prove that $D$ is invertible as an $A$-bimodule, it
  suffices to prove that the canonical mappings $D\underset A \otimes
  {\mathrm{Hom}}_A(D,A)\to A$ and ${\mathrm{Hom}}_{A^{\mathrm{op}}}(D,A) \underset{A}\otimes D
  \to A$ are bijective.
  Consider the former one. The following mapping is an
  isomorphism in ${\mathrm{mod}}(\Lambda^e)$ (see
  \ref{sec:equiv-acti-duals-1} and \ref{sec:26})
  \begin{equation}
    \label{eq:46}
    \begin{array}{rcl}
      (D\underset A \otimes {\mathrm{Hom}}_A(D,A)) \sharp H 
      &         \longrightarrow
      & (D\sharp \,^\sigma H) \underset \Lambda \otimes
        ({\mathrm{Hom}}_A(D,A)\sharp\,^{\sigma^{-1}} H) \\
      (d\otimes \varphi) \otimes \ell
      & \longmapsto & (d\otimes 1) \underset \Lambda \otimes (\varphi
        \otimes \ell)
    \end{array}
  \end{equation}
  Combining the isomorphisms given in parts (1) and (2) of
  \ref{sec:equiv-acti-duals-2} yields the following isomorphism in
  ${\mathrm{mod}}(\Lambda^e)$
  \begin{equation}
    \label{eq:47}
    \begin{array}{rcl}
      {\mathrm{Hom}}_A(D,A) \sharp\,^{\sigma^{-1}} H 
      &\longrightarrow
      & {\mathrm{Hom}}_\Lambda(D\sharp\,^\sigma H, \Lambda) \\
      \varphi \otimes \ell
      & \longmapsto
      & (d\otimes h\mapsto \sigma^{-1}(h_2)
        \varphi(S^{-1-2i}(h_1)\rightharpoonup d)\ell)\,.
    \end{array}
  \end{equation}
  Now, since $D\sharp\,^\sigma H$ is invertible, the following
  canonical mapping is an isomorphism in ${\mathrm{mod}}(\Lambda^e)$
  \begin{equation}
    \label{eq:48}
    (D\sharp \,^\sigma H) \underset \Lambda \otimes
    {\mathrm{Hom}}_\Lambda(D\sharp\,^\sigma H,\Lambda) \to \Lambda\,.
  \end{equation}
  Thus, combining \eqref{eq:46}, \eqref{eq:47} and \eqref{eq:48}
  yields an isomorphism in ${\mathrm{mod}}(\Lambda^e)$
  \begin{equation}
    \label{eq:49}
    \begin{array}{rcl}
      (D\underset A \otimes {\mathrm{Hom}}_A(D,A))\sharp H
      & \to & \Lambda \\
      (d\otimes \varphi) \otimes \ell
      & \mapsto & \varphi(d)\ell\,.
    \end{array}
  \end{equation}
  Note that the canonical mapping
  $D\underset A\otimes {\mathrm{Hom}}_A(D,A) \to A$ is obtained upon
  applying
  $-\underset \Lambda \otimes A\colon {\mathrm{mod}}(A^e)\to
  {\mathrm{mod}}(A)$ to \eqref{eq:49}, hence it is bijective.

  Using similar considerations proves that the canonical mapping
  ${\mathrm{Hom}}_{A^{\mathrm{op}}}(D,A)\underset A\otimes D \to A$ is
  bijective. Thus, $A$ is invertible as an $A$-bimodule.
\end{proof}

\subsubsection{}
\label{sec:corollary}
Combining the results in \ref{sec:invertible-dg-lambda-1} and
\ref{sec:invertible-dg-lambda-2} yields the following corollary.

\begin{cor}
  Let $H$ be a Hopf algebra with invertible antipode. Let $A$ be an
  $H$-module $\k$-algebra.  Let $D$ be an $H_{S^{2i}}$-equivariant
  $A$-bimodule (for some $i\in \mathbb Z$). Let $\sigma$ be an
  automorphism of $H$ which commutes with $S^2$ and satisfies the
  identity
  $\sigma(h)_1\otimes \sigma(h)_2= S^{2i}(h_1)\otimes \sigma(h_2)$ in
  $H$. Then, the following assertions are equivalent.
  \begin{enumerate}[(i)]
  \item $D$ is invertible as an $A$-bimodule.
  \item $D\sharp^\sigma H$ is invertible as an $A\sharp H$-bimodule.
  \end{enumerate}
\end{cor}

\section{The inverse dualising complex of \texorpdfstring{$A\sharp H$}{AH}}
\label{sec:proof-main-result}

This section assumes that $S$ is invertible and describes
${\mathrm{RHom}}_{\Lambda^e}(\Lambda,\Lambda^e)$. Starting from the observations
that $A\in \mathcal C(\Delta_0)$ and
$\Lambda^e\otimes_{\Delta_0}A \simeq \Lambda$ in
$\mathcal C(\Lambda^e)$, the description is obtained by transforming
${\mathrm{RHom}}_{\Lambda^e}(\Lambda^e\otimes_{\Delta_0}A,\Lambda^e)$ using a
series of adjunctions and the existence of some
$D_A\in \mathcal C(\Delta_1)$ such that
$D_A \simeq {\mathrm{RHom}}_{A^e}(A,A^e)$ in $\mathcal D(A^e)$. For this
purpose, Section~\ref{sec:ident-involv-dg} collects some needed
identities in $\Delta_0$ and $\Delta_1$, Section~\ref{sec:TensorH}
proves technical details used in the series of adjunctions,
Section~\ref{sec:prel-homol-alg} establishes a sufficient condition
for $\Lambda$ to be homologically smooth, Section~\ref{sec:dualA}
introduces $D_A$, and Section~\ref{sec:inverse-dual-compl} gives the
description of ${\mathrm{RHom}}_{\Lambda^e}(\Lambda,\Lambda^e)$.

\subsection{Identities on the involved dg algebras}
\label{sec:ident-involv-dg}

\subsubsection{}
\label{sec:Delta}

The following identities hold true in $\Delta_0$
\begin{equation}
  \label{eq:6}
  \begin{array}{rcl}
h\times
    (a\otimes b)
    & = &
          ((h_1\rightharpoonup a)\otimes (h_3\rightharpoonup
          b))\times  h_2\\
    & =&
         ((h_1\otimes
    S(h_3))\rightharpoonup (a\otimes b))\times h_2\,.
    \end{array}
\end{equation}
\begin{equation}
  \label{eq:2}
\begin{array}{rcl}
  (a\otimes b) \times h & = & h_2\times ((S^{-1}(h_1)\rightharpoonup a)
  \otimes (S(h_3)\rightharpoonup b)) \\
  & = & h_2\times ((S^{-1}(h_1)\otimes S^2(h_3))\rightharpoonup
        (a\otimes b))\,.
\end{array}
\end{equation}
The following identities hold true in $\Delta_1$
\begin{equation}
  \label{eq:5}
h\times (a\otimes b)= (h_1\rightharpoonup a\otimes
S^2(h_3)\rightharpoonup b) \times h_2\,.
\end{equation}
\begin{equation}
\label{eq:3}
(a\otimes b) \times h = h_2\times (S^{-1}(h_1)\rightharpoonup a
\otimes S^3(h_3)\rightharpoonup b)\,.
\end{equation}

\subsubsection{Dg module structures over $A$ and $A^e$}
\label{sec:ModuleStructuresOnA}

There is a natural structure of left dg $\Lambda$-module on $A$ defined by
$(a h)\rightharpoonup x= a\,(h\rightharpoonup x)$.
Also there is a natural structure of left dg $\Delta_0$-module on $A$ defined by
$(a\otimes b\otimes h)\cdot  x= (-1)^{{\mathrm{deg}}(b)\,{\mathrm{deg}}(x)}a (h\rightharpoonup x) b$.

\begin{lem}
  There is an isomorphism
  $\Lambda^e\underset{\Delta_0}{\otimes} A\to \Lambda$ in
  $\mathcal C(\Lambda^e)$ which maps $(u\otimes v) \otimes a$ to
  $(-1)^{\mathrm{deg}(v)\,\mathrm{deg}(a)}uav$ (where $u,v\in \Lambda$
  and $a\in A$, $v$ and $a$ being homogeneous).
\end{lem}
\begin{proof}
Let $\mu\colon \Lambda^e\otimes A\to \Lambda$ be defined by
$\mu((u\otimes v) \otimes a)= uav$. Then
\[
\mu(((u\otimes v)\times
(h_1\otimes S(h_2)))\otimes a) = \pm u\underset{h\rightharpoonup
  a}{\underbrace{h_1aS(h_2)}}v = \mu((u\otimes v)\otimes
  (h\rightharpoonup a))
\]
and
$\mu(((u\otimes v)\times (b\otimes c)) \otimes a) = \pm ubacv= \pm
\mu((u\otimes v) \otimes (bac)$. Whence the
existence of the morphism in $\mathcal C(\Lambda^e)$ given in the
statement of the lemma. Denote it by $\nu$. 
 
Let $\lambda\colon \Lambda\to \Lambda^e\underset{\Delta_0}{\otimes}
  A$ be defined by $a h \mapsto (a\otimes h)\underset{\Delta_0}\otimes 1$. Then
$\lambda$ is a right inverse for $\mu$ and
\[
\begin{array}{rcl}
  \lambda\circ \mu ((ah\otimes bk) \otimes c) 
  & = &
         \pm \lambda(ahcbk) = \pm  (a(h_1\rightharpoonup c) (h_2\rightharpoonup
        b) \otimes h_3k)\underset{\Delta_0}\otimes 1 \\
  & = & 
        \pm ( a \otimes h_3 k) \underset{\Delta_0}\otimes (h_1\rightharpoonup
c) (h_2\rightharpoonup b) \\
  & = &
        (a \otimes (h_2\rightharpoonup b) h_3 k) \underset{\Delta_0}\otimes (h_1\rightharpoonup c)
  \\
 & = & 
       ( (a\otimes (h_3\rightharpoonup b) h_4k) \times (h_1 \otimes S(h_2)))
\underset{\Delta_0}\otimes c \\
  & = & (ah \otimes bk) \underset{\Delta_0}\otimes c\,.
\end{array}
\]
Thus, $\lambda$ is an isomorphism in $\mathcal C(\Lambda^e)$.
\end{proof}

\subsection{Useful (bi)module structures on morphism spaces}
\label{sec:TensorH}

\begin{lem}
Let $U\in \mathcal C(\Delta_1)$ and $Y\in \mathcal C(H^{\mathrm{op}})$.
\begin{enumerate}
\item 
$U\otimes H^e$ has a  structure of dg  $H-\Lambda^e$-bimodule where the action of $H$ is given by
$\ell \rightharpoonup (u\otimes h\otimes k) = \ell_2\rightharpoonup u
\otimes S^2(\ell_3)h \otimes k S(\ell_1)$
 and the action of $\Lambda^e$ is given (in bimodule notation) by
\[
\begin{array}{rcl}
  a(u \otimes h \otimes k)b
  & = & ((S^{-1}(k_1)\rightharpoonup a) u  (h_1\rightharpoonup b)) 
         \otimes h_2\otimes k_2\\
  h'(u\otimes h\otimes k)k'
  & = & u\otimes hk'\otimes h'k\,,
\end{array}
\]
\item the canonical mapping  $U\otimes H^e\to U\underset{A^e}{\otimes}
  \Lambda^e$ is both $H$-linear and $\Lambda^e$-linear,
\item $U\otimes H\otimes Y$ has a  structure of dg right
  $\Lambda^e$-module such that (in bimodule notation)  $a(u\otimes
  \ell\otimes y) b  = au(\ell_1\rightharpoonup b) \otimes \ell_2 \otimes
  y$ and $h(u\otimes \ell\otimes y)k
   = (h_1\rightharpoonup u) \otimes S^2(h_2)\ell k\otimes
  y\leftharpoonup S^{-1}(h_3)$,
\item the mapping $U\otimes H \otimes Y \to Y\underset{H}{\otimes} (U \otimes 
  H^e)$ defined by
  \[
  u\otimes \ell \otimes y
  \mapsto (-1)^{\mathrm{deg}(y)\,\mathrm{deg}(u)}y\otimes (u\otimes
  \ell\otimes 1)
  \]
  is a $\Lambda^e$-linear isomorphism.
  \end{enumerate}
\end{lem}
\begin{proof}
  (1) By construction the given actions define a structure of dg 
  $H-(H^e)$-bimodule. The given action of $A^e$ on $U\otimes H^e$ may
  be rewritten as the 
  following composite map
\[
(U\otimes H^e)\otimes A^e \xrightarrow{{\mathrm{Id}}\otimes \tau} U\otimes
A^e\otimes H^e \xrightarrow{\alpha\otimes {\mathrm{Id}}} U\otimes H^e\,.
\]
Here, $\tau\colon H^e\otimes A^e\to A^e\otimes H^e$ is defined by
$\tau(h\otimes k\otimes a\otimes b)=(h_1\rightharpoonup a\otimes
S^{-1}(k_1)\rightharpoonup b) \times (h_2\otimes k_2)$, and
$\alpha\colon U\otimes A^e\to U$ is the action of $A^e$ 
on $U$ inherited from the action of $\Delta_1$.
These considerations together with (\ref{eq:4}) explain that
the given actions in the statement of the lemma define a structure of
right dg $\Lambda^e$-module on $U\otimes H^e$. Hence, there only remains
to prove that the given actions of $H$ on the left and of $A^e$ on the
right commute in order to prove that $U\otimes H^e\in \mathcal
C(H\otimes (\Lambda^e)^{\mathrm{op}})$. To this end, adopt the
$A$-bimodule notation to show that  $a(\ell \rightharpoonup (u \otimes
h\otimes k))b$ is equal to
\[
\begin{array}{cl}
  & a(\ell_2\rightharpoonup u
    \otimes S^2(\ell_3)h \otimes k S(\ell_1))b\\
  = &
      (((\ell_2 S^{-1}(k_1))\rightharpoonup a) (\ell_3\rightharpoonup
      u) ((S^2(\ell_4)h_1)\rightharpoonup b))
      \otimes
      S^2(\ell_5) h_2 \otimes k_2 S(\ell_1) \\
  = &
      (\ell_2\rightharpoonup ((S^{-1}(k_1)\rightharpoonup a) u
      (h_1\rightharpoonup b)))
      \otimes S^2(\ell_3)h_2 \otimes k_2 S(\ell_1) \\
  = &
      \ell \rightharpoonup (a(u\otimes h \otimes k)b)\,.
\end{array}
\]


(2) The action of $H$ on $U\underset{A^e}\otimes \Lambda^e$ which is
being considered here is the one given in part (1) of the lemma in
\ref{sec:AdjunctionDelta}. 
The given mapping is $H$-linear. Note that the structure of right dg
$\Lambda^e$-module of $U\underset{A^e}{\otimes}\Lambda^e$   
is inherited from the one of $\Lambda^e$ itself. Hence, the given mapping
is $\Lambda^e$-linear. 


(3) The given actions define structures of right dg $H^e$-module and
right dg $A^e$-module on $U\otimes H\otimes Y$. They form a structure
of right dg $\Lambda^e$-module in view of the following computations
(with bimodule notation)
\[
\begin{array}{rcl}
  (h_1\rightharpoonup a) (h_2(u\otimes \ell \otimes y))
  & = &
        (h_1\rightharpoonup a) (h_2\rightharpoonup u) \otimes S^2(h_3)
        \ell \otimes y\leftharpoonup S^{-1}(h_4) \\
  & = &
        h(a(u\otimes h \otimes y))\,, \\ \\
  ((u\otimes \ell \otimes y)(k_1\rightharpoonup b))k_2
  & = &
        u(\ell_1\rightharpoonup (k_1\rightharpoonup b)) \otimes \ell_2
        k_2 \otimes y \\
  & = &
        ((u\otimes \ell \otimes y)k)b
\end{array}
\]
\[
\begin{array}{rcl}
  h(( u\otimes \ell \otimes y)b)
  & = &
        h_1\rightharpoonup (u (\ell_1\rightharpoonup b))
        \otimes S^2(h_2)\ell_2 \otimes y \leftharpoonup S^{-1}(h_3) \\
  & = &
        (h_1\rightharpoonup u) ((S^2(h_2) \ell_1) \rightharpoonup b)
        \otimes S^2(h_3) \ell_2
        \otimes y \leftharpoonup S^{-1}(h_4) \\
  & = &
        (h(u\otimes \ell \otimes y))b\,.
\end{array}
\]


(4)
Since $U\otimes H^e\in \mathcal C(H\otimes (\Lambda^e)^{\mathrm{op}})$,
then  $Y\underset
H\otimes (U\otimes H^e)\in \mathcal C((\Lambda^e)^{\mathrm{op}})$. Note the
following identity in $U\otimes H^e$ 
\begin{equation}
\label{eq:7}
(u\otimes h\otimes k) = S^{-1}(k_3) \rightharpoonup
(k_1\rightharpoonup u \otimes S^2(k_2)h \otimes 1)\,,
\end{equation}
Indeed,
\[
\begin{array}{rcl}
  S^{-1}(k_3) \rightharpoonup
  (k_1\rightharpoonup u \otimes S^2(k_2)h \otimes 1)
  & = &
        S^{-1}(k_2) \rightharpoonup (k_1\rightharpoonup u) \otimes h
         \otimes k_3 \\
  & = &
        (u\otimes h\otimes k)\,.
\end{array}
\]
The following computations in $Y\underset H\otimes (U\otimes H^e)$
show that the mapping from $U\otimes H\otimes Y$ to $Y\underset
H\otimes (U\otimes H^e)$ given in the statement of the lemma is $\Lambda^e$-linear,
\[
\begin{array}{rcl}
  a(y\underset H\otimes (u\otimes \ell\otimes 1) ) b
  & = &
        (-1)^{\mathrm{def}(y)\,\mathrm{deg}(a)}y \underset H\otimes (a(u\otimes \ell\otimes 1) b) \\
  & = &
        (-1)^{\mathrm{def}(y)\,\mathrm{deg}(a)} y \underset H\otimes (au(\ell_1\rightharpoonup b) \otimes
        \ell_2 \otimes 1)
\end{array}
\]
and
\[
\begin{array}{rcl}
  h ( y\underset H\otimes (u\otimes\ell\otimes 1))k
  & = &
        y\underset H\otimes h(u\otimes \ell\otimes 1) k \\
  & = &
        y \underset H\otimes (u \otimes \ell k\otimes h ) \\
  & \underset{\eqref{eq:7}} = &
                                y \underset H\otimes S^{-1}(h_3)
                                \rightharpoonup
                                (h_1\rightharpoonup u \otimes S^2(h_2)
                                \ell k \otimes 1) \\
  & = &
        y \leftharpoonup S^{-1}(h_3) \underset H\otimes
        (h_1\rightharpoonup u \otimes S^2(h_2)\ell k \otimes 1)\,.
\end{array}
\]
The given mapping is  bijective with inverse 
$Y\underset{H}{\otimes} (U\otimes H^e) \to U\otimes H\otimes Y$
well-defined by $y\otimes (u\otimes h\otimes k) \mapsto
(-1)^{\mathrm{deg}(y)\,\mathrm{deg}(u)}
(k_1\rightharpoonup u) \otimes
S^2(k_2) h\otimes y\leftharpoonup S^{-1}(k_3)$ (see \eqref{eq:7}).
\end{proof}

\subsection{Preliminaries on homological algebra}
\label{sec:prel-homol-alg}
This section gives some needed material to transform
$\mathrm{RHom}_{\Lambda^e}(\Lambda,\Lambda^e)$ using adjunctions. It
also proves that $\Lambda$ is homologically smooth if $A$ and $H$ are
so.

\subsubsection{}
\label{sec:prel-homol-alg-2}
The following result gives sufficient conditions for
${\mathrm{Hom}}_{A^e}(X,M)$ to be homotopically injective in $\mathcal C(H)$
given $M\in \mathcal C(\Lambda^e)$ and $X\in \mathcal
C(\Delta_0)$. Note that the structure of left dg $H$-module on 
${\mathrm{Hom}}_{A^e}(X,M)$ is taken from \ref{sec:AdjunctionDelta} (part (2)).
\begin{lem}
  Let $X\in \mathcal C(\Delta_0)$ and $M\in \mathcal C(\Lambda^e)$. For
  every $N\in \mathcal C(H)$, there is a natural structure of
  left dg $\Delta_0$-module on ${\mathrm{Hom}}_{\k}(N,M)$ such that (functorially)
\[
{\mathrm{Hom}}_H(N,{\mathrm{Hom}}_{A^e}(X,M))\simeq {\mathrm{Hom}}_{\Delta_0}(X,{\mathrm{Hom}}_{\k}(N,M))\,.
\]
 As a consequence, ${\mathrm{Hom}}_{A^e}(X,M)$ is homotopically injective in $\mathcal C(H)$ if
  and only if
 $X$ is homotopically projective in $\mathcal C(\Delta_0)$.
\end{lem}
\begin{proof}
On the one hand, the following action of $H$ on ${\mathrm{Hom}}_{\k}(N,M)$ is a
structure of left dg
$H$-module
\begin{equation}
  \label{eq:37}
(h\rightharpoonup f) (n) = h_1 f(S^{-1}(h_2)\rightharpoonup n) S(h_3)\,.
\end{equation}
On the other hand, the action of $A^e$ on $M$ defines a structure of 
left dg  $A^e$-module 
 on ${\mathrm{Hom}}_{\k}(N,M)$. Those two structures form a structure of left dg 
$\Delta_0$-module. This claim follows from the following
computation where $f\in {\mathrm{Hom}}_\k(N,M)$,
\[
\begin{array}{rcl}
  (h\rightharpoonup (afb))(n) 
  & = &
        h_1 af(S^{-1}(h_2) \rightharpoonup n) bS(h_3) \\
  & = &
        (h_1\rightharpoonup a) h_2 f(S^{-1}(h_3) \rightharpoonup n)
        S(h_4) (h_5\rightharpoonup b) \\
  & = &
        ((h_1\rightharpoonup a) (h_2\rightharpoonup f)
        (h_3\rightharpoonup b))(n)\,.
\end{array}
\]

Here is a mapping $F\colon {\mathrm{Hom}}_H(N,{\mathrm{Hom}}_{A^e}(X,M))\to
{\mathrm{Hom}}_{\Delta_0} (X,{\mathrm{Hom}}_\k(N,M))$ which fits the conclusion of the
lemma. 
Let $\lambda\in {\mathrm{Hom}}_H(N,{\mathrm{Hom}}_{A^e}(X,M))$ and
define $\mu\in {\mathrm{Hom}}_{\k}(X,{\mathrm{Hom}}_{\k}(N,M))$ by
$\mu(x)(n)=(-1)^{{\mathrm{deg}}(n)\cdot{\mathrm{deg}}(x)} \lambda(n)(x)$. Then,
$\mu$ is $\Delta_0$-linear.  Indeed,
\[
\begin{array}{rcl}
  \mu((a\otimes b)x) (n)
  & = &
        (-1)^{{\mathrm{deg}}(n)\cdot ({\mathrm{deg}}(a\otimes b) +{\mathrm{deg}}(x))}
        \lambda(n)((a\otimes b) x) \\
  & = &
        (-1) ^{{\mathrm{deg}}(n)\cdot ({\mathrm{deg}}(a\otimes b) +{\mathrm{deg}}(x))
        + {\mathrm{deg}}(a\otimes b) \cdot {\mathrm{deg}}(\lambda(n))}
        (a\otimes b) \lambda(n)(x) \\
  & = &
        (-1)^{{\mathrm{deg}}(a\otimes b)\cdot {\mathrm{deg}}(\lambda)} \mu(x)(n)        
\end{array}
\]
 and
\[
\begin{array}{rcl}
  (h\rightharpoonup
  \mu(x))(n)
  & \underset{~\eqref{eq:37}}= & 
        h_1 \mu(x)(S^{-1 }(h_2)\rightharpoonup n) S(h_3) \\
  & = &
        \pm h_1\lambda(S^{-1}(h_2)\rightharpoonup n)(x) S(h_3) \\
  & = &
        \pm h_1 (S^{-1}(h_2)\rightharpoonup \lambda(n))(x) S(h_3) \\
  & \underset{~\eqref{eq:30}}= &
                                 \pm
                                 h_1 S(h_2) \lambda(n)(h_3\rightharpoonup x) h_4 S(h_5) \\
  & = &
        \mu(h\rightharpoonup x)(n)\,,
\end{array}
\]
where $\pm$ is the sign $(-1)^{{\mathrm{deg}}(x)\cdot {\mathrm{deg}}(n)}$.
Thus, defining $F(\lambda)$ by  $F(\lambda)=\mu$
yields an injective mapping 
\[
F\colon {\mathrm{Hom}}_H(N,{\mathrm{Hom}}_{A^e}(X,M)) \to 
{\mathrm{Hom}}_{\Delta_0}(X,{\mathrm{Hom}}_{\k}(N,M))\,.
\]

Here is why $F$ is surjective.
Let $\mu\in {\mathrm{Hom}}_{\Delta_0}(X,{\mathrm{Hom}}_{\k}(N,M))$. For every $n$, define
$\lambda(n)\in {\mathrm{Hom}}_{\k}(X,M)$ by $\lambda(n)(x)=(-1)^{{\mathrm{deg}}(x)\cdot {\mathrm{deg}}(n)}\mu(x)(n)$. Then,
$\lambda(n)$ is $A^e$-linear because $\mu$ is $\Delta_0$-linear. Moreover, the 
mapping $N\to {\mathrm{Hom}}_{A^e}(X,M)$ defined by $n\mapsto \lambda(n)$ is
$H$-linear:
\[
\begin{array}{rcl}
  (h\rightharpoonup \lambda(n))(x)
  & \underset{~\eqref{eq:30}}= &
        S^2(h_3)\lambda(n)(S(h_2)\rightharpoonup x) S(h_1) \\
  & = &
        \pm S^2(h_3) \mu(S(h_2)\rightharpoonup x)(n) S(h_1) \\
  & = &
        \pm S^2(h_3) (S(h_2)\rightharpoonup \mu(x))(n) S(h_1) \\
  &  \underset{~\eqref{eq:37}}= &
                                  \pm S^2(h_5) S(h_4) \mu(x)(h_3\rightharpoonup n) S^2(h_2) S(h_1)
  \\
  & = &
        \lambda(h\rightharpoonup n)(x)\,,
\end{array}
\]
where $\pm$ is the degree $(-1)^{{\mathrm{deg}}(x)\cdot {\mathrm{deg}}(n)}$.
Thus, $\lambda\in {\mathrm{Hom}}_H(N,{\mathrm{Hom}}_{A^e}(X,M))$ and
$F(\lambda)=\mu$. This proves that $F$ is surjective, and hence
bijective. And it is functorial. Whence the first assertion of the
lemma.  The remaining assertion follows immediately.
\end{proof}

\subsubsection{}
\label{sec:prel-homol-alg-3}
The following result asserts that $\Lambda$ is homologically smooth if
$A$ and $H$ are so. For ordinary algebras this 
was already proved in \cite[Proposition 2.11]{MR2905560}.
\begin{prop}
  Let $H$ be a Hopf algebra with invertible antipode. Let $A$ be an
  $H$-module dg algebra.
If  $_H\k\in {\mathrm{per}}(H)$ and 
$A\in {\mathrm{per}}(A^e)$, then $\Lambda\in
{\mathrm{per}}(\Lambda^e)$.
\end{prop}
\begin{proof}
 Let $P\to A$ be a cofibrant replacement in $\mathcal
 C(\Delta_0)$. This is also a 
cofibrant replacement in $\mathcal C(A^e)$ because
$\Delta_0\simeq A^e\otimes H$ in $\mathcal C(A^e)$. Moreover, the
induced 
morphism $\Lambda^e\underset{\Delta_0}{\otimes}P\to
\Lambda^e\underset{\Delta_0}{\otimes}A\simeq \Lambda$ is a cofibrant replacement
in $\mathcal C(\Lambda^e)$. Let $Q\to \,_H\k$ be a
cofibrant replacement in $\mathcal C(H)$. In order to prove the
statement of the lemma, it suffices to prove that $\Lambda$ is compact in $\mathcal
D(\Lambda^e)$, that is, for any given family  $(M_i)_{i\in I}$ in
$\mathcal C(\Lambda^e)$
with direct sum denoted by $M$, the canonical mapping
$\oplus_{i\in I}{\mathrm{Hom}}_{\Lambda^e}(\Lambda^e\underset{\Delta_0}{\otimes}P,M_i)\to
{\mathrm{Hom}}_{\Lambda^e}(\Lambda^e\underset{\Delta_0}{\otimes}P,M)$ is a 
quasi-isomorphism. By  adjunction, this reduces to proving
that the canonical mapping $\oplus_{i\in I}{\mathrm{Hom}}_{\Delta_0}(P,M_i)\to
{\mathrm{Hom}}_{\Delta_0}(P,M)$ is a quasi-isomorphism.

 Note that, if $N$ is either $M$ or one of the $M_i$,
then  ${\mathrm{Hom}}_{A^e}(P,N)$
  has a structure of left dg $H$-module which is 
  functorial in $N$ (see~\ref{sec:AdjunctionDelta}, part (2)) and such that there
  is a functorial isomorphism ${\mathrm{Hom}}_{\Delta_0}(P,N)\xrightarrow{\simeq}
  {\mathrm{Hom}}_H(\,_H\k,{\mathrm{Hom}}_{A^e}(P,N))$.

Since ${\mathrm{Hom}}_{A^e}(P,M)$ is homotopically injective
in $\mathcal C(H)$ (see \ref{sec:prel-homol-alg-2}), the functorial isomorphism
${\mathrm{Hom}}_{\Delta_0}(P,M)\xrightarrow{\simeq} 
  {\mathrm{Hom}}_H(\,_H\k,{\mathrm{Hom}}_{A^e}(P,M))$ induces a quasi-isomorphism
\[
{\mathrm{Hom}}_{\Delta_0}(P,M)
\xrightarrow{\lambda}
{\mathrm{Hom}}_H(Q,{\mathrm{Hom}}_{A^e}(P,M))\,. 
\]
Since $P\to A$ is a cofibrant replacement in $\mathcal C(A^e)$ and
since $A\in {\mathrm{per}}(A^e)$, the canonical mapping 
$\oplus_{i\in I} {\mathrm{Hom}}_{A^e}(P,M_i)\to {\mathrm{Hom}}_{A^e}(P,M)$ is a
quasi-isomorphism. Therefore it
induces a quasi-isomorphism
\[
{\mathrm{Hom}}_H(Q,\oplus_{i\in I} {\mathrm{Hom}}_{A^e}(P,M_i)) 
\xrightarrow{\mu_1}
{\mathrm{Hom}}_H(Q,{\mathrm{Hom}}_{A^e}(P,M))\,.
\]
Now, using that $Q\to \,_H\k$ is a cofibrant replacement in $\mathcal
C(H)$ and that $\,_H\k\in {\mathrm{per}}(H)$, it follows that  the 
canonical mapping below is a quasi-isomorphism 
\[
\oplus_{i\in I} {\mathrm{Hom}}_H(Q,{\mathrm{Hom}}_{A^e}(P,M_i))
\xrightarrow{\mu_2}
{\mathrm{Hom}}_H(Q,\oplus_{i\in I} {\mathrm{Hom}}_{A^e}(P,M_i))\,.
\]
The analogues of $\lambda$ for the $M_i$ (instead of for $M$) give
rise to a quasi-isomorphism
\[
\oplus_{i\in I} {\mathrm{Hom}}_{\Delta_0}(P,M_i)
\xrightarrow{\nu}
\oplus_{i\in I} {\mathrm{Hom}}_H(Q, {\mathrm{Hom}}_{A^e}(P,M_i))\,.
\]
Therefore, the canonical mapping $\oplus_{i\in I}
{\mathrm{Hom}}_{\Delta_0}(P,M_i) \to {\mathrm{Hom}}_{\Delta_0}(P,M)$ is a quasi-isomorphism
because it fits into a commutative
diagram
\[
\xymatrix{
 \oplus_{i\in I}
{\mathrm{Hom}}_{\Delta_0}(P,M_i) \ar[rr] \ar[d]_{\nu} && {\mathrm{Hom}}_{\Delta_0}(P,M) \ar[d]^\lambda\\
\oplus_{i\in I} {\mathrm{Hom}}_H(Q,{\mathrm{Hom}}_{A^e}(P,M_i))
\ar[rr]_{\mu_1\circ \mu_2} && {\mathrm{Hom}}_H(Q,{\mathrm{Hom}}_{A^e}(P,M))\,.
}
\]

Thus, $\Lambda$ is compact in
$\mathcal D(\Lambda^e)$, and hence $\Lambda\in {\mathrm{per}}(\Lambda^e)$.
\end{proof}

\subsection{Additional structure on the inverse dualising complex of
  $A$}
\label{sec:dualA}\label{sec:dualA.2}
The following is a direct consequence of
\ref{sec:AdjunctionDelta}. Recall that $A^e$ is a dg $H^e$-module by
means of \eqref{eq:15} and that, if $X\in \mathcal C(\Delta_0)$, then
${\mathrm{Hom}}_{A^e}(X,A^e)$ is a left dg $\Delta_1$-module obtained by combining its
natural structure of dg  $A$-bimodule and  the structure of
left dg $H$-module
 defined by \eqref{eq:30}.
\begin{prop}
  Let $H$ be a Hopf algebra with invertible antipode. Let $A$ be
  an $H$-module dg algebra.
  There exists $D_A\in \mathcal C(\Delta_1)$ such that $D_A\simeq
  \mathrm{RHom}_{A^e}(A,A^e)$ in $\mathcal D((A^e)^{\mathrm{op}})$. 
  More precisely, if $P\to A$ is a quasi-isomorphism in 
$\mathcal C(\Delta_0)$ such that $P$ is cofibrant in $\mathcal C(A^e)$,
then $D_A$ may be taken equal to  ${\mathrm{Hom}}_{A^e}(P,A^e)$.
\end{prop}

Whenever a cochain $\varphi$ of ${\mathrm{Hom}}_{A^e}(P,A^e)$ is given
and denoted symbolically by
$p\mapsto \varphi'(p)\otimes \varphi''(p)$, the structure of dg
$\Delta_1$-module of ${\mathrm{Hom}}_{A^e}(P,A^e)$ is such that
$h\rightharpoonup \varphi$ is the cochain
\begin{equation}
  \label{eq:52}
  p \mapsto S^2(h_3)\rightharpoonup \varphi'(S(h_2)\rightharpoonup p)
  \otimes h_1\rightharpoonup \varphi''(S(h_2)\rightharpoonup p)\,.
\end{equation}
Note that, when $S^2={\mathrm{Id}}_H$, then $(h\rightharpoonup \varphi)(p)
= h_3\rightharpoonup \varphi'(S(h_2)\rightharpoonup p) \otimes
h_1\rightharpoonup \varphi''(S(h_2)\rightharpoonup p)$.

\begin{ex}
Consider the setting of the running example in
\ref{sec:running-example}. Let $P$ be the Koszul resolution
$K^\bullet$. Given $X\in \mathfrak g$ and given a cochain
$\varphi\in {\mathrm{Hom}}_{A^e}(K^\bullet,A^e)$, then $X\rightharpoonup
\varphi$ is the cochain denoted by $\partial_X(\varphi)$
\begin{equation}
  \label{eq:54}
\begin{array}{crcl}
  \partial_X(\varphi) \colon & K^\bullet & \to & A^e \\
  & \omega & \mapsto & \partial_X(\varphi(\omega)) -
                       \varphi(\partial_X(\omega))\,.
\end{array}  
\end{equation}
\end{ex}

\subsection{An expression of an inverse dualising complex of $\Lambda$}
\label{sec:inverse-dual-compl}

\subsubsection{}
\label{sec:4}
Here is a description  of ${\mathrm{RHom}}_{\Lambda^e}(\Lambda,\Lambda^e)$.
\begin{prop}
  Let $A$ be an $H$-module dg algebra where $H$ is a Hopf algebra with
  invertible antipode. Assume that  both
  $A$ and $H$ are homologically smooth. Then, $\Lambda$ is
  homologically smooth and $D_A\otimes H\otimes E_H\simeq
  {\mathrm{RHom}}_{\Lambda^e}(\Lambda,\Lambda^e)$ in $\mathcal
  D((\Lambda^e)^{\mathrm{op}})$ where
  \begin{itemize}
  \item $D_A$ is any dg $\Delta_1$-module isomorphic to
    ${\mathrm{RHom}}_{A^e}(A,A^e)$ in $\mathcal D(A^e)$, see \ref{sec:dualA},
  \item $E_H\in \mathcal C(H^{\mathrm{op}})$ is any cofibrant replacement of
    ${\mathrm{RHom}}_H(\,_H\k,H)$, 
  \item the right dg $\Lambda^e$-module structure is given (with
    bimodule notation) by
    \begin{equation}
      \label{eq:11}
      \begin{array}{rcl}
        a(d\otimes \ell\otimes e)b
        & = &
              ad(\ell_1\rightharpoonup b) \otimes \ell_2 \otimes e \\
        h(d\otimes \ell \otimes e)k
        & = &
              h_1\rightharpoonup d \otimes S^2(h_2)\ell k \otimes
              e\leftharpoonup S^{-1}(h_3)\,.
      \end{array}
    \end{equation}
    \end{itemize}
  \end{prop}
\begin{proof}
  Let $P\to A$ be a quasi-isomorphism in $\mathcal C(\Delta_0)$ such
  that $P$ is cofibrant in $\mathcal C(A^e)$. Note that any cofibrant
  replacement of $A$ in $\mathcal C(\Delta_0)$ fits this requirement
  because $\Delta_0\simeq A^e\otimes H$ in $\mathcal C(A^e)$. Let
  $Q\to \,_H\k$ be a cofibrant replacement in $\mathcal C(H)$. Then,
  ${\mathrm{Hom}}_{A^e}(P,A^e)$ may be used for $D_A$ (see \ref{sec:dualA}) and
  ${\mathrm{Hom}}_H(Q,H)$ may be used for $E_H$.  Let $\Lambda^e\to I$ be a
  fibrant replacement in $\mathcal C((\Lambda^e)^e)$.  Then,
  ${\mathrm{Hom}}_{\Lambda^e}(\Lambda,I)$ has a structure of right dg
  $\Lambda^e$-module inherited by the one of $I$ and
\[
{\mathrm{Hom}}_{\Lambda^e}(\Lambda,I)\simeq \mathrm{RHom}_{\Lambda^e}(\Lambda,\Lambda^e) \ \text{in $\mathcal
  D((\Lambda^e)^{\mathrm{op}})$.}
  \]

Since $\Lambda\simeq \Lambda^e\underset{\Delta_0}{\otimes} A$ in
$\mathcal C(\Lambda^e)$ (see \ref{sec:ModuleStructuresOnA}), there is
an isomorphism in $\mathcal C((\Lambda^e)^{\mathrm{op}})$
\[
{\mathrm{Hom}}_{\Lambda^e}(\Lambda,I)\simeq {\mathrm{Hom}}_{\Delta_0}(A,I)\,.
\]
Since $\Lambda^e\simeq \Delta_0\otimes H$ in
$\mathcal C(\Delta_0)$,  then $I$ is
fibrant in $\mathcal C(\Delta_0)$. Hence,
the quasi-isomorphism $P\to A$ induces a  quasi-isomorphism in
$\mathcal C((\Lambda^e)^{\mathrm{op}})$
\[
{\mathrm{Hom}}_{\Delta_0}(A,I)\xrightarrow{\mathrm{qis}} {\mathrm{Hom}}_{\Delta_0}(P,I)\,.
\]
Using the functorial construction in \ref{sec:AdjunctionDelta} (part
(2)),  the structure of right dg $\Lambda^e$-module  on  
 ${\mathrm{Hom}}_{A^e}(P,I)$ may be extended to a structure of dg  $H-
 \Lambda^e$-bimodule
  such that there is an  isomorphism in $\mathcal C((\Lambda^e)^{\mathrm{op}})$
\[
{\mathrm{Hom}}_{\Delta_0}(P,I)\simeq {\mathrm{Hom}}_H(\,_H\k,{\mathrm{Hom}}_{A^e}(P,I))\,.
\]
It follows from \ref{sec:prel-homol-alg-2} that  ${\mathrm{Hom}}_{A^e}(P,I)$ is
homotopically injective in $\mathcal C(H)$. Therefore, $Q\to \,_H\k$ induces a
 quasi-isomorphism in $\mathcal C ((\Lambda^e)^{\mathrm{op}})$
\[
{\mathrm{Hom}}_H(\,_H\k,{\mathrm{Hom}}_{A^e}(P,I))\xrightarrow{\mathrm{qis}}
{\mathrm{Hom}}_H(Q,{\mathrm{Hom}}_{A^e}(P,I))\,.
\]
Like for ${\mathrm{Hom}}_{A^e}(P,I)$, the structure of right dg $\Lambda^e$-module of
${\mathrm{Hom}}_{A^e}(P,\Lambda^e)$ extends to a structure of dg $H-
\Lambda^e$-bimodule. Since $P$ and $Q$ are cofibrant in $\mathcal
C(A^e)$ and $\mathcal C(H)$, respectively, the
quasi-isomorphism  $\Lambda^e\to I$  
induces a quasi-isomorphism in $\mathcal C ((\Lambda^e)^{\mathrm{op}})$
\[
{\mathrm{Hom}}_H(Q,{\mathrm{Hom}}_{A^e}(P,\Lambda^e))
\xrightarrow{\mathrm{qis}}
{\mathrm{Hom}}_H(Q,{\mathrm{Hom}}_{A^e}(P,I))\,.
\]
Thanks to \ref{sec:AdjunctionDelta} (parts (1) and (6)), there is a
structure of left dg
$H$-module on ${\mathrm{Hom}}_{A^e}(P,A^e)\underset{A^e}{\otimes}\Lambda^e$ such
that the canonical mapping from
${\mathrm{Hom}}_{A^e}(P,A^e)\underset{A^e}{\otimes}\Lambda^e$ to 
${\mathrm{Hom}}_{A^e}(P,\Lambda^e)$ is $H$-linear. This is a quasi-isomorphism
because $A\in {\mathrm{per}}(A^e)$. And it is $\Lambda^e$-linear (to the right) by
functoriality of the involved constructions.  It therefore induces a
 quasi-isomorphism in $\mathcal C((\Lambda^e)^{\mathrm{op}})$
\[
{\mathrm{Hom}}_H(Q,{\mathrm{Hom}}_{A^e}(P,A^e)\underset{A^e}{\otimes}\Lambda^e)
\xrightarrow{\mathrm{qis}}
{\mathrm{Hom}}_H(Q,{\mathrm{Hom}}_{A^e}(P,\Lambda^e))\,.
\]
Recall that ${\mathrm{Hom}}_{A^e}(P,A^e)\in \mathcal C(\Delta_1)$ as detailed in
\ref{sec:AdjunctionDelta} (part (5)). Accordingly (see
\ref{sec:TensorH}, parts (1) and (2)),
${\mathrm{Hom}}_{A^e}(P,A^e)\otimes H^e\in \mathcal C(H\otimes (\Lambda^e)^{\mathrm{op}})$ in such a way that the canonical mapping
${\mathrm{Hom}}_{A^e}(P,A^e)\otimes H^e\to
{\mathrm{Hom}}_{A^e}(P,A^e)\underset{A^e}{\otimes}\Lambda^e$ is both $H$-linear
and $\Lambda^e$-linear. It is an isomorphism because $\Lambda^e\simeq
A^e\otimes H^e$ in $\mathcal C(A^e)$. Therefore, it induces an
isomorphism in $\mathcal C((\Lambda^e)^{\mathrm{op}})$
\[
{\mathrm{Hom}}_H(Q,{\mathrm{Hom}}_{A^e}(P,A^e)\otimes H^e)
\xrightarrow{\sim}
{\mathrm{Hom}}_H(Q,{\mathrm{Hom}}_{A^e}(P,A^e)\underset{A^e}{\otimes}\Lambda^e)\,.
\]
Since $\,_H\k\in {\mathrm{per}}(H)$, the following canonical mapping
is a quasi-isomorphism in $\mathcal C((\Lambda^e)^{\mathrm{op}})$
\[
{\mathrm{Hom}}_H(Q,H)\underset{H}{\otimes}({\mathrm{Hom}}_{A^e}(P,A^e)\otimes H^e)
\xrightarrow{\mathrm{qis}}
{\mathrm{Hom}}_H(Q,{\mathrm{Hom}}_{A^e}(P,A^e)\otimes H^e)
\,.
\]
Finally, using  the structure of  left dg  $\Delta_1$-module on
${\mathrm{Hom}}_{A^e}(P,A^e)$, 
 there is an associated  structure of right
dg $\Lambda^e$-module on
${\mathrm{Hom}}_{A^e}(P,A^e)\otimes H\otimes {\mathrm{Hom}}_H(Q,H)$ introduced
in 
\ref{sec:TensorH} (part (3)). For this structure, there is an
isomorphism in $\mathcal C((\Lambda^e)^{\mathrm{op}})$ (see \ref{sec:TensorH}, part (4))
\[
{\mathrm{Hom}}_{A^e}(P,A^e)\otimes H\otimes {\mathrm{Hom}}_H(Q,H)
\xrightarrow{\sim} {\mathrm{Hom}}_H (Q,H)\underset{H}{\otimes}
({\mathrm{Hom}}_{A^e}(P,A^e)\otimes H^e)\,.
\]
Therefore ${\mathrm{RHom}}_{\Lambda^e}(\Lambda,\Lambda^e)\simeq
{\mathrm{Hom}}_{A^e}(P,A^e)\otimes H\otimes {\mathrm{Hom}}_H(Q,H)$ in $\mathcal
D((\Lambda^e)^{\mathrm{op}})$ for the structure given in \ref{sec:TensorH},
part (3). This structure is precisely the one given in
\eqref{eq:11}. This proves the proposition.
\end{proof}

\subsubsection{}
\label{sec:5}
The description of ${\mathrm{RHom}}_{\Lambda^e}(\Lambda,\Lambda^e)$
made in \ref{sec:4} gets simpler when $H$ satisfies the left
Artin-Schelter condition.  Indeed, keep the setting of \ref{sec:4} and
assume that there exists a natural integer $d$ such that
${\mathrm{dim}}_\k\,{\mathrm{Ext}}^i_H(\,_H\k,H)$ equals $1$ if $i=d$
and $0$ otherwise.  The left homological integral is
$\int_\ell = {\mathrm{Ext}}_H^d(\,_H\k,H)$. Also, denote by
$\int_{\ell}\colon H\to \k$ the algebra homomorphism such that the
right $H$-module structure of ${\mathrm{Ext}}^d_H(\,_H\k,H)$ is given
by $\alpha \leftharpoonup h= \int_\ell(h)\alpha$.  Let
$D_A\in \mathcal C(\Delta_1)$ be such that
$D_A\simeq \mathrm{RHom}_{A^e}(A,A^e)$ in
$\mathcal D((A^e)^{\mathrm{op}})$. Given that
$\Xi_{{\int_\ell}\circ S^{-1}}^r=(\Xi_{\int_\ell}^r)^{-1}$, the
conclusion of \ref{sec:4} entails that
${\mathrm{RHom}}_{\Lambda^e}(\Lambda,\Lambda^e)\simeq D_A\otimes
H[-d]$ in $\mathcal D((\Lambda^e)^{\mathrm{op}})$, where $D_A\otimes H[-d]$ has
structure of dg $\Lambda$-bimodule given by
\begin{equation}
  \label{eq:25}
  \begin{array}{l}
    a(d\otimes \ell)b = (ad(\ell_1\rightharpoonup b)) \otimes \ell_2\\
    h (d\otimes \ell) k  = (h_1\rightharpoonup d)
  \otimes (S^{-2}\circ \Xi_{{\int_\ell}}^r)^{-1}(h_2) \ell k\,.
\end{array}
\end{equation}

To sum up:
\begin{cor}
  \label{sec:31}
  Let $H$ be a Hopf algebra with Van den Bergh duality in dimension
  $d$ (and hence with invertible antipode $S$, see
  \ref{sec:38}). Let $A$ be an $H$-module dg
  algebra. Assume that $A$ is 
  homologically smooth and let $D_A\in \mathcal C(\Delta_1)$ be such
  that $D_A\simeq \mathrm{RHom}_{A^e}(A,A^e)$ in $\mathcal
  D((A^e)^{\mathrm{op}})$ (see \ref{sec:dualA}). Then, $\Lambda$ is
  homologically smooth and
  ${\mathrm{RHom}}_{\Lambda^e}(\Lambda,\Lambda^e)\simeq
  D_A\sharp\,^{(S^{-2}\circ \Xi_{\int_\ell}^r)^{-1}}H[-d]$ in
  $\mathcal D((\Lambda^e)^{\mathrm{op}})$. 
\end{cor}

\section{Application to constructions of Calabi-Yau algebras}
\label{sec:appl-constr-calabi}

This section describes the deformed Calabi-Yau completions of
$\Lambda$. The undeformed case is treated in
Section~\ref{sec:calabi-yau-compl} and the general case is treated in
Section~\ref{sec:deformed-calabi-yau}. Section~\ref{sec:smash-products-with}
makes a specialisation to the case where $H$ is involutive and
Calabi-Yau.

Given a homologically smooth dg algebra $A$ and an integer
$n\in\mathbb Z$, the Calabi-Yau completion $\Pi_n(A)$ is the dg
algebra $T_A(D_A[n-1])$, where $D_A$ is any cofibrant replacement of
${\mathrm{RHom}}_{A^e}(A,A^e)$ in $\mathcal C((A^e)^{\mathrm{op}})$. Every $\alpha\in {\mathrm{HH}}_{n-2}(A)$ yields a
deformation $\Pi_n(A,\alpha)$ called a deformed Calabi-Yau completion
of $A$: Since ${\mathrm{HH}}_{n-2}(A)\simeq H^0{\mathrm{Hom}}_{A^e}(D_A[n-1],A[1])$,
then $\alpha$ is represented by some $0$-cocycle $c\colon D_A[n-1]\to
A[1]$; Then, $\Pi_n(A,\alpha)$ is the unique dg algebra such that
\begin{itemize}
\item it has the same underlying graded algebra as $\Pi_n(A)$,
\item its differential extends the one of $A$,
\item if $d_A\in D_A[n-1]$, then its differential in $\Pi_n(A,\alpha)$
  is the sum of its differential in $D_A[n-1]$ and of $c(d_A)$.
\end{itemize}
$\Pi_n(A)$ is $n$-Calabi-Yau (see \cite{MR2795754}), and so is
$\Pi_n(A,\alpha)$ if $A$ is finitely cellular and $c$ lifts to the
negative cyclic homology of $A$ (see \cite{yeung}).

In this section, $H$ is a Hopf algebra with Van den Bergh
duality in dimension $d$ and $A$
is a homologically smooth 
$H$-module dg 
algebra. Denote $A\sharp H$ by $\Lambda$. Let $n\in\mathbb Z$.
Let $\sigma=S^2\circ \Xi_{\int_\ell\circ S}^r$. Let $D_A$ be a
cofibrant left dg $\Delta_1$-module such that $D_A\simeq
{\mathrm{RHom}}_{A^e}(A,A^e)$ in $\mathcal D(A^e)$. In particular, $D_A$ is
cofibrant in $\mathcal C(A^e)$ (see the proof in
\ref{sec:prel-homol-alg-3}).
Until the end of the section, $\Pi_n(A)$ and $T_A(D_A[n-1])$ are
identified. 
Recall that, for every $\ell \geqslant
1$, the dg $A$-bimodule $D_A^{\underset A \otimes \ell}$ is
$H_{S^{2\ell}}$-equivariant (see \ref{sec:27}). In particular,
$T_A(D_A[n-1])$ is a left dg $H$-module.
Denote $D_A\sharp\,^\sigma H[-d]$ by $D_{\Lambda}$.
\begin{lem}
  $D_\Lambda$ is a cofibrant replacement of
  ${\mathrm{RHom}}_{\Lambda^e}(\Lambda,\Lambda^e)$ in $\mathcal C(\Lambda^e)$.
\end{lem}
\begin{proof}
  Following \ref{sec:5}, it is sufficient to prove that $D_A\sharp
  \,^\sigma H$ is cofibrant in $\mathcal C(\Lambda^e)$. Since $D_A$ is
  cofibrant in $\mathcal C(\Delta_1)$, this follows from the
  isomorphism proved in \ref{sec:43}.
\end{proof}
Until the end of the section, $\Pi_{n+d}(\Lambda)$ and
$T_\Lambda(D_\Lambda[n+d-1])$ are identified. Note that
$D_\Lambda[n+d-1]=D_A\sharp\,^\sigma H[n-1]$.

\subsection{The Calabi-Yau completion of $A\sharp H$}
\label{sec:calabi-yau-compl}

The dg algebra $\Pi_n(A)\sharp\,^{\sigma^*}H$ was
defined in \ref{sec:27}. It has $T_A(D_A[n-1])\otimes H$ as underlying
complex, it contains $A\sharp H$ as a dg subalgebra, and its product
is given by  \eqref{eq:29}.

\begin{prop}
  \begin{enumerate}
    \item 
There is an isomorphism of dg algebras 
$\Pi_n(A)\sharp
\,^{\sigma^*}H\xrightarrow{\sim} \Pi_{n+d}(A\sharp H)$.
\item If $H$ is involutive and Calabi-Yau, then $\Pi_n(A)$ is an
  $H$-module dg algebra and
  $\Pi_n(A)\sharp H\simeq \Pi_{n+d}(A\sharp H)$.
  \end{enumerate}
\end{prop}
\begin{proof}
(1) This follows from the considerations in  \ref{sec:27}.


(2) Following \ref{sec:char-calabi-yau}, $\int_\ell=\epsilon$, and
hence $\sigma={\mathrm{Id}}_H$. Thus (2) follows from (1).
\end{proof}

Note that the isomorphism $\Pi_n(A)\sharp\,^{\sigma^*}H\to
\Pi_{n+d}(A\sharp H)$ in the previous result is uniquely determined by
the following properties
\begin{itemize}
\item it extends the identity mappings $A\to A$ and $H\to H$,
\item it maps every $d_A\in D_A[n-1]$ to $d_A\otimes 1\in D_{\Lambda}[n+d-1]$.
\end{itemize}

\subsection{Deformed Calabi-Yau completions of $A\sharp H$}
\label{sec:deformed-calabi-yau}

This section gives some descriptions of the deformed Calabi-Yau
completions of $A\sharp H$.

\subsubsection{}
\label{sec:41}

For this purpose, the following lemma relates the cohomology classes
in ${\mathrm{HH}}_{n+d-2}(\Lambda)$ to Hochschild cohomology classes of $A$.
\begin{lem}
  The mapping $c\mapsto c_{|D_A\otimes 1}$ yields an isomorphism from
  ${\mathrm{Hom}}_{\Lambda^e}(D_A\sharp\,^\sigma H,\Lambda)$ to the subcomplex of
  ${\mathrm{Hom}}_{A^e}(D_A,\Lambda)$ consisting of those $c$ satisfying the
  identity 
  \begin{equation}
    \label{eq:1}
    c(h\rightharpoonup d) = h_1c(d)\int_\ell(h_2)S^3(h_3)\,.
  \end{equation}
\end{lem}
\begin{proof}
  Note that every dg $\Lambda^e$-module is a dg $\Delta_1$-module by
  restriction-of-scalars along \eqref{eq:31}. By adjunction, there is
  an isomorphism
  \[{\mathrm{Hom}}_{\Lambda^e}(D_A\sharp\,^\sigma H,\Lambda)\simeq 
  {\mathrm{Hom}}_{\Delta_1}(D_A, {\mathrm{Hom}}_{\Lambda^e}(\Lambda\otimes\,^\varphi
  \Lambda,\Lambda))\,,\]
  where $\varphi\in {\mathrm{Aut}}_{\mathrm{dg-alg}}(\Lambda)$ is
  defined by $\varphi(ah)= a\sigma(S^{-2}(h))$ (see
  \ref{sec:43}). Moreover, the dg $\Lambda^e$-module
  ${\mathrm{Hom}}_{\Lambda^e}(\Lambda\otimes \,^\varphi \Lambda,\Lambda)$
  identifies with $\Lambda^\varphi$. Finally,
  \ref{sec:AdjunctionDelta} yields that
  ${\mathrm{Hom}}_{\Delta_1}(D_A,\Lambda^\varphi)$ is the subcomplex of
  ${\mathrm{Hom}}_{A^e}(D_A,\Lambda)$ consisting of those $c$ satisfying the
  identity
  \[
  c(h\rightharpoonup d) = h_1 c(d) \varphi(S^3(h_2))\,.
  \]
  Note that $\varphi(S^3(h)) = \sigma(S(h)) = S^3(h_2)
  \int_\ell(S^2(h_1)) = S^3(h_2) \int_\ell(h_1)$. Whence the
  conclusion of the lemma.
\end{proof}

Note that the inverse of the isomorphism established in the previous lemma maps
every $c$ to the morphism $D_A\sharp\,^\sigma H\to \Lambda\,,d\otimes
\ell \mapsto c(d)\ell$.
Moreover, this lemma identifies ${\mathrm{HH}}_{n+d-2}(\Lambda)$
 with the $0$-th cohomology space of the subcomplex of
${\mathrm{Hom}}_{A^e}(D_A[n-1],\Lambda[1])$ consisting of those morphisms $c$
satisfying the identity \eqref{eq:1}.

\subsubsection{}
\label{sec:32}
In general, a deformed Calabi-Yau completion of $A\sharp H$ may not be
expressed in terms of one of
$A$. Instead, it may be related to a deformation of $\Pi_n(A)\sharp^{\sigma^*}H$. This
deformation is determined by the cocycle introduced by the following
lemma. 
Let $\alpha\in {\mathrm{HH}}_{n+d-2}(\Lambda)$ be represented by  $c\in 
Z^0{\mathrm{Hom}}_{A^e}(D_A[n-1],\Lambda[1])$ (in the sense of \ref{sec:41}) satisfying the identity
\eqref{eq:1}.
\begin{lem}
  There exists a unique degree $1$ square-zero derivation $\partial$
  on the graded algebra $\Pi_n(A)\sharp\,^{\sigma^*}H$ such that
  \begin{enumerate}[(a)]
  \item $\partial_{|\Lambda}$ is the differential of $\Lambda$,
  \item $\partial_{|D_A[n-1]}$ takes
    its values in $\Lambda[1] \oplus D_A[n]$ and is equal to the
    sum of $c\colon D_A[n-1]\to \Lambda[1]$ and the differential of $D_A[n-1]$.
  \end{enumerate}
\end{lem}
\begin{proof}
  Since the graded algebra $\Pi_n(A)\sharp\,^{\sigma^*}H$ is generated
  by $\Lambda$ and $D_A$, the given conditions force $\partial$ to be
  unique when it exists. Note the following identities in $\Lambda$
  \[
  \begin{array}{rcl}
    c(h_1\rightharpoonup d) \sigma (h_2)
    & = &
          h_1c(d) \int_\ell(h_2)
          S^3(h_3) \sigma(h_4) \\
    & = &
          h_1 c(d) \int_\ell(h_2) S^3(h_3) S^2(h_4) \int_\ell(S(h_5))
    \\
    & = &
          hc(d)\,.
  \end{array}
  \]
  Given that $c$ is $A^e$-linear and that $\sigma$ satisfies the
  identity $\sigma(h)_1\otimes \sigma(h)_2=S^2(h_1) \otimes \sigma(h_2)$, these considerations entail that
  there exists a degree $1$ derivation $\partial$ on the graded
  algebra $\Pi_n(A)\sharp\,^{\sigma^*}H$ satisfying (a) and (b). Since
  $c$ is a cocycle, then $\partial^2=0$.
\end{proof}

\subsubsection{}
\label{sec:33}
The following proposition expresses the deformed Calabi-Yau
completions of $A\sharp H$ as deformations of $\Pi_n(A)\sharp\,^{\sigma^*}H$.
\begin{prop}
  Let $\alpha\in {\mathrm{HH}}_{n+d-2}(\Lambda)$ be represented by the
  cocycle $D_\Lambda[n+d-1]\to \Lambda[1]\,,d\otimes \ell \mapsto
  c(d)\otimes \ell$ where  $c\in Z^0{\mathrm{Hom}}_{A^e}(D_A[n-1],\Lambda[1])$ satisfies
  the identity \eqref{eq:1}.
  Then, $\Pi_{n+d}(A\sharp H,\alpha)\simeq
  (\Pi_n(A)\sharp\,^{\sigma^*}H,\partial)$ as dg algebras, where $\partial$
  is as in \ref{sec:32}.
  \end{prop}
\begin{proof}
Following \ref{sec:calabi-yau-compl}, there exists an isomorphism of
dg algebras $\varphi\colon T_A(D_A[n-1])\sharp\,^{\sigma^*}H\to
T_{\Lambda}(D_\Lambda[n+d-1])$
extending the identity mappings $A\to A$ and $H\to H$ and mapping
every $d_A\in D_A[n-1]$ to $d_A\otimes 1\in D_\Lambda[n+d-1]$.
Comparing the differentials of $d_A$ and $d_A\otimes 1$ in
$(\Pi_n(A)\sharp\,^{\sigma^*}H,\partial)$ and $\Pi_{n+d}(A\sharp
H,\alpha)$, respectively, yields that $\varphi$ is also an isomorphism
of dg algebras $(\Pi_n(A)\sharp\,^{\sigma^*}H,\partial) \to
\Pi_{n+d}(A\sharp
H,\alpha)$.
\end{proof}

\subsection{Smash products with involutory Hopf algebras}
\label{sec:smash-products-with}

Assume that $S^2={\mathrm{Id}}_H$ and $H$ is
Calabi-Yau. In this particular case, the following result shows that
every deformed Calabi-Yau completion of $A$ is an $H$-module dg
algebra, and that the resulting smash product is a deformed Calabi-Yau 
completion of $A\sharp H$.

Indeed, $\int_\ell=\epsilon$
(see \ref{sec:char-calabi-yau}) and $\sigma={\mathrm{Id}}_H$.
Then, any $H$-linear cocycle $c\colon D_A[n-1]\to A[1]$
may be considered as lying in $Z^0{\mathrm{Hom}}_{A^e}( D_A[n-1],\Lambda[1])$ and, as
such,  satisfies the identity \eqref{eq:1}; In such a situation, the
corresponding cocycle  (see \ref{sec:41}) in
${\mathrm{Hom}}_{\Lambda^e}(D_\Lambda[n+d-1],\Lambda[1])$ is denoted by
$\overline c$. 
\begin{prop}
  Let $H$ be an involutive Hopf algebra which is moreover Calabi-Yau
  in dimension $d$. Let $A$ be a homologically smooth $H$-module dg
  algebra. Let $n\in \mathbb Z$. Let
  $c\in Z^0 {\mathrm{Hom}}_{A^e}(D_A[n-1],A[1])$ be $H$-linear. Denote by
  $\alpha\in {\mathrm{HH}}_{n-2}(A)$ and
  $\overline \alpha\in {\mathrm{HH}}_{n+d-2}(\Lambda)$ the cohomology
  classes of $c$ and $\overline c$, respectively. Then
  $\Pi_n(A,\alpha)$ is an $H$-module dg algebra and
  $\Pi_n(A,\alpha)\sharp H\simeq \Pi_{n+d}(A\sharp H,\overline
  \alpha)$.
\end{prop}
\begin{proof}
  Denote by $\delta$ the differential of $D_A[n-1]$. Hence, $\delta
  \otimes {\mathrm{Id}}_H$ is the differential of
  $D_{\Lambda}[n+d-1]$. Note that $\overline c$ is given by $\overline
  c(d_A\otimes \ell)=c(d_A)\otimes \ell$ for every $d_A\otimes \ell \in D_\Lambda$.

  The dg algebra $\Pi_n(A,\alpha)$ is uniquely determined by the
  following properties
  \begin{itemize}
  \item it has the same underlying graded algebra as $\Pi_n(A)$,
  \item its differential extends the one of $A$,
  \item if $d_A\in D_A[n-1]$, then its differential is
    $\delta(d_A)+c(d_A)$.
  \end{itemize}
  Since $\Pi_n(A)$ is an $H$-module dg algebra (see
  \ref{sec:calabi-yau-compl}, part (2)) and since $c$ is $H$-linear,
  then $\Pi_n(A,\alpha)$ is an $H$-module dg algebra.

  The dg algebra $\Pi_{n+d}(A\sharp H,\overline \alpha)$ is uniquely
  determined by the following properties
  \begin{itemize}
  \item it has the same underlying graded algebra as
    $\Pi_{n+d}(A\sharp H)$,
  \item its differential extends the one of $A\sharp H$,
  \item if $d_A\otimes \ell\in D_\Lambda[n+d-1]$, then its
    differential is $(\delta \otimes {\mathrm{Id}}_H)(d_A\otimes \ell) +
    \overline c(d_A\otimes \ell)=\delta(d_A)\otimes \ell +
    c(d_A)\otimes \ell$.
  \end{itemize}
  Now, the isomorphism of dg algebras $\Pi_n(A)\sharp H\to
  \Pi_{n+d}(A\sharp H)$ of \ref{sec:calabi-yau-compl} (part (2))
  extends the identity mappings $A\to A$ and $H\to H$  and it maps
  every $d_A\in D_A[n-1]$ to $d_A\otimes 1\in
  D_{\Lambda}[n+d-1]$. Hence, it also is an isomorphism of dg algebras
  $\Pi_n(A,\alpha)\sharp H\to \Pi_{n+d}(A\sharp H,\overline \alpha)$.
\end{proof}

\section{Application to Van den Bergh and Calabi-Yau duality}
\label{sec:appl-dual-non}

Assume that $A$ is a $\k$-algebra. This
section studies when $\Lambda=A\sharp
H$ has Van den Bergh duality (in Section~\ref{sec:18}) or is skew Calabi-Yau (in
Section~\ref{sec:19}). In the latter case
a Nakayama automorphism is given for
$\Lambda$. This is based on the notion of weak homological determinant
given in Section~\ref{sec:17}. This notion  extends the definition of homological
determinant when the latter is not defined properly. Finally, a
characterisation of when $\Lambda$ is Calabi-Yau 
is given in Section~\ref{sec:calabi-yau-duality} when $A$ is connected
($\mathbb N$-)graded
and $H$ is Calabi-Yau.

\subsection{The inverse dualising bimodule}
\label{sec:18}
Following \ref{sec:dualA}, when the antipode of $H$ is invertible, the
cohomology space ${\mathrm{Ext}}^i_{A^e}(A,A^e)$ is an
$H_{S^2}$-equivariant $A$-bimodule (equivalently a left
$\Delta_1$-module). Recall (\ref{sec:25}) that a Hopf algebra with Van
den Bergh duality has an invertible antipode.  Note that in the
particular case where $H$ is Calabi-Yau, it is proved in \cite[Theorem
17]{MR2115022} that if $A$ has Van den Bergh duality then so does
$\Lambda$.
\begin{prop}
   Let $A$ be an $H$-module algebra where $H$ is a Hopf algebra.
   Assume that the antipode $S$ is invertible and that both $A$ and
   $H$ are homologically smooth. Then, the following assertions are
   equivalent
   \begin{enumerate}[(i)]
   \item $A$ and $H$ have Van den Bergh duality,
   \item $\Lambda$ has Van den Bergh duality.
   \end{enumerate}
   When these conditions are satisfied and $n$, $d$ are the
   corresponding homological dimensions of $A$ and $H$, respectively,
   then $\Lambda$ has dimension $n+d$ and
   \[{\mathrm{Ext}}^{n+d}_{\Lambda^e}(\Lambda,\Lambda^e)\simeq {\mathrm{Ext}}^n_{A^e}(A,A^e)\sharp\,^{(S^{-2}\circ \Xi_{\int_\ell}^r)^{-1}}H\,.\]
\end{prop}
\begin{proof}
  The implication $(i)\Rightarrow (ii)$ follows from  \ref{sec:5} and
  \ref{sec:invertible-dg-lambda-1}.

  Conversely, assume that $\Lambda$ has Van den Bergh duality.
  There exist maximal
  integers $n$ and $d$  such 
  that ${\mathrm{Ext}}^n_{A^e}(A,A^e)$ and ${\mathrm{Ext}}^d_H(\,_H\k,H)$ are
  nonzero. Denote ${\mathrm{Ext}}_{A^e}^n(A,A^e)$, ${\mathrm{Ext}}^d_H(\,_H\k,H)$ and ${\mathrm{Ext}}^{n+d}_{\Lambda^e}(\Lambda,\Lambda^e)$ by $D_A$, $E$ and
  $D_\Lambda$, respectively. The proof that $A$ has Van den
  Bergh duality proceeds according the following steps
  \begin{enumerate}[(a)]
  \item $\Lambda$ has Van den Bergh duality in dimension $n+d$,
    moreover ${\mathrm{Ext}}_{A^e}^i(A,A^e)\neq 0$ if and only if $i=n$, and
    ${\mathrm{Ext}}_H^j(\,_H\k,H)\neq 0$ if and only if $j=d$,
  \item given any left $\Lambda$-module $M$, there exists an
    isomorphism of left $\Lambda$-modules
    $D_{\Lambda}\underset{\Lambda}{\otimes} M \simeq
    (D_A\underset{A}{\otimes} M) \otimes E$ where the module
    structure of the right hand-side term is given by $a(d\otimes m
    \otimes e) = ad\otimes m \otimes e$ and $h(d\otimes m \otimes e) =
    h_1\rightharpoonup d \otimes S^2(h_2) m \otimes e\leftharpoonup
    S^{-1}(h_3)$,
    \item $H$ has Van den Bergh duality,
    \item $A$ has Van den Bergh duality.
    \end{enumerate}

    \textbf{Step (a) - } Taking cohomology in the description of
    ${\mathrm{RHom}}_{\Lambda^e}(\Lambda^e,\Lambda)$ (\ref{sec:4}) shows that
    $n+d$ is the largest natural integer such that ${\mathrm{Ext}}^{n+d}_{\Lambda^e}(\Lambda,\Lambda^e)\neq 0$. Therefore, $\Lambda$
    has Van den Bergh duality in dimension $n+d$. The same
    consideration then shows that ${\mathrm{Ext}}^i_{A^e}(A,A^e)=0$ if
    $i\neq n$ and ${\mathrm{Ext}}^j_H(\,_H\k,H)=0$ if $j\neq d$.
    Using \ref{sec:4}, the $\Lambda$-bimodule $D_\Lambda$ is
    identified with $D_A\otimes H \otimes E$ for the structure of
    $\Lambda$-bimodule described in \eqref{eq:11}.
    
    \textbf{Step (b) - } In view of the structure of right
    $\Lambda$-module of $D_\Lambda=D_A\otimes H\otimes E$, there are
    well-defined linear mappings $\Phi\colon (D_A\underset{A}{\otimes}
    M) \otimes E\to D_\Lambda\underset{\Lambda}{\otimes} M$ and
    $\Psi\colon D_\Lambda\underset{\Lambda}{\otimes} M \to
    (D_A\underset A \otimes M) \otimes E$ given by $\Phi(d\otimes m
    \otimes e) = (d\otimes 1\otimes e) \otimes m$ and $\Psi(d\otimes
    \ell \otimes e) \otimes m = d\otimes \ell m \otimes e$. Note the
 identity $(d\otimes \ell \otimes e) \otimes m = (d\otimes 1 \otimes
 e) \otimes \ell m$ in $D_\Lambda\underset \Lambda \otimes
 M$. Therefore $\Phi$ and $\Psi$ are inverse to each other. The claim
 of Step (b) is therefore proved by transporting the left
 $\Lambda$-module structure of $D_\Lambda\underset \Lambda \otimes M$
 to $(D_A\underset A \otimes M) \otimes E$ using $\Phi$ and $\Psi$.

 \textbf{Step (c) - } Let $M=D_\Lambda^{-1}\underset \Lambda\otimes A\in
 {\mathrm{mod}}(\Lambda)$. From the previous step, it follows that $A\simeq
 (D_A\underset A\otimes M)^{({\mathrm{dim}}_\k\,E)}$ in ${\mathrm{mod}}(A)$. Thus, ${\mathrm{dim}}_\k\,E<\infty$. Using
 \ref{sec:suff-cond-right}, \ref{sec:25} and step (a) yields that $H$ has Van
 den Bergh duality in dimension $d$.

 \textbf{Step (d) - } It suffices to prove that $D_A$ is invertible as
 an $A$-bimodule. Since $H$ has Van den Bergh duality, it follows from
 \ref{sec:5} that $D_\Lambda \simeq D_A\sharp \,^\sigma H$ where
 $\sigma = (S^{-2}\circ \Xi_{\int_\ell}^r)^{-1}$. Now, applying
 \ref{sec:invertible-dg-lambda-2} yields the desired conclusion.
\end{proof}

More explicitly, ${\mathrm{Ext}}_{\Lambda^e}^{n+d}(\Lambda,\Lambda^e)$ is
isomorphic to the vector space ${\mathrm{Ext}}_{A^e}^n(A,A^e)\otimes H$
endowed with the following structure of $\Lambda$-bimodule:
\begin{equation}
  \label{eq:12}
\begin{array}{rcl}
  a(d\otimes \ell)b
 & = &
       (ad(\ell_1\rightharpoonup b)) \otimes
       \ell_2\\
  h(d\otimes \ell)k
  & = &
        (h_1 \rightharpoonup d) \otimes (S^{-2}\circ
        \Xi^r_{\int_\ell})^{-1}(h_2) \ell k\,.
\end{array}
\end{equation}

\subsection{(Weak) homological determinants}
\label{sec:17}\label{sec:37}

Assume that the antipode of $H$ is invertible, that $A$ has Van den
Bergh duality in dimension $n$, and that
${\mathrm{Ext}}^n_{A^e}(A,A^e)$ is free of rank one in
${\mathrm{mod}}(A)$ (for instance, $A$ is skew Calabi-Yau). Weak
homological determinants are introduced here to express a Nakayama
automorphism for $\Lambda$.  When the homological determinant is
well-defined (like in \cite{MR2568355,MR1758250,MR3250287}), the two
notions coincide.

Fix a free generator $e_A$ of the left $A$-module ${\mathrm{Ext}}^n_{A^e}(A,A^e)$. It determines a unique
algebra homomorphism $\mu_A\colon A\to A$ such that the following identity
holds in ${\mathrm{Ext}}^n_{A^e}(A,A^e)$: $e_Aa=\mu_A(a)e_A$. When $A$
happens to be skew Calabi-Yau in dimension $n$,
then $\mu_A$ is a Nakayama automorphism of $A$.

Denote by $f_A$ the morphism of left
$A$-modules ${\mathrm{Ext}}^n_{A^e}(A,A^e)\to A$ defined by $f_A(e_A)=1$.
Since ${\mathrm{Ext}}^n_{A^e}(A,A^e)$ is an $H_{S^2}$-equivariant
$A$-bimodule (see \ref{sec:AdjunctionDelta}, part (5)), then
${\mathrm{Hom}}_A({\mathrm{Ext}}^n_{A^e}(A,A^e),A)$ is an $H_{S^{-2}}$-equivariant
$A$-bimodule in the sense of \ref{sec:equiv-acti-duals-1} (part (1)).


This setting is assumed until the end of the section.

\begin{ex}
  Consider the situation of the example in
  \ref{sec:running-example}. Keep the notation and material introduced
  there and in \ref{sec:dualA.2}. Then $A=\k[x_1,\ldots,x_n]$ is
  $n$-Calabi-Yau and $e_A$ may be taken equal to the cohomology class
  of the cochain
\[
\begin{array}{crcl}
  \varphi_A\colon & A\otimes \Lambda^nV \otimes A & \to & A^e \\
  & a\otimes x_1\wedge \cdots \wedge x_n \otimes b & \mapsto & a\otimes b\,.
\end{array}
\]
The Nakayama automorphism of $A$ corresponding to $e_A$ is
${\mathrm{Id}}_A$. Indeed, for a given $i\in \{1,\ldots,n\}$, the
coordinate cochain
\[
\varphi_i\colon  A\otimes \Lambda^{n-1}V \otimes A  \to  A^e
\]
associated with the term $1\otimes x_1\wedge \cdots
\widehat{x_i}\cdots \wedge x_n\otimes 1$ of the canonical basis of
$A\otimes \Lambda^{n-1}V \otimes A$ has coboundary
$(-1)^{n-1}\varphi_i\circ d_K$, which is given by
\[
\begin{array}{crcl}
  A\otimes \Lambda^n V \otimes A & \to & A\otimes A \\
  1\otimes x_1\wedge \cdots \wedge x_n \otimes 1
                                 & \mapsto &
                                             (-1)^{n+i}(x_i\otimes 1 -
                                             1\otimes x_i)\,;
\end{array}
\]
This is $(-1)^{n+i}(\varphi_A x_i -x_i \varphi_A)$; Therefore
$\varphi_A x_i$ and $x_i\varphi_A$ are cohomologous, and hence
$e_A x_i=x_i e_A$.

Given $X\in \mathfrak g$, then $X\rightharpoonup e_A$ is represented
by $\partial_X(\varphi_A)$ which is such that (see \eqref{eq:54})
\[
\begin{array}{rcl}
  \partial_X(\varphi_A)(1\otimes x_1\wedge \cdots \wedge x_n\otimes 1)
  & = &
        - \varphi_A(\partial_X(1\otimes x_1\wedge \cdots \wedge
        x_n\otimes 1)) \\
  & = &
        - \sum_i \partial_X(x_i)_i'\otimes \partial_X(x_i)_i''\,.
\end{array}
\]
Thus,
$X\rightharpoonup e_A = - \sum_i \partial_X(v)_i''\partial_X(v)_i'\
e_A$ and, writing $\mathrm{div}$ for the usual divergence,
\begin{equation}
  \label{eq:53}
X\rightharpoonup e_A =- \mathrm{div}(\partial_X)\,e_A \,.
\end{equation}
\end{ex}

\subsubsection{}
\label{sec:34}
The definition of the weak homological determinant is made possible by
the following technical result.
\begin{lem}
In the setting introduced previously,
  \begin{enumerate}
  \item $f_A$ is a free generator of the right
    $A$-module ${\mathrm{Hom}}_A({\mathrm{Ext}}^n_{A^e}(A,A^e),A)$,
    \item Let $\lambda,w\colon H\to A$ be the mappings such that the
      following identities hold in ${\mathrm{Ext}}^n_{A^e}(A,A^e)$ and
      ${\mathrm{Hom}}_A({\mathrm{Ext}}^n_{A^e}(A,A^e),A)$, respectively
      \[
      \left\{
        \begin{array}{rcl}
          h \rightharpoonup e_A & = & \lambda(h) e_A \\
          h \rightharpoonup f_A & = & f_A w(h)\,.
        \end{array}
      \right.
      \]
      Then, they are connected by the following relations
      \[
      \left\{
        \begin{array}{rcl}
          w(h) & = & S^{-2}(h_2) \rightharpoonup \lambda(S^{-3}(h_1))
          \\
          \lambda (h) & = & h_1\rightharpoonup (w(S^3(h_2)))\,.
        \end{array}
        \right.
        \]
        \item $\lambda$ satisfies the identity $\lambda(hk) = (h_1
          \rightharpoonup \lambda(k)) \lambda(h_2)$ in $A$.
        \item $w$ satisfies the identity $w(hk)=w(h_1)
          (S^{-2}(h_2)\rightharpoonup w(k))$ in $A$.
  \end{enumerate}
\end{lem}
\begin{proof}
  (1) follows from the definition of $f_A$.


  (2) Since ${\mathrm{Ext}}_{A^e}^n(A,A^e)$ is free of rank one in
  ${\mathrm{mod}}(A)$, the mappings $\lambda$ and $\mu$ exist and are unique.
  The third equality is due to the  following computation
  \[
  \begin{array}{rcl}
    (h\rightharpoonup f_A)(e_A)
    & = &
          (f_A w(h))(e_A) = f_A(e_A)w(h)=w(h)\\
    & \underset{~\ref{sec:equiv-acti-duals-1}\  part\ (1)}
      = &
          S^{-2}(h_2)\rightharpoonup
          f_A(S^{-3}(h_1)\rightharpoonup
          e_A) \\
    & = &
          S^{-2}(h_2)\rightharpoonup f_A(\lambda(S^{-3}(h_1))e_A) \\
    & = &
          S^{-2}(h_2) \rightharpoonup \lambda(S^{-3}(h_1))\,.
  \end{array}
  \]

  The fourth equality is due to the following computation
  \[
    h_1\rightharpoonup w(S^3(h_2)) =
    h_1\rightharpoonup (S(h_2)\rightharpoonup \lambda(h_3)) =
          \lambda(h)\,.
  \]


  (3) The equality is due to the following computation,
  \[
  \begin{array}{rcl}
  (hk)\rightharpoonup e_A &=& h\rightharpoonup (k \rightharpoonup e_A)
  = h\rightharpoonup (\lambda(k) e_A) = (h_1\rightharpoonup
                              \lambda(k)) (h_2\rightharpoonup e_A) \\
    & = & ((h_1\rightharpoonup
          \lambda(k))\lambda(h_2))e_A\,.
  \end{array}
  \]

  
  (4) The equality is due to the computation below 
  \[
  \begin{array}{rcl}
    h\rightharpoonup (k\rightharpoonup f_A) 
    & = &
          (hk)\rightharpoonup f_A = f_A w(hk) \\
    & = &
          h\rightharpoonup (f_A w(k)) \\
    & \underset{~\ref{sec:equiv-acti-duals-1}\ part\ (1)}
      = &
          (h_1\rightharpoonup f_A)  (S^{-2}(h_2)\rightharpoonup
          w(k)) \\
    & = &
          f_A w(h_1) (S^{-2}(h_2) \rightharpoonup w(k))\,.          
  \end{array}
  \]
\end{proof}

\subsubsection{}
\label{sec:35}\label{sec:39}
Apply \ref{sec:34} and call the \emph{weak homological determinant} associated
with $e_A$ the linear mapping ${\mathrm{whdet}}\colon H\to A$ which
satisfies the following identity in ${\mathrm{Hom}}_A({\mathrm{Ext}}_{A^e}^n(A,A^e),A)$
\[
h\rightharpoonup f_A =f_A{\mathrm{whdet}}(h) \,.
\]
Whence the following identity in ${\mathrm{Ext}}^n_{A^e}(A,A^e)$
\begin{equation}
  \label{eq:17}
h\rightharpoonup e_A = (h_1\rightharpoonup {\mathrm{whdet}}(S^3(h_2)))e_A\,.
\end{equation}
The weak homological
determinant is compatible with the multiplicative structure of $H$ and
$A$ in the following sense
\begin{equation}
  \label{eq:19}
{\mathrm{whdet}}(hk) = {\mathrm{whdet}}(h_1) (S^{-2}(h_2)\rightharpoonup {\mathrm{whdet}}(k))\,.
\end{equation}
In view of expressing a Nakayama automorphism of $\Lambda$, denote by
$\theta_{\mathrm{whdet}}$ the mapping $H\to \Lambda$ defined by
\begin{equation}
  \label{eq:9}
\theta_{\mathrm{whdet}}(h)={\mathrm{whdet}}(S^2(h_1)) h_2\,.
\end{equation}
According to \eqref{eq:19}, this is an algebra homomorphism from $H$
to $\Lambda$.

\begin{ex}
  Consider the situation of the example in \ref{sec:running-example}
  and keep the notation and material introduced there, in
  \ref{sec:dualA.2} and in \ref{sec:37}.  Then, the mapping
  $\lambda\colon H\to A$ of \ref{sec:34} (part (2)) is given by
  $\lambda(1)=1$ and (see \eqref{eq:53})
\[
(\forall X\in \mathfrak g)\ \ \lambda(X) = - {\mathrm{div}}(\partial_X)\,.
\]
Therefore, ${\mathrm{whdet}}\colon H\to A$ is given by ${\mathrm{whdet}}(1) = 1$ and (see
\ref{sec:34}, part (2))
\begin{equation}
  \label{eq:56}
(\forall X\in \mathfrak g)\ \ {\mathrm{whdet}}(X) = {\mathrm{div}}(X)\,.  
\end{equation}
And hence $\theta_{\mathrm{whdet}}\colon H\to \Lambda$ is given by
$\theta_{\mathrm{whdet}}(1) = 1$ and
\begin{equation}
  \label{eq:55}
\theta_{\mathrm{whdet}}(X) = {\mathrm{div}}(\partial_X)+X\,.  
\end{equation}
\end{ex}

\subsubsection{}
\label{sec:42}
Distinct choices for $e_A\in {\mathrm{Ext}}^n_{A^e}(A,A^e)$ may yield
distinct weak homological determinants related to each other as
follows (because ${\mathrm{Hom}}_A(\mathrm{Ext}^n_{A^e}(A,A^e),A)$ is an
$H_{S^{-2}}$-equivariant $A$-bimodule, see
\ref{sec:equiv-acti-duals-1}).

\begin{lem}
  Let $e_A'$ be free generator of the left $A$-module ${\mathrm{Ext}}^n_{A^e}(A,A^e)$ and let ${\mathrm{whdet}}'$ be the associated weak
  homological determinant. Let $a_0\in A^\times$ be such that
  $e_A'=a_0e_A$. Then,
  ${\mathrm{whdet}}'(h) = a_0{\mathrm{whdet}}(h_1) (S^{-2}(h_2)\rightharpoonup
  a_0^{-1})$ for all $h\in H$.
\end{lem}

\subsubsection{}
\label{sec:36}

If $A$ is connected graded and $e_A$ is chosen to be homogeneous, then
$\mathrm{whdet}$ is equal to the homological determinant already
developed in \cite{MR1758250,MR2568355,MR3250287}. Besides, distinct
choices for $e_A$ (with the homogeneity requirement) yield the same
homological determinants and the same Nakayama automorphisms of
$A$. See \cite[Definition 3.7 and Remark 3.8]{MR3250287} for details,
keeping in mind that $R^d\Gamma_{\mathfrak m_A}(A)^*$ there is
${\mathrm{Hom}}_A(\mathrm{Ext}^n_{A^e}(A,A^e),A)$ here.

In general, $\k\cdot e_A$ is an $H$-submodule of
${\mathrm{Ext}}^n_{A^e}(A,A^e)$ if and only if $\lambda$ and
$\mathrm{whdet}$ take values in $\k$. In this case, and following the
spirit of \cite[Definition 3.7]{MR3250287}, the action of $H$ on $A$
is said to \emph{have a homological determinant} and $\mathrm{whdet}$
is denoted by $\mathrm{hdet}$. In particular,
$\mathrm{hdet}\colon H \to \k$ is an algebra homomorphism (see
\eqref{eq:19}), $\mathrm{hdet} \circ S^2 = \mathrm{hdet}$, the
identity $h\rightharpoonup e_A= {\mathrm{hdet}}(S(h))e_A$ holds in
$\mathrm{Ext}^n_{A^e}(A,A^e)$, and
$\theta_{\mathrm{hdet}}=\Xi^{\ell}_{\mathrm{hdet}}$.

\subsection{Nakayama automorphisms of $A\sharp H$}
\label{sec:19}

This section gives necessary and sufficient conditions for $\Lambda=A\sharp H$ to be
skew Calabi-Yau. In such a case, it gives a Nakayama automorphism for $\Lambda$.
Recall that the antipode is invertible as soon as $H$ has Van den
Bergh duality (\ref{sec:25}).

\subsubsection{}
\label{sec:19-1}

The following result gives a sufficient condition for $\Lambda$ to be
skew Calabi-Yau. Note that it was first established in \cite[Theorem
4.1]{MR3250287} assuming that $A$ is connected graded and $H$ is
finite dimensional. See also \cite{MR3575984} for similar conclusions
about Hopf Galois objects of skew Calabi-Yau Hopf algebras.

\begin{prop}
  Let $A$ be an $H$-module algebra where $H$ is a Hopf algebra. Assume
  that $A$ and $H$ are skew Calabi-Yau in dimension $n$ and $d$,
  respectively. Then, $\Lambda=A\sharp H$ is skew Calabi-Yau in
  dimension $n+d$ and admits a Nakayama automorphism $\mu_{\Lambda}$
  given by
\begin{equation}
  \label{eq:18}
\mu_{\Lambda}=\mu_A\sharp (\theta_{\mathrm{whdet}}\circ \mu_H)
\end{equation}
where $\mu_A$ is the Nakayama
automorphism of $A$  and ${\mathrm{whdet}}\colon H\to A$ is the  weak homological
determinant associated to any generator of ${\mathrm{Ext}}^n_{A^e}(A,A^e)$
in ${\mathrm{mod}}(A)$,
and  $\mu_H=S^{-2}\circ \Xi_{\int_{\ell}}^r$.

If, moreover, the action of $H$ on $A$ has a homological determinant
(\ref{sec:36}), 
then
\begin{equation}
  \label{eq:20}
  \mu_{A\sharp H} = \mu_A\sharp ( \Xi_{\mathrm{hdet}}^\ell\circ \mu_H)\,.
\end{equation}
\end{prop}
\begin{proof}
First, note (\ref{sec:18}) that $\Lambda$ has Van den Bergh duality
in dimension $n+d$ and ${\mathrm{Ext}}_{\Lambda^e}^{n+d}(\Lambda,\Lambda^e)\simeq {\mathrm{Ext}}_{A^e}^n(A,A^e)\sharp\,^{\mu_H^{-1}}H$. Let $e_A\in {\mathrm{Ext}}_{A^e}^n(A,A^e)$ be a free  generator in ${\mathrm{mod}}(A)$ satisfying the
identity $e_Aa=\mu_A(a)e_A$. 
Since ${\mathrm{Ext}}^{n+d}_{\Lambda^e}(\Lambda,\Lambda^e)\simeq {\mathrm{Ext}}^n_{A^e}(A,A^e)\sharp \,^{\mu_H^{-1}}H$ in ${\mathrm{mod}}(\Lambda^{\mathrm{op}})$, then $e_A\otimes 1$ is a free generator of
${\mathrm{Ext}}^{n+d}_{\Lambda^e}(\Lambda,\Lambda^e)$ in ${\mathrm{mod}}(\Lambda^{\mathrm{op}})$. Since ${\mathrm{Ext}}_{\Lambda^e}^{n+d}(\Lambda,\Lambda^e)\simeq H^{\mu_H}\sharp {\mathrm{Ext}}^n_{A^e}(A,A^e)$ in ${\mathrm{mod}}(\Lambda)$ (see \ref{sec:45}),
then $e_A\otimes 1$ is also a free generator of ${\mathrm{Ext}}^{n+d}_{\Lambda^e}(\Lambda,\Lambda^e)$ in ${\mathrm{mod}}(\Lambda)$.
Hence, in order to prove the result, it
suffices to show that the mapping $\mu_{\Lambda}\colon \Lambda\to
\Lambda$  satisfies the
following identity in ${\mathrm{Ext}}_{A^e}^n(A,A^e)\sharp\,^{\mu_H^{-1}}H$
\[
(e_A\otimes 1) (ah) = \mu_{\Lambda}(ah) (e_A\otimes 1)\,.
\]
This is done in the following computation
\[
\begin{array}{rcl}
  (e_A\otimes 1) (ah)
  & = &
        (e_A a\otimes h) \\
  & = &
        \mu_A(a)(e_A\otimes h) \\
  & \underset{~\eqref{eq:44}}=
      &
        \mu_A(a) \mu_H(h_2) (S^{-3}(h_1)\rightharpoonup e_A\otimes 1) \\
  & 
    = &
        \mu_A(a)\mu_H(h)_2 (S^{-1}(\mu_H(h)_1)\rightharpoonup e_A\otimes
        1) \\
  & \underset{~\eqref{eq:17}}{=} &
                                   \mu_{\Lambda}(ah) (e_A\otimes 1)\,.
\end{array}
\]
The description of $\mu_{A\sharp H}$ when the action of $H$ on $A$ has
a homological determinant follows from the discussion in \ref{sec:36}.
\end{proof}

\begin{ex}
Let $A=\k[x_1,\ldots,x_n]$ and let $\mathfrak g$ be a $d$-dimensional
Lie algebra. Let $\mathfrak g\to {\mathrm{Der}}_\k(A)$ be a Lie algebra
homomorphism inducing an action of $H=\mathcal U(\mathfrak g)$ on
$A$. Denote by $\partial_X\colon A\to A$ the derivation associated
with $X$ for every $X\in \mathfrak g$.  Then, (see \cite[Corollary
2.2]{MR1762922} or \eqref{eq:50}) $\mu_H$ is given by
\[
\mu_H(X) = X + {\mathrm{Tr}}({\mathrm{ad}}_X)\,.
\]
And $\Lambda=A\sharp H$ has a Nakayama automorphism
$\mu_\Lambda\colon \Lambda \to \Lambda$ such that, for all $a\in A$
and $X\in \mathfrak g$ (see \ref{sec:37} and \eqref{eq:55})
\begin{equation}
  \label{eq:57}
  \begin{array}{rcl}
    \mu_\Lambda(a) & = & a \\
    \mu_\Lambda(X) & = & X + {\mathrm{div}}(\partial_X) + {\mathrm{Tr}}({\mathrm{ad}}_X)\,.
  \end{array}
\end{equation}
Note that $\mu_\Lambda(X)$ need not lie in $H$.
\end{ex}

\subsubsection{}
\label{sec:19-2}
The following result is a partial converse to the implication proved
in \ref{sec:19-1}.
\begin{prop}
  Let $H$ be a Hopf algebra with invertible antipode. Let $A$ be an
  $H$-module algebra. Assume that $A$ and $H$ are homologically
  smooth and that $A\sharp H$ is skew Calabi-Yau. Then, $H$ is
  skew Calabi-Yau. If, moreover, the action of $H$ on $A$ has a
  homological determinant, then $A$ is skew Calabi-Yau as well.
\end{prop}
\begin{proof}
  Note that $A$ and $H$ have Van den Bergh duality, say in dimension
  $n$ and $d$, respectively, and hence the dimension of
  $\Lambda$ (as a skew Calabi-Yau algebra) is $n+d$ (see
  \ref{sec:18}). In particular,
  $H$ is skew Calabi-Yau (see \ref{sec:25}). Let $\mu_H=S^{-2}\circ
  \Xi_{\int_\ell}^r$.
  Then, ${\mathrm{Ext}}^{n+d}_{\Lambda^e}(\Lambda,\Lambda^e)\simeq {\mathrm{Ext}}^n_{A^e}(A,A^e)\sharp\,^{\mu_H^{-1}}H$ in ${\mathrm{mod}}(\Lambda^e)$ (see \ref{sec:18}).

  Applying $-\underset\Lambda \otimes A\colon {\mathrm{mod}}(\Lambda)\to
  {\mathrm{mod}}(A)$ to the free of rank one left $\Lambda$-module ${\mathrm{Ext}}^{n+d}_{\Lambda^e}(\Lambda,\Lambda^e)$ yields that ${\mathrm{Ext}}^n_{A^e}(A,A^e)$ is free of rank one in ${\mathrm{mod}}(A)$ (see
  \ref{sec:18}). Therefore, a weak homological determinant is
  defined (in the sense of \ref{sec:34}).

  Assume that a homological determinant ${\mathrm{hdet}}\colon H\to \k$
  exists (in the sense of \ref{sec:36}), say associated to a free
  generator $e_A\in {\mathrm{Ext}}^n_{A^e}(A,A^e)$ in ${\mathrm{mod}}(A)$. In particular,
  $\theta_{\mathrm{hdet}}=\Xi_{\mathrm{hdet}}^\ell$. Let $\mu_A\colon A\to A$ be
  the algebra homomorphism such that the identity $e_Aa=\mu_A(a)e_A$
  holds in ${\mathrm{Ext}}^n_{A^e}(A,A^e)$.
  In order to prove that $A$ is skew Calabi-Yau, it is convenient to
  prove that $\mu_A$ is an automorphism of $A$. For this purpose, the
  following arguments first prove that $\mu_A\sharp (\Xi_{\mathrm{hdet}}^\ell\circ \mu_H)$ is an automorphism of $\Lambda$.
  Part of the considerations made
  in the proof of \ref{sec:19-1} are valid in the present situation. In
  particular, $e_A\otimes 1$ is a free generator of ${\mathrm{Ext}}^n_{A^e}(A,A^e) \sharp\, ^{\mu_H}H$ in ${\mathrm{mod}}(\Lambda)$, and the following identity holds true (for
  $\lambda\in \Lambda$)
  \[
  (e_A\otimes 1) \lambda = (\mu_A\sharp (\Xi_{\mathrm{hdet}}^\ell\circ
  \mu_H))(\lambda) (e_A\otimes 1)\,.
  \]
  Since $\Lambda$ is skew Calabi-Yau, the lemma in
  \ref{sec:dual-cond-dg-2} applies to $D={\mathrm{Ext}}_{\Lambda^e}^{n+d}(\Lambda,\Lambda^e)\simeq {\mathrm{Ext}}^n_{A^e}(A,A^e)\sharp\, ^{\mu_H^{-1}}H$. It yields that
  $\mu_A\sharp (\Xi_{\mathrm{hdet}}^\ell\circ \mu_H)\colon \Lambda \to
  \Lambda$ is an automorphism. In particular
  \begin{itemize}
  \item $\mu_A$ is an injective mapping,
  \item composing with $1\otimes \epsilon \colon \Lambda \to A$ yields
    that $\mu_A$ is a surjective mapping.
  \end{itemize}
  Thus, $\mu_A\in {\mathrm{Aut}}_{\k-{\mathrm{alg}}}(A)$. Since the mapping
  \[
  \begin{array}{rcl}
    A^{\mu_A}& \to & {\mathrm{Ext}}^n_{A^e}(A,A^e) \\
    a & \mapsto  & ae_A
  \end{array}
  \]
  is an isomorphism in ${\mathrm{mod}}(A^e)$, then $A$ is skew Calabi-Yau.
\end{proof}

\subsection{When is $A\sharp H$ Calabi-Yau?}
\label{sec:calabi-yau-duality}

The algebra $A$ is Calabi-Yau if and only
if it is skew Calabi-Yau and any Nakayama automorphism is inner
(equivalently, the identity map of $A$ is a Nakayama automorphism). Using
\ref{sec:19}, this section gives necessary and sufficient
conditions for  $\Lambda$ to be 
Calabi-Yau. Unfortunately, the techniques used here require
restrictive hypotheses on $A$, that is, $A$ is augmented or connected graded.

\subsubsection{}
\label{sec:22}

Recall that the augmented ideal of an augmented $H$-module algebra is
always assumed to be an $H$-submodule.
\begin{lem}
  Assume that $A$ is an augmented $H$-module algebra.  Denote by
  $p\colon A\to \k$ the augmentation. Keep the hypotheses made in the
  proposition in \ref{sec:19-1} as well as the notation introduced
  there. Assume, moreover, that $\Lambda$ is Calabi-Yau.  Then,
\begin{enumerate}
\item $p({\mathrm{whdet}}(h_1))\int_\ell(h_2) = \epsilon(h)$ for
  all $h\in H$;
\item if the action of $H$ admits a homological determinant,
  then $\mathrm{hdet}(h_1)\int_\ell(h_2) = \epsilon(h)$ for all
  $h\in H$;
\item if $A$ is connected graded, then
  \begin{equation}
    \label{eq:14}
    (\exists h_A\in H^\times)\ (\forall a\in A)\ \
    \mu_A(a)=h_A\rightharpoonup a \underset{\text{in $\Lambda$}} = h_Aah_A^{-1}\,.
  \end{equation}
\end{enumerate}
\end{lem}
\begin{proof}
(1) The Nakayama automorphism ($\mu_{\Lambda}$) of $\Lambda$  given in
\ref{sec:19-1} is inner because $\Lambda$ is Calabi-Yau. Hence,
  there exists $\lambda\in \Lambda^\times$  such that the
  following identities hold in $\Lambda$
  \begin{equation}
    \label{eq:13}
    \left\{
      \begin{array}{l}
        \mu_{\Lambda}(h)=\theta_{\mathrm{whdet}}\circ \mu_H(h)=\lambda h\lambda^{-1}\\
        \mu_{\Lambda}(a)=\mu_A(a)=\lambda a\lambda^{-1}\,.
        \end{array}\right.
  \end{equation}
  Denote by $\mathfrak M$ the kernel of $p$. Since
  $\mathfrak M\otimes H$ is a two-sided ideal of $\Lambda$, there
  exist $k,k'\in H$ such that $\lambda\in k + \mathfrak M\otimes H$
  and $\lambda^{-1}\in k'+\mathfrak M\otimes H$. Since $\lambda$ is
  invertible, then $k$ is invertible in $H$ and $k'$ is its
  inverse. Let $h\in H$; Then,
  $\lambda h\lambda^{-1}\in khk^{-1}+\mathfrak M\otimes H$;
  Consequently
  $(p\otimes \epsilon) (\lambda h \lambda^{-1})=\epsilon(h)$; Now,
  \eqref{eq:13} entails that
  $(p\otimes \epsilon) (\lambda h \lambda^{-1})= (p \otimes \epsilon)
  \circ \theta_{\mathrm{whdet}}\circ \mu_H(h)$;
  Note that
  $(p \otimes \epsilon) \circ \theta_{\mathrm{whdet}} = p \circ {\mathrm{whdet}}
  \circ S^2$
  (see \eqref{eq:9}); Since $\mu_H = S^{-2}\circ \Xi^r_{\int_\ell}$,
  then $\epsilon(h) = p({\mathrm{whdet}}(h_1))\,\int_\ell(h_2)$.

  
  (2) The additional hypothesis means that ${\mathrm{whdet}}$ takes its
  values in $\k$, and (by definition) ${\mathrm{hdet}}={\mathrm{whdet}}$. The
  conclusion therefore follows from (1).


   (3) Assume that $A$ is connected graded. Then, there exists
   $\ell\in \mathbb Z$ such that
   ${\mathrm{Ext}}^n_{A^e}(A,A^e)\simeq A^{\mu_A}(\ell)$ as graded
   $A$-bimodules. In particular, $\mu_A\colon A\to A$ is
   homogeneous. Let $a\in A$ be homogeneous. Since $A$ is connected
   graded, then $\Lambda^\times = H^\times$, and hence
   $\mu_A(a)=kak^{-1}=(k_1\rightharpoonup a)k_2k^{-1}$ in $\Lambda$;
   Applying ${\mathrm{Id}}_A\otimes \epsilon \colon \Lambda\to A$ yields
   the equality $\mu_A(a) = \epsilon(k^{-1})k\rightharpoonup a$. Thus
   \eqref{eq:14} holds true taking $h_A=\epsilon(k^{-1})k$.
\end{proof}

Keep the setting of part (3) in the previous result. Since $\mu_A$ is
homogeneous, then $\mu_{\Lambda}(A_n\otimes H)=A_n\otimes
H$ for every $n\in \mathbb N$. Therefore, in the previous proof, one
may assume  that $\lambda=h_A$ when $A$ is connected graded. 

\subsubsection{}
\label{sec:6}

The following result determines when $A\sharp H$ is Calabi-Yau
assuming that $H$ is Calabi-Yau and $A$ is connected graded.
\begin{thm}
  Let $H$ be a Calabi-Yau Hopf algebra. 
   Let $A$ be a connected graded  $H$-module algebra.  Let $h_0\in
  H^\times$ be such that $S^{-2}$ is the inner automorphism of $h_0$
  (see \ref{sec:char-calabi-yau}). Then, $\Lambda=A\sharp H$ is
  Calabi-Yau if and only if the following conditions hold
  \begin{enumerate}[(a)]
    \item $A$ is skew Calabi-Yau,
    \item  ${\mathrm{hdet}}=\epsilon$,
    \item $(\exists k_A\in Z(H^\times))
          (\forall a\in A)\ \ 
          \mu_A(a)=(h_0k_A)\rightharpoonup a
          \underset{\text{in $\Lambda$}} = (h_0k_A) a(h_0k_A)^{-1}$.
  \end{enumerate}
\end{thm}
\begin{proof}
  In order to prove the equivalence it may be assumed that $A$ is
  skew Calabi-Yau in the graded sense (see \ref{sec:36},
  \ref{sec:19-1}, \ref{sec:19-2} and \ref{sec:dual-cond-dg-3}). Note
  that $\int_\ell=\epsilon$ due to \ref{sec:char-calabi-yau} applied
  to $H^{\mathrm{op}}$.

Assume that $\Lambda$ is Calabi-Yau. Following  \ref{sec:22}  and the
the final remark made there, ${\mathrm{hdet}}=\epsilon$ and there exists
$h_A\in H^{\times}$ such 
that the following identities hold
\[
\begin{array}{rcll}
  \Xi_{\mathrm{hdet}}^\ell\circ \mu_H(h)  & = & h_Ahh_A^{-1}  & \text{in $H$}\\
  \mu_A(a)  &  = & h_A\rightharpoonup a & \text{in $A$} \\
  & = & h_Aah_A^{-1} & \text{in $\Lambda$}\,.
\end{array}
\]
Since ${\mathrm{hdet}}=\epsilon$ and $\mu_H$ is given by $\bullet \mapsto
h_0 \bullet h_0^{-1}$, the first identity implies that
$h_0^{-1}h_A\in Z(H^\times)$. Set $k_A=h_0^{-1}h_A$. Then, $k_A\in
Z(H^{\times})$ and $k_A$ satisfies the identities $\mu_A(a) =
(h_0k_A)\rightharpoonup a$ in $A$ and $\mu_A(a) = (h_0k_A) a
(h_0k_A)^{-1}$ in $\Lambda$.

  Conversely, assume that ${\mathrm{hdet}}=\epsilon$ and
  that there exists $k_A\in Z(H^\times)$ such that the identities
  $\mu_A(a)=(h_0k_A)\rightharpoonup a$ and $\mu_A(a) =
  (h_0k_A)a(h_0k_A)^{-1}$ hold in $A$ and $\Lambda$, respectively. Then, the Nakayama
  automorphism of $\Lambda$ given in  \ref{sec:19-1} is inner (and
  associated with $h_0k_A\in \Lambda^\times$). Therefore, $\Lambda$ is
  Calabi-Yau.
\end{proof}

\subsubsection{}
\label{sec:28}
The previous result simplifies as follows when $A$ is assumed to be
Calabi-Yau (see \cite{MR2809906,MR2905560,MR2813562}).
\begin{cor}
  Let $H$ be a Calabi-Yau Hopf algebra. Let $A$ be a connected graded
  $H$-module algebra which is moreover Calabi-Yau. The following
  assertions are equivalent
  \begin{enumerate}[(i)]
  \item $A\sharp H$ is Calabi-Yau,
  \item ${\mathrm{hdet}}=\epsilon$.
  \end{enumerate}
\end{cor}
\begin{proof}
  Since $A$ is connected graded and Calabi-Yau, then $\mu_A={\mathrm{Id}}_A$. The conclusion therefore follows from \ref{sec:6}.
\end{proof}

\subsubsection{}
\label{sec:49}

Combining the theorem in \ref{sec:6} and \cite[Theorem 1.1]{MR3188338}
yields a proof of \cite[Conjecture 6.4]{MR3250287} for connected
graded Artin-Schelter regular algebras.
\begin{cor}
  Let $A$ be an augmented skew Calabi-Yau algebra with Nakayama
  automorphism $\mu_A$. Then,
  \begin{enumerate}
  \item $p\circ {\mathrm{whdet}}(\mu_A)=1$;
  \item if $A$ is connected graded and $\mu_A$ is also graded, then
    ${\mathrm{hdet}}(\mu_A)=1$.
  \end{enumerate}
\end{cor}
\begin{proof}
  Let $H=\k\mathbb Z$. It is Calabi-Yau and $\mu_A$ determines a
  structure of $H$-module augmented algebra on $A$. According to
  \cite[Theorem 1.1]{MR3188338}, $A\sharp H$ is Calabi-Yau.
  Therefore, (1) follows from \ref{sec:22} (note that
  $\int_\ell=\epsilon$ according to \ref{sec:char-calabi-yau}),
  whereas (2) follows from \ref{sec:6}.
\end{proof}

\section{Example: actions of $\mathcal U_q(\mathfrak{sl}_2)$ on the quantum
  plane}
\label{sec:exampl-acti-u_qm}

Let $q\in \mathbb C^\times$. Assume that $q$ is not a root of
unity. Let $A$ be the quantum plane $\mathbb C_q[x,y]$. Let $H$ be the
quantum enveloping algebra $\mathcal U_q(\mathfrak{sl}_2)$. Assume
that $\mathbb C_q[x,y]$ is endowed with a structure of
$\mathcal U_q(\mathfrak{sl}_2)$-module algebra. This section applies
Section~\ref{sec:19} to the computation of a Nakayama automorphism of
$\mathbb C_q[x,y]\sharp \mathcal U_q(\mathfrak{sl}_2)$. For this
purpose, Section~\ref{sec:remind-u_qm-mathbb} recalls known material
on $A$ and $H$, Section~\ref{sec:acti-u_qm-mathbb} endows the Koszul
resolution of $A$ as a bimodule over itself with an action of $H$ so
that it lies in $\mathcal C(\Delta_0)$, Section~\ref{sec:acti-u_qm-rm}
computes the resulting action of $H$ on $\mathrm{Ext}^2_{A^e}(A,A^e)$
in terms of the mapping $\lambda$ introduced in \ref{sec:34},
Section~\ref{sec:comp-weak-homol} computes the weak homological
determinant of the action of $H$ on $A$, and
Section~\ref{sec:descr-nakay-autom} computes the Nakayama
automorphism. The computations are made according to the
classification of the actions of $H$ on $A$ made in \cite{MR2789302}.

\subsection{Reminder on $\mathcal U_q(\mathfrak{sl}_2)$ and on $\mathbb
  C_q[x,y]$}
\label{sec:remind-u_qm-mathbb}

As a $\mathbb C$-algebra,  $\mathcal U_q(\mathfrak{sl}_2)$ is given by generators
$E,F,K,K^{-1}$ and relations
\[
KK^{-1}=1 = K^{-1}K,\ KEK^{-1}=q^2E,\ KFK^{-1}=q^{-2}F,\ [E,F] =
\frac{K-K^{-1}}{q-q^{-1}}\,.
\]
The comultiplication, the counit and the antipode of
$\mathcal U_q(\mathfrak{sl}_2)$ are given by
\[
\begin{array}{c}
\Delta(K) = K \otimes K,\ \Delta(E) = E \otimes K + 1 \otimes E,\
\Delta(F) = F \otimes 1 + K^{-1} \otimes F, \\
\epsilon(K) = 1,\ \epsilon(E) = 0,\ \epsilon(F) = 0, \\
S(K) = K^{-1},\ S(E) = -EK^{-1},\ S(F) = -KF\,.
\end{array}
\]
$H$ is $2$-Calabi-Yau (see \cite[Theorem 3.3.2]{MR2054387}), and hence
(see \ref{sec:char-calabi-yau} and its dual version)
\begin{equation}
  \label{eq:58}
  \int_\ell=\int_r = \epsilon\,.
\end{equation}
The quantum plane is the Koszul $\mathbb C$-algebra
$A = \mathbb C_q[x,y] = \mathbb C \langle x,y\ |\ yx =
qxy\rangle$.
Let $V=\mathbb C\cdot x\oplus \mathbb C\cdot y$. Then,
$\mathbb C_q[x,y]$ admits the following Koszul resolution as a
bimodule over itself (\cite[Proposition 4.1]{MR1252939})
\begin{equation}
  \label{eq:60}
  P\colon \ 0 \to A\otimes \Lambda^2 V \otimes A \xrightarrow{d} A \otimes V
  \otimes A \xrightarrow{d} A\otimes A \to 0
\end{equation}
where, using ``$|$'' instead of ``$\otimes$'' for the ease of
readability,
\[
  \begin{array}{rcl}
  d(1 | x | 1)
  & = &
        x | 1 -1 |x \\
  d(1 | y | 1)
  & = &
        y | 1 -1 |y \\
  d(1 | x \wedge y | 1)
  & = &
        x | y | 1 - q^{-1} | y | x - q^{-1} y | x | 1 + 1 | x | y\,.
  \end{array}
  \]
  The algebra $\mathbb C_q[x,y]$ is (connected $\mathbb N$-graded)
  Artin-Schelter regular (see \cite[p. 172]{MR917738}), and hence
  skew Calabi-Yau (see \cite[Lemma 1.2]{MR3250287}) with a Nakayama
  automorphism as follows (see \cite[p. 76]{MR3341818}). Let
  $\varphi_A \in {\mathrm{Hom}}_{A^e}(A \otimes \Lambda^2 V \otimes A, A^e)$ be
  the following $2$-cocycle whose cohomology class in
  $\mathrm{Ext}^2_{A^e}(A,A^e)$ is denoted by $e_A$
  \begin{equation}
    \label{eq:61}
\begin{array}{crcl}
  \varphi_A \colon 
  & A \otimes \Lambda^2 V \otimes A & \to & A\otimes A \\
  & 1 | x\wedge y | 1 & \mapsto & 1 \otimes 1\,.
\end{array}
\end{equation}
Then,
  $e_A$ is a free generator of the left $A$-module ${\mathrm{Ext}}_{A^e}^2(A,A^e)$.
By considering the coboundaries of the two $1$-cocycles $A \otimes V
\otimes A \to A \otimes A$ defined by
\[
\begin{array}{lll}
  \left\{
  \begin{array}{ccc}
    1 | x | 1 & \to & 1 \otimes 1 \\
    1 | y | 1 & \to & 0
  \end{array}\right.
              &
                \text{and}
                \left\{
                \begin{array}{ccc}
                  1 | x | 1 & \to & 0 \\
                  1 | y | 1 & \to & 1 \otimes 1,
                \end{array}\right.
\end{array}
\]
respectively, it appears that $e_Ax = q^{-1}x e_A$ and $e_Ay = qye_A$.
Therefore, the Nakayama automorphism $\mu_A$ of $A$ corresponding to
$e_A$ is given by
\begin{equation}
  \label{eq:59}
  \mu_A(x) = q^{-1}x \ \text{and}\ \mu_A(y) = qy\,.
\end{equation}

The actions of $\mathcal U_q(\mathfrak{sl}_2)$ on $\mathbb C_q[x,y]$
are classified into six families with parameters (see \cite[Table
1]{MR2789302} and Table~\ref{tab:1} below). The figures in the first
column of Table~\ref{tab:1} serve as an internal reference for the
corresponding action.  Except for the first and last case, the action
does not preserve the augmentation ideal of $\mathbb C_q[x,y]$.
Therefore $\mathbb C_q[x,y]\sharp \mathcal U_q(\mathfrak{sl}_2)$ may
be non augmented and non connected graded.

\begin{table}[!ht]
  \tiny 
\begin{tabularx}{\textwidth}{>{\hsize=.3\hsize}X>{\hsize=1.1\hsize}X>{\hsize=\hsize}X>{\hsize=.6\hsize}X}
\hline
case & actions of $E,F,K$ on $x,y$ & & parameters \\
\hline
\\
0 &
$\begin{array}{l}
  K \rightharpoonup x = \pm x \\
  E \rightharpoonup x = 0 \\
  F \rightharpoonup x = 0 
\end{array}$
&
$\begin{array}{l}
    K \rightharpoonup y = \pm y \\
    E \rightharpoonup y = 0 \\
    F \rightharpoonup y = 0
\end{array}$
& $\emptyset$ \\
\hline
\\
1 &
$\begin{array}{l}
  K \rightharpoonup x = qx \\
  E \rightharpoonup x = 0 \\
  F \rightharpoonup x = b_0^{-1}xy 
\end{array}$
&
$\begin{array}{l}
    K \rightharpoonup y = q^{-2}y \\
    E \rightharpoonup y = b_0 \\
    F \rightharpoonup y = -qb_0^{-1}y^2
\end{array}$
& $b_0\in \mathbb C^\times$ \\
\hline
\\
2 &
$\begin{array}{l}
  K \rightharpoonup x = q^2x \\
  E \rightharpoonup x = -qc_0^{-1}x^2 \\
  F \rightharpoonup x = c_0 
\end{array}$
&
$\begin{array}{l}
    K \rightharpoonup y = q^{-1}y \\
    E \rightharpoonup y = c_0^{-1}xy \\
    F \rightharpoonup y = 0
\end{array}$
& $c_0\in \mathbb C^\times$ \\
\hline \\
3 &
$\begin{array}{l}
  K \rightharpoonup x = q^{-2}x \\
  E \rightharpoonup x = a_0 \\
  F \rightharpoonup x = -qa_0^{-1}x^2+ty^4 
\end{array}$
&
$\begin{array}{l}
    K \rightharpoonup y = q^{-1}y \\
    E \rightharpoonup y = 0 \\
    F \rightharpoonup y = -qa_0^{-1}xy+sy^3
\end{array}$
& $
\begin{array}{ll}
  a_0\in \mathbb C^\times\\
  s,t\in \mathbb C
  \end{array}$ \\
\hline \\
4 &
$\begin{array}{l}
   K \rightharpoonup x = qx  \\
  E \rightharpoonup x = -qd_0^{-1}xy+sx^3 \\
  F \rightharpoonup x = 0 
 \end{array}$
 &
 $\begin{array}{l}
    K \rightharpoonup y = q^2y \\
    E \rightharpoonup y = -qd_0^{-1}y^2+tx^4 \\
    F \rightharpoonup y = d_0
\end{array}$
&
$\begin{array}{ll}
  d_0\in \mathbb C^\times\\
  s,t\in \mathbb C
  \end{array}$ \\
\hline \\
5 &
$\begin{array}{l}
  K \rightharpoonup x = qx \\
  E \rightharpoonup x = 0 \\
  F \rightharpoonup x = \tau^{-1}y 
\end{array}$
&
$\begin{array}{l}
    K \rightharpoonup y = q^{-1}y \\
    E \rightharpoonup y = \tau x \\
    F \rightharpoonup y = 0
\end{array}$
& $\tau \in \mathbb C^\times$ \\
\hline
\end{tabularx}
  \caption{Classification of the actions of
    $\mathcal U_q(\mathfrak{sl}_2)$ on $\mathbb C_q[x,y]$}
  \label{tab:1}
\end{table}

\subsection{Action of $\mathcal U_q(\mathfrak{sl}_2)$ on the Koszul resolution
  $\mathbb C_q[x,y]$}
\label{sec:acti-u_qm-mathbb}

\begin{lem}
  There exists an action of $\mathcal U_q(\mathfrak{sl}_2)$ on the Koszul
  resolution $P$ of $\mathbb C_q[x,y]$ such that
  \begin{itemize}
  \item $P$ is complex of $\mathcal U_q(\mathfrak{sl}_2)$-equivariant
    $\mathbb C_q[x,y]$-bimodules (or, $P\in \mathcal C(\Delta_0)$),
\item the action on $\mathbb C_q[x,y]\otimes \mathbb C_q[x,y]$ is the
  natural one ($h\rightharpoonup (a\otimes b) = h_1\rightharpoonup a
  \otimes h_2 \rightharpoonup b$),
\item if $t_1,t_2\in \mathbb C^\times$ denote the scalars such that
  $K\rightharpoonup x = t_1x$ and $K\rightharpoonup y = t_2
  y$ (see Table~\ref{tab:1}), then
  \[
  \begin{array}{rclcrcl}
    K \rightharpoonup 1|x|1 & = & t_1 |x|1
    &&
       K^{-1} \rightharpoonup 1|x|1 & = & t_1^{-1} |x|1 \\
    K \rightharpoonup 1|y|1 & = & t_2 |y|1
     &&
        K^{-1} \rightharpoonup 1|y|1 & = & t_2^{-1} |y|1 \\
    K \rightharpoonup 1 |x\wedge y |1 & = & t_1t_2 | x\wedge
                                            y |1
     &&
    K^{-1} \rightharpoonup 1 |x\wedge y |1 & = & (t_1t_2)^{-1}|
                                                 x\wedge y |1 \\
  \end{array}
  \]
\item the actions of $E$ and $F$ on $1|x|1$, $1|y|1$ and $1|x \wedge
  y|1$ are such as in Table~\ref{tab:2}.
  \end{itemize}
\begin{table}[!ht]
  \tiny
\begin{tabularx}{\textwidth}{>{\hsize=.2\hsize}X>{\hsize=1.8\hsize}X}
  \hline
  \\
  Action of Table~\ref{tab:1} & Actions of
  $E,F$ on the generators $1|x|1$, $1|y|1$ and $1| x\wedge y|1$,
  ($1|*|1$ is any such generator). \\
  \hline
  \\
  Case 0 &
  $
  \begin{array}{rclcrcl}
    E \rightharpoonup 1|*|1
    & = &
          0
    &&
      F \rightharpoonup 1|*|1
    & = &
          0
  \end{array}
  $
  \\
  \hline
  \\
  Case 1 &
  $
  \begin{array}{rcl}
    E \rightharpoonup 1|*|1
    & = &
          0
    \\
      F \rightharpoonup 1|x|1
    & = &
          b_0^{-1}(x|y|1+1|x|y)
    \\
        F \rightharpoonup 1|y|1
    & = &
          -qb_0^{-1}(y|y|1 + 1|y|y)
    \\
       F \rightharpoonup 1|x \wedge y |1
    & = &
          -b_0^{-1}qy|x \wedge y | 1
  \end{array}
  $
 \\
  \hline
  \\
  Case 2 &
  $
  \begin{array}{rcl}
    E \rightharpoonup 1|x|1
    & = &
          -qc_0^{-1}(x|x|1 + 1|x|x)
    \\
    E \rightharpoonup 1|y|1
    & = &
          c_0^{-1}(x|y|1+1|x|y)
    \\
    E \rightharpoonup 1|x \wedge y |1
    & = &
          -c_0^{-1}q|x\wedge y| x
    \\
       F \rightharpoonup 1|* |1
    & = &
          0
  \end{array}
  $
 \\
  \hline
  \\
  Case 3 &
  $
  \begin{array}{rcl}
    E \rightharpoonup 1|*|1
    & = &
          0
    \\
    F \rightharpoonup 1|x|1
    & = &
          -q a_0^{-1} (x|x|1 + 1|x|x)
          + t(y^3|y|1 + y^2|y|y + y|y|y^2 + 1|y|y^3)
    \\
    F \rightharpoonup 1|y|1
    & = &
          -q a_0^{-1}(x|y|1 + 1|x|y)
          + s (y^2|y|1 + y|y|y + 1|y|y^2)
    \\
    F \rightharpoonup 1|x \wedge y |1
    & = &
          -a_0^{-1}((q+q^3)x|x \wedge y |1
          + q^2|x \wedge y |x) \\
    &   &
          + s(
          q^2 |x \wedge y | y^2
          + y^2 |x \wedge y | 1
          +q y |x \wedge y | y)
  \end{array}
  $
 \\
  \hline
  \\
  Case 4 &
  $
  \begin{array}{rclcrcl}
    E \rightharpoonup 1|x|1
    & = &
          -qd_0^{-1}(x|y|1 + 1|x|y)
          +s(x^2|x|1 + x|x|x + 1|x|x^2)
    \\
    E \rightharpoonup 1|y|1
    & = &
          -qd_0^{-1}(y|y|1 + 1|y|y)
          +t(x^3|x|1 + x^2|x|x + x|x|x^2 + 1|x|x^3)
    \\
    E \rightharpoonup 1|x \wedge y |1
    & = &
          -d_0^{-1}(
          q^2y |x \wedge y | 1
          +q^3 |x \wedge y | y
          +q   |x \wedge y |y
          )
          \\
    &   &
          +s(
          q^2x^2 |x \wedge y | 1
          +q x |x \wedge y | x
          +1 |x \wedge y | x^2
          )
    \\
       F \rightharpoonup 1|* |1
    & = &
          0
  \end{array}
  $
 \\
  \hline
  \\
  Case 5 &
  $
  \begin{array}{rclcrcl}
    E \rightharpoonup 1|x|1
    & = &
          0
    &&
      F \rightharpoonup 1|x|1
    & = &
          \tau^{-1}|y|1
    \\
    E \rightharpoonup 1|y|1
    & = &
          \tau|x|1
    &&
        F \rightharpoonup 1|y|1
    & = &
          0
    \\
    E \rightharpoonup 1|x \wedge y |1
    & = &
          0
    &&
       F \rightharpoonup 1|x \wedge y |1
    & = &
          0
  \end{array}
  $
 \\
  \hline
\end{tabularx}
\caption{Actions of $\mathcal U_q(\mathfrak{sl}_2)$ on the Koszul resolution
    of $\mathbb C_q[x,y]$}
  \label{tab:2}
\end{table}  
\end{lem}
\begin{proof}
  For each $h\in \{E,F,K,K^{-1}\}$, the action of $h$ on
  $\mathbb C_q[x,y] \otimes \mathbb C_q[x,y]$, $1|x|1$, $1|y|1$ and
  $1|x\wedge y |1$ given in the statement of the lemma may be extended
  to a $\mathbb C$-linear mapping
  \[
  \begin{array}{rcl}
    P & \to & P \\
    u & \mapsto & h\rightharpoonup u
  \end{array}
  \]
  in such a way that, for all $g\in \{x,y,x\wedge y\}$, and $a,b$.
  \[
  h \rightharpoonup (a |g |b) = (h_1\rightharpoonup a)
  (h_2\rightharpoonup (1|g|1)) (h_3 \rightharpoonup b)\,.
  \]
  Elementary (though tedious) computations show that this is a
  morphism of complexes of vector spaces. In order to prove the lemma,
  it is therefore sufficient to prove that, for all
  $u\in \{1|x|1,\,1|y|1,\,1|x \wedge y|1\}$,
  \begin{enumerate}[(a)]
  \item $K^{-1}\rightharpoonup (K\rightharpoonup u) = u = K
    \rightharpoonup (K^{-1}\rightharpoonup u)$,
  \item $K \rightharpoonup (E \rightharpoonup u) = q^2 E
    \rightharpoonup (K \rightharpoonup u)$,
  \item $K \rightharpoonup (F \rightharpoonup u) = q^{-2} F
    \rightharpoonup (K \rightharpoonup u)$,
  \item $E \rightharpoonup (F \rightharpoonup u) - F \rightharpoonup
    (E \rightharpoonup u) = \frac{1}{q-q^{-1}} ( K \rightharpoonup u -
    K^{-1} \rightharpoonup u)$.
  \end{enumerate}
The rest of the proof of the lemma explains why these equalities hold
true.


  \emph{Proof of (a).} This follows from the definition of
  the actions of $K$ and $K^{-1}$.


  \emph{Proof of (b) and (c).} In case $0$, the equalities may be
  checked directly. In the other cases, denote by $\gamma$ the
  relative integer such that $K\rightharpoonup u = q^\gamma u$. Then,
  it can be checked from Table~\ref{tab:2} that $E\rightharpoonup u$
  (or $F \rightharpoonup u$) is either $0$ or an eigenvector of the
  action of $K$ with eigenvalue $q^{\gamma+2}$ (or $q^{\gamma-2}$),
  which proves the equality (b) (or, (c), respectively).


  \emph{Proof of (d).} In case $0$, the equality may be checked
  directly from Table~\ref{tab:2}.

  Assume that the considered action is the one of case $1$. Then,
  using that $E\rightharpoonup x = E\rightharpoonup 1|*|1 = 0$,
  \[
  \begin{array}{rcl}
    E \rightharpoonup (F \rightharpoonup (1 |x |1))
    & = &
          b_0^{-1}E \rightharpoonup (x |y|1 + 1 |x|y) \\
    & = &
          b_0^{-1}(1|x|1) (E\rightharpoonup y) \\
    & = &
          1 |x |1 \\
    F \rightharpoonup (E \rightharpoonup (1 |x|1))
    & = &
          0 \\
    K \rightharpoonup (1|x|1)
    & = &
          q |x|1 \\
    K^{-1}\rightharpoonup (1|x|1)
    & = &
          q^{-1}|x|1,
  \end{array}
  \]
  from which the considered equality may be checked directly when
  $u=1|x|1$. Next,
  \[
  \begin{array}{rcl}
    E \rightharpoonup (F \rightharpoonup (1|y|1))
    & = &
          -qb_0^{-1} E\rightharpoonup (y|y|1 + 1|y|y) \\
    & = &
          -qb_0^{-1}((E\rightharpoonup y) (K \rightharpoonup
          (1|y|1)) + (1|y|1)(E\rightharpoonup y)) \\
    & = &
          -(q^{-1}+q)|y|1 \\
    F \rightharpoonup (E \rightharpoonup (1|y|1))
    & = &
          0 \\
    K \rightharpoonup (1|y|1)
    & = &
          q^{-2}|y|1 \\
    K^{-1} \rightharpoonup (1|y|1)
    & = &
          q^2|y|1,        
  \end{array}
  \]
  from which the considered equality may be proved directly when
  $u=1|y|1$. Finally,
  \[
  \begin{array}{rcl}
    E \rightharpoonup (F \rightharpoonup (1|x \wedge y|1))
    & = &
          -b_0^{-1}qE \rightharpoonup (y | x\wedge y |1) \\
    & = &
          -b_0^{-1}q(E\rightharpoonup y) (K\rightharpoonup (1|x \wedge
          y|1) \\
    & = &
          -1|x \wedge y |1 \\
    F \rightharpoonup (E \rightharpoonup (1|x \wedge y|1))
    & = &
          0 \\
    K \rightharpoonup (1| x \wedge y|1)
    & = &
          q^{-1}| x \wedge y |1 \\
    K^{-1} \rightharpoonup (1|x \wedge y|1)
    & = &
          q|x \wedge y|1,        
  \end{array}
  \]
  from which the considered equality may be proved directly when $u=
  1| x\wedge y |1$.

  Now, assume that the considered action is the one of case $2$. Note
  that, there is an isomorphism of Hopf algebras between
  $\mathcal U_q(\mathfrak{sl}_2)$ and
  $\mathcal U_q(\mathfrak{sl}_2)^{\mathrm{op}}$ which exchanges $K$ and
  $K^{-1}$ and exchanges $E$ and $F$. Under this isomorphism, the
  action of case $2$ corresponds to the one of case $1$ provided that
  the following changes are made,
  \begin{itemize}
  \item exchange $x$ and $y$, and next,
  \item transform each tensor $a|*|b$ into $b|*|a$.
  \end{itemize}
  The previous considerations for case $1$ therefore apply to case
  $2$, and hence prove the equality in the latter case.

  Now assume that the considered action is the one of case $3$. Then,
  using that $E\rightharpoonup y= E\rightharpoonup (1|*|1)=0$,
    \[
  \begin{array}{rcl}
    E \rightharpoonup (F \rightharpoonup (1 |x |1))
    & = &
          -qa_0^{-1} E\rightharpoonup (x|x|1 + 1|x|x) \\
    &   &
          +t E \rightharpoonup  (y^3|y|1+\cdots) \\
    & = &
          -qa_0^{-1}((E\rightharpoonup x) (K \rightharpoonup (1|x|1))
          + (1|x|1) (E \rightharpoonup x)) \\
    & = &
          -(q+q^{-1})|x|1 \\
    F \rightharpoonup (E \rightharpoonup (1 |x|1))
    & = &
          0 \\
    K \rightharpoonup (1|x|1)
    & = &
          q^{-2} |x|1 \\
    K^{-1}\rightharpoonup (1|x|1)
    & = &
          q^2|x|1,
  \end{array}
  \]
  from which the considered equality may be checked directly when
  $u=1|x|1$. Next,
  \[
  \begin{array}{rcl}
    E \rightharpoonup (F \rightharpoonup (1|y|1))
    & = &
          -q a_0^{-1} E\rightharpoonup (x|y|1 + 1|x|y) \\
    &   &
          + s E \rightharpoonup (y^2|y|1+\cdots) \\
    & = &
          -qa_0^{-1}(E \rightharpoonup x) (K \rightharpoonup (1|y|1)) \\
    & = &
          -1|y|1 \\
    F \rightharpoonup (E \rightharpoonup (1|y|1))
    & = &
          0 \\
    K \rightharpoonup (1|y|1)
    & = &
          q^{-1}|y|1 \\
    K^{-1} \rightharpoonup (1|y|1)
    & = &
          q|y|1,        
  \end{array}
  \]
  from which the considered equality may be proved directly when
  $u=1|y|1$. Finally,
  \[
  \begin{array}{rcl}
    E \rightharpoonup (F \rightharpoonup (1|x \wedge y|1))
    & = &
          -a_0^{-1}E \rightharpoonup ((q+q^3)x|x \wedge y |1 + q^2| x
          \wedge y |x) \\
    &   &
          +s E\rightharpoonup (q^2|x \wedge y |y^2+\cdots) \\
    & = &
          -a_0^{-1}((q+q^3) (E\rightharpoonup x) (K \rightharpoonup (1
          | x\wedge y |1)) \\
    &   &
          + (q^2 | x\wedge y |1) (E\rightharpoonup x)) \\
        & = &
              -(q^{-2}+1+q^2) | x \wedge y |1 \\
    F \rightharpoonup (E \rightharpoonup (1|x \wedge y|1))
    & = &
          0 \\
    K \rightharpoonup (1| x \wedge y|1)
    & = &
          q^{-3}| x \wedge y |1 \\
    K^{-1} \rightharpoonup (1|x \wedge y|1)
    & = &
          q^3|x \wedge y|1,        
  \end{array}
  \]
  from which the considered equality may be proved directly when $u=
  1| x\wedge y |1$.

  The equality in case $4$ may be deduced from the one in case
  $3$ the same way the one in case $2$ was deduced from the one in
  case $1$.

  To end with, when the considered action is the one of case $5$, then
  the equality may be checked directly from Table~\ref{tab:2}.
\end{proof}

\subsection{The action of $\mathcal U_q(\mathfrak{sl}_2)$ on ${\mathrm{Ext}}^2_{\mathbb C_q[x,y]^e}(\mathbb C_q[x,y],\mathbb C_q[x,y]^e)$}
\label{sec:acti-u_qm-rm}

\begin{lem}
  Let $e_A$ be the free
  generator of
  ${\mathrm{Ext}}^2_{\mathbb C_q[x,y]^e}(\mathbb C_q[x,y],\mathbb
  C_q[x,y]^e)$
  introduced in \eqref{eq:61}. Denote by $\lambda$ the
  mapping $\mathcal U_q(\mathfrak{sl}_2)\to \mathbb C_q[x,y]$ such that the
  identity $h \rightharpoonup e_A = \lambda(h) e_A$ holds. Then,
  Table~\ref{tab:3} describes the mapping $\lambda$.
  \begin{table}[!ht]
    \centering
    \tiny
    \begin{tabularx}{\textwidth}{>{\hsize=.5\hsize}X>{\hsize=1.5\hsize}X}
      \hline
      \\
      Action of Table~\ref{tab:1} & Description of $\lambda$ \\
      \hline
      \\
      Case 0 &
      $
      \begin{array}{rcl}
        K
        & \mapsto &
                    1
                    \\
        E
        & \mapsto &
                    0
                    \\
        F
        & \mapsto &
                    0
      \end{array}
      $
      \\
      \hline
      \\
      Case 1 &
      $
      \begin{array}{rcl}
        K
        & \mapsto &
                    q
                    \\
        E
        & \mapsto &
                    0
                    \\
        F
        & \mapsto &
                    b_0^{-1}q^{-1}y
      \end{array}
      $
      \\
      \hline
      \\
      Case 2 &
      $
      \begin{array}{rcl}
        K
        & \mapsto &
                    q^{-1}
                    \\
        E
        & \mapsto &
                    c_0^{-1}x
                    \\
        F
        & \mapsto &
                    0
      \end{array}
      $
      \\
      \hline
      \\
      Case 3 &
      $
      \begin{array}{rcl}
        K
        & \mapsto &
                    q^3
                    \\
        E
        & \mapsto &
                    0
                    \\
        F
        & \mapsto &
                    a_0^{-1}(q^{-5}+q^{-3}+q^{-1})x
                    -s(q^{-1} + q^{-3} + q^{-2})y^2
      \end{array}
      $
      \\
      \hline
      \\
      Case 4 &
      $
      \begin{array}{rcl}
        K
        & \mapsto &
                    q^{-3}
                    \\
        E
        & \mapsto &
                    d_0^{-1}(q+1+q^{-2})y
                    -s(q+q^{-1}+q^{-3})x^2
                    \\
        F
        & \mapsto &
                    0
      \end{array}
      $
      \\
      \hline
      \\
      Case 5 &
      $
      \begin{array}{rcl}
        K
        & \mapsto &
                    1
                    \\
        E
        & \mapsto &
                    0
                    \\
        F
        & \mapsto &
                    0
      \end{array}
      $
      \\
      \hline
    \end{tabularx}
    \caption{The action of $\mathcal U_q(\mathfrak{sl}_2)$ on ${\mathrm{Ext}}^2_{\mathbb C_q[x,y]^e}(\mathbb C_q[x,y],\mathbb
      C_q[x,y]^e)$}
    \label{tab:3}
  \end{table}
\end{lem}
\begin{proof}
  Let $\varphi_A$ be as in \eqref{eq:61}.  For all $h$, then
  $h\rightharpoonup e_A$ is the cohomology class of
  $h\rightharpoonup \varphi_A$, which is given by \eqref{eq:52}; In
  view of the considerations of \ref{sec:remind-u_qm-mathbb}, if
  $\sum_i a_i\otimes b_i$ is an element of $\mathbb C_q[x,y]^e$ such
  that
  $h \rightharpoonup \varphi_A(1|x \wedge y | 1) = \sum_i a_i \otimes
  b_i$,
  then $h\rightharpoonup \varphi_A$ equals $\sum_i b_i \varphi_A a_i$,
  which is cohomologous to $(\sum_i b_i\mu_A(a_i))\varphi_A$, and
  hence $\lambda(h) = \sum_ib_i\mu_A(a_i)$.  In order to determine
  $\lambda(K)$, $\lambda(E)$ and $\lambda(F)$, it is therefore enough
  to determine $(h\rightharpoonup \varphi_A)(1|x\wedge y |1)$ for
  $h\in \{K,E,F\}$.
  For this purpose, the following piece of notation is used for all
  $h,k,a,b$,
  \[
  (h\otimes k) \rightharpoonup (a \otimes b) = (h\rightharpoonup a)
  \otimes (k \rightharpoonup b)\,.
  \]
  
  First consider $\lambda(K)$. Using \eqref{eq:52}, and using the fact
  that $K^{-1}\rightharpoonup (1| x\wedge y |1) \in \mathbb C\cdot
  (1|x \wedge y |1)$, then
  \[
  \begin{array}{rcl}
    (K \rightharpoonup \varphi_A)(1|x \wedge y|1)
    & = &
          (K \otimes K) \rightharpoonup \varphi_A(K^{-1}\rightharpoonup
          (1|x\wedge y|1)) \\
    & = &
          \varphi_A(K^{-1}\rightharpoonup
          (1|x\wedge y|1))\,.
  \end{array}
  \]
  Therefore, if $t_1,t_2\in \mathbb C^\times$ denote the scalars such
  that $K\rightharpoonup x = t_1 x$ and $K\rightharpoonup y = t_2y$,
  respectively, then $(K^{-1} \rightharpoonup \varphi_A)(1|x \wedge y|1) =
  (t_1t_2)^{-1}| x \wedge y |1$, and hence $\lambda(K) =
  (t_1t_2)^{-1}$. Whence the values of $\lambda(K)$ given in
  Table~\ref{tab:3}.

  Next, consider $\lambda(E)$. Using \eqref{eq:52} and the equality
  $E\rightharpoonup 1= 0$, and using that $K^{-1}\rightharpoonup (1|x
  \wedge y|1)\in \mathbb C\cdot (1|x \wedge y|1)$, then
  \[
  \begin{array}{rcl}
    (E \rightharpoonup \varphi_A)(1|x \wedge y|1)
    & = &
          (q^2E \otimes 1) \rightharpoonup \varphi_A(1|x \wedge y|1) \\
    &   &
          - (K \otimes 1) \rightharpoonup
          \varphi_A((EK^{-1})\rightharpoonup (1|x \wedge y|1)) \\
    &   &
          + (K \otimes E) \rightharpoonup \varphi_A(K^{-1}
          \rightharpoonup (1|x \wedge y|1)) \\
    & = &
                    - (K \otimes 1) \rightharpoonup
          \varphi_A((EK^{-1})\rightharpoonup (1|x \wedge y|1))\,.
  \end{array}
  \]
  Now, use Table~\ref{tab:2} in order to determine $(E\rightharpoonup
  \varphi_A)(1|x \wedge y |1)$. This is zero in cases $0$, $1$, $3$
  and $5$; Therefore, $\lambda(E)=0$ in these cases. In case $2$, then
  $(EK^{-1})\rightharpoonup (1|x \wedge y |1) = -c_0^{-1}|x \wedge y
  |x$, and hence $(E \rightharpoonup \varphi_A)(1|x \wedge y|1) =
  c_0^{-1}\otimes x$; Accordingly, $\lambda(E) = c_0^{-1}x$. And in
  case $4$, then $(EK^{-1})\rightharpoonup (1|x \wedge y |1)$ equals
  \[
          q^{-3}(
      -d_0^{-1}(
      q^2y |x \wedge y | 1
      +q^3 |x \wedge y | y
      +q   |x \wedge y |y
      )
    +s(
     q^2x^2 |x \wedge y | 1
     +q x |x \wedge y | x
     +1 |x \wedge y | x^2
     )),
    \]
and hence $(E \rightharpoonup \varphi_A)(1|x \wedge y|1)$ equals
\[
q^{-3}(
d_0^{-1}(
q^4y \otimes 1
+q^3 \otimes y
+q   \otimes y
)
-s(
q^4x^2 \otimes  1
+q^2 x \otimes  x
+1 \otimes x^2
));
\]
Consequently
\[
\lambda(E) = d_0^{-1}(q+1+q^{-2})y - s(q+q^{-1}+q^{-3})x^2\,.
\]

Finally, consider $\lambda(F)$. Then,
\[
\begin{array}{rcl}
  (F \rightharpoonup \varphi_A)(1|x \wedge y |1)
  & = &
        (1 \otimes F) \rightharpoonup \varphi_A(1|x\wedge y |1) \\
  &   &
        - (1 \otimes K^{-1}) \rightharpoonup \varphi_A((KF)
        \rightharpoonup (1| x \wedge y |1)) \\
  &   &
        + (q^{-2}F \otimes K^{-1}) \rightharpoonup \varphi_A(K
        \rightharpoonup (1|x \wedge y |1)) \\
  & = &
        - (1 \otimes K^{-1}) \rightharpoonup \varphi_A((KF)
        \rightharpoonup (1| x \wedge y |1)),
\end{array}
\]
from which the value of $\lambda(F)$ is determined using
considerations analogous to those used in order to determine $\lambda(E)$.
\end{proof}

\subsection{Computation of the weak homological determinant}
\label{sec:comp-weak-homol}

\begin{lem}
  Let $e_A$ be the free
  generator of ${\mathrm{Ext}}^2_{\mathbb C_q[x,y]^e}(\mathbb
  C_q[x,y],\mathbb C_q[x,y]^e)$ introduced in
  \eqref{eq:61}. Denote by ${\mathrm{whdet}}$ the weak
  homological determinant corresponding to $e_A$. Then, Table~\ref{tab:4}
  gives the values of ${\mathrm{whdet}}$ on $K$, $E$ and $F$. And
  Table~\ref{tab:5} gives the values of $\theta_{\mathrm{whdet}}$ on $K$,
  $E$ and $F$.
  \begin{table}[!ht]
    \centering
    \tiny
    \begin{tabularx}{\textwidth}{>{\hsize=.5\hsize}X>{\hsize=1.5\hsize}X}
      \hline
      \\
      Action of Table~\ref{tab:1} & Description of ${\mathrm{whdet}}$ \\
      \hline
      \\
      Case 0 &
      $
      \begin{array}{rcl}
        K
        & \mapsto &
                    1
                    \\
        E
        & \mapsto &
                    0
                    \\
        F
        & \mapsto &
                    0
      \end{array}
      $
      \\
      \hline
      \\
      Case 1 &
      $
      \begin{array}{rcl}
        K
        & \mapsto &
                    q^{-1}
                    \\
        E
        & \mapsto &
                    0
                    \\
        F
        & \mapsto &
                    -q^2b_0^{-1}y
      \end{array}
      $
      \\
      \hline
      \\
      Case 2 &
      $
      \begin{array}{rcl}
        K
        & \mapsto &
                    q
                    \\
        E
        & \mapsto &
                    -q^{-1}c_0^{-1}x
                    \\
        F
        & \mapsto &
                    0
      \end{array}
      $
      \\
      \hline
      \\
      Case 3 &
      $
      \begin{array}{rcl}
        K
        & \mapsto &
                    q^{-3}
                    \\
        E
        & \mapsto &
                    0
                    \\
        F
        & \mapsto &
                    s(1+q^2+q^4)y^2
                    -a_0^{-1}(q+q^3+q^4)x
      \end{array}
      $
      \\
      \hline
      \\
      Case 4 &
      $
      \begin{array}{rcl}
        K
        & \mapsto &
                    q^3
                    \\
        E
        & \mapsto &
                    s(q^2+1+q^{-2})x^2
                    -d_0^{-1}(q^2+q+q^{-1})y
                    \\
        F
        & \mapsto &
                    0
      \end{array}
      $
      \\
      \hline
      \\
      Case 5 &
      $
      \begin{array}{rcl}
        K
        & \mapsto &
                    1
                    \\
        E
        & \mapsto &
                    0
                    \\
        F
        & \mapsto &
                    0
      \end{array}
      $
      \\
      \hline
    \end{tabularx}
    \caption{The values of ${\mathrm{whdet}}$ on $K$,
      $E$ and $F$}
    \label{tab:4}
  \end{table}
  \begin{table}[!ht]
    \centering
    \tiny
    \begin{tabularx}{\textwidth}{>{\hsize=.5\hsize}X>{\hsize=1.5\hsize}X}
      \hline
      \\
      Action of Table~\ref{tab:1} & Description of $\theta_{\mathrm{whdet}}$ \\
      \hline
      \\
      Case 0 &
      $
      \begin{array}{rcl}
        K
        & \mapsto &
                    K
                    \\
        E
        & \mapsto &
                    E
                    \\
        F
        & \mapsto &
                    F
      \end{array}
      $
      \\
      \hline
      \\
      Case 1 &
      $
      \begin{array}{rcl}
        K
        & \mapsto &
                    q^{-1}K
                    \\
        E
        & \mapsto &
                    E
                    \\
        F
        & \mapsto &
                    qF-b_0^{-1}yK
      \end{array}
      $
      \\
      \hline
      \\
      Case 2 &
      $
      \begin{array}{rcl}
        K
        & \mapsto &
                    qK
                    \\
        E
        & \mapsto &
                    E-qc_0^{-1}xK
                    \\
        F
        & \mapsto &
                    q^{-1}F
      \end{array}
      $
      \\
      \hline
      \\
      Case 3 &
      $
      \begin{array}{rcl}
        K
        & \mapsto &
                    q^{-3}K
                    \\
        E
        & \mapsto &
                    E
                    \\
        F
        & \mapsto &
                    q^3F
                    +s(q^{-2}+1+q^2)y^2
                    -a_0^{-1}(q^{-1}+q+q^2)x
      \end{array}
      $
      \\
      \hline
      \\
      Case 4 &
      $
      \begin{array}{rcl}
        K
        & \mapsto &
                    q^3K
                    \\
        E
        & \mapsto &
                    E
                    +s(q^4+q^2+1)x^2K
                    -d_0^{-1}(q^4+q^3+q)yK
                    \\
        F
        & \mapsto &
                    q^{-3}F
      \end{array}
      $
      \\
      \hline
      \\
      Case 5 &
      $
      \begin{array}{rcl}
        K
        & \mapsto &
                    K
                    \\
        E
        & \mapsto &
                    E
                    \\
        F
        & \mapsto &
                    F
      \end{array}
      $
      \\
      \hline
    \end{tabularx}
    \caption{The values of $\theta_{\mathrm{whdet}}$ on $K$,
      $E$ and $F$}
    \label{tab:5}
  \end{table}
\end{lem}
\begin{proof}
  Let $\lambda$ be such as in \ref{sec:acti-u_qm-rm}.   Recall from \ref{sec:34} and \ref{sec:35} that
  \begin{itemize}
  \item $\lambda$ is characterised by the identity $h \rightharpoonup
    e_A = \lambda(h) e_A$ and it satisfies the identity
    $\lambda(hk) = h_1 \rightharpoonup \lambda(k) \lambda(h_2)$,
  \item ${\mathrm{whdet}}$ is defined by ${\mathrm{whdet}}(h) = S^{-2}(h_2)
    \rightharpoonup \lambda(S^{-3}(h_1))$,
  \item $\theta_{\mathrm{whdet}}$ is defined by $\theta_{\mathrm{whdet}}(h) =
    {\mathrm{whdet}}(S^2(h_1)) h_2$.
  \end{itemize}
  Before proving the lemma, it is convenient to prove the following
  equalities first,
  \begin{enumerate}
  \item ${\mathrm{whdet}}(K) = \lambda(K^{-1})$ and $\theta_{\mathrm{whdet}}(K)
    = \lambda(K^{-1}) K$,
  \item ${\mathrm{whdet}}(K^{-1}) = \lambda(K)$ and $\theta_{\mathrm{whdet}}(K^{-1})
    = \lambda(K) K^{-1}$,
  \item ${\mathrm{whdet}}(E) = -q^{-2}\lambda(E) \lambda(K^{-1})$ and $\theta_{\mathrm{whdet}}(E)
    = E - \lambda(E) \lambda(K^{-1}) K$,
  \item ${\mathrm{whdet}}(F) = -q^2\lambda(K) \lambda(F)$ and $\theta_{\mathrm{whdet}}(F)
    = \lambda(K) F - \lambda(K) \lambda(F)$.
  \end{enumerate}

  (1) Since $\lambda(1) = 1$ and $\lambda(K) \in \mathbb C$, then
  $\lambda(K^{-1}) = \lambda(K)^{-1}$ and $\lambda(K^{-1}) \in \mathbb
  C$. Therefore, ${\mathrm{whdet}}(K) = K \rightharpoonup \lambda(K^{-1}) =
  \lambda(K^{-1})$, and hence $\theta_{\mathrm{whdet}}(K) = {\mathrm{whdet}}(K)
  K = \lambda(K^{-1}) K$.


  (2) is proved using the same considerations.


  (3) ${\mathrm{whdet}}(E) = S^{-2}(E) \rightharpoonup \lambda(1) + K
  \rightharpoonup \lambda(S^{-3}(E))$. Note that $S^{-2}(E) = q^{-2}
  E$; Moreover, since $S(E) = -EK^{-1}$ and $S(K) = K^{-1}$, then
  $S^{-1}(E) = -K^{-1}E$ and $S^{-3}(E) = - q^{-2} K^{-1}E$. Therefore,
  \[
  \begin{array}{rcll}
    {\mathrm{whdet}}(E)
    & = &
          q^{-2}E \rightharpoonup 1 - q^{-2}K \rightharpoonup
          \lambda(K^{-1}E) \\
    & = &
          -q^{-2}K \rightharpoonup (K^{-1}\rightharpoonup
          \lambda(E)\lambda(K^{-1})) \\
    & = &
          -q^{-2}\lambda(E) \lambda(K^{-1})
          & \text{(since $\lambda(K^{-1}) \in \mathbb C$).}
  \end{array}
  \]
  Now,
  \[
  \begin{array}{rcl}
    \theta_{\mathrm{whdet}}(E)
    & = &
          {\mathrm{whdet}}(S^2(E))K + {\mathrm{whdet}}(1) E \\
    & = &
          -\lambda(E) \lambda(K^{-1})K + E\,.
  \end{array}
  \]


  (4)
  \[
  \begin{array}{rcll}
    {\mathrm{whdet}}(F)
    & = &
    S^{-2}(1) \rightharpoonup \lambda(S^{-3}(F)) + S^{-2}(F)
          \rightharpoonup \lambda(S^{-3}(K^{-1})) \\
    & = &
          q^2 \lambda(S^{-1}(F)) + q^2 F \rightharpoonup \lambda(K) \\
    & = &
          q^2 \lambda(S^{-1}(F)) & \text{\tiny (since $\lambda(K) \in
                                   \mathbb C$).}
  \end{array}
  \]
  Since $S(F) = -KF$ and $S(K) = K^{-1}$, then $S^{-1}(F) =
  -FK$. Therefore,
  \[
  \begin{array}{rcll}
    {\mathrm{whdet}}(F)
    & = &
          -q^2 \lambda(FK) \\
    & = &
          -q^2 (F \rightharpoonup \lambda(K) \lambda(1) +
          K^{-1}\rightharpoonup \lambda(K) \lambda(F) ) \\
    & = &
          -q^2 \lambda(K) \lambda(F) &
                                       \text{\tiny (since
                                       $\lambda(K),\lambda(1) \in
                                       \mathbb C$).}
  \end{array}
  \]
  Now,
  \[
  \begin{array}{rcl}
    \theta_{\mathrm{whdet}}(F)
    & = &
          {\mathrm{whdet}}(S^2(F)) + {\mathrm{whdet}}(S^2(K^{-1})) F \\
    & = &
          q^{-2}{\mathrm{whdet}}(F) + \lambda(K) F \\
    & = &
          -\lambda(K) \lambda(F) + \lambda(K) F\,.
  \end{array}
  \]


  Now, using equalities (1), (2), (3), and (4), it is possible to
  recover Table~\ref{tab:4} and Table~\ref{tab:5} using
  Table~\ref{tab:3}, which proves the lemma.
\end{proof}

\subsection{Description of the Nakayama automorphism}
\label{sec:descr-nakay-autom}

\begin{prop}
  Let $q\in \mathbb C^\times$ be a non root of unity. Assume that
  $\mathbb C_q[x,y]$ is endowed with a structure of
  $\mathcal U_q(\mathfrak{sl}_2)$-module algebra. Then,
  $\mathbb C_q[x,y] \sharp \mathcal U_q(\mathfrak{sl}_2)$ is a skew Calabi-Yau
  algebra with a Nakayama automorphism given by
  \[
  x \mapsto q^{-1} x\ \text{and}\ y \mapsto qy\]
  and which values on $K$, $E$ and $F$ are given in Table~\ref{tab:6}.
  \begin{table}[!ht]
    \centering
    \tiny
    \begin{tabularx}{\textwidth}{>{\hsize=.5\hsize}X>{\hsize=1.5\hsize}X}
      \hline
      \\
      Action of Table~\ref{tab:1} & Values of the Nakayama
      automorphism on $K$, $E$ and $F$ \\
      \hline
      \\
      Case 0 &
      $
      \begin{array}{rclccrclccrcl}
        K
        & \mapsto &
                    K
        \\
             E
        & \mapsto &
                    q^{-2}E
        \\
        F
        & \mapsto &
                    q^2F
      \end{array}
      $
      \\
      \hline
      \\
      Case 1 &
      $
      \begin{array}{rclccrclccrcl}
        K
        & \mapsto &
                    q^{-1}K
        \\
        E
        & \mapsto &
                    q^{-2}E
        \\
        F
        & \mapsto &
                    q^3F-b_0^{-1}q^2yK
      \end{array}
      $
                    \\
                    \hline
                    \\
                    Case 2 & 
      $
      \begin{array}{rcl}
        K
        & \mapsto &
                    qK
                    \\
        E
        & \mapsto &
                    q^{-2}E-c_0^{-1}q^{-1}xK
                    \\
        F
        & \mapsto &
                    qF
      \end{array}
      $
      \\
      \hline
      \\
      Case 3 &
      $
      \begin{array}{rcl}
        K
        & \mapsto &
                    q^{-3}K
                    \\
        E
        & \mapsto &
                    q^{-2}E
                    \\
        F
        & \mapsto &
                    q^5F
                    +s(1+q^2+q^4)y^2
                    -a_0^{-1}(q+q^3+q^4)x
      \end{array}
      $
      \\
      \hline
      \\
      Case 4 &
      $
      \begin{array}{rcl}
        K
        & \mapsto &
                    q^3K
                    \\
        E
        & \mapsto &
                    q^{-2}E
                    +s(q^2+1+q^{-2})x^2K
                    -d_0^{-1}(q^2+q+q^{-1})yK
                    \\
        F
        & \mapsto &
                    q^{-1}F
      \end{array}
      $
      \\
      \hline
      \\
      Case 5 &
      $
      \begin{array}{rcl}
        K
        & \mapsto &
                    K
                    \\
        E
        & \mapsto &
                    q^{-2}E
                    \\
        F
        & \mapsto &
                    q^2F
      \end{array}
      $
      \\
      \hline
    \end{tabularx}
    \caption{Nakayama automorphism of $\mathbb C_q[x,y]\sharp
      \mathcal U_q(\mathfrak{sl}_2)$}
    \label{tab:6}
  \end{table}
\end{prop}
\begin{proof}
  It follows from  \ref{sec:19-1} that $\mathbb C_q[x,y]\sharp
  \mathcal U_q(\mathfrak{sl}_2)$ is skew Calabi-Yau and admits $\mu_A \sharp
  (\theta_{\mathrm{whdet}} \circ \mu_H)$ as a Nakayama automorphism, where
  \begin{itemize}
  \item $\mu_A$ is as in \eqref{eq:59},
  \item ${\mathrm{whdet}}$ and $\theta_{\mathrm{whdet}}$ are as in
    \ref{sec:comp-weak-homol},
  \item $\mu_H = S^{-2}\circ \Xi_{\int_{\ell}}^r$.
  \end{itemize}
  Note that $\Xi_{\int_\ell}^r={\mathrm{Id}}_H$ because
  $\int_\ell=\epsilon$ (see \eqref{eq:58}), and $S^{-2}$ is given by
  $K\mapsto K$, $E\mapsto q^{-2}E$ and $F\mapsto q^2F$. The
  description given in Table~\ref{tab:6} therefore follows from the
  one given in Table~\ref{tab:5}.
\end{proof}

\begin{rem}
\begin{enumerate}
\item In case $5$, the given Nakayama automorphism is the inner
  automorphism of $\mathbb C_q[x,y]\sharp \mathcal U_q(\mathfrak{sl}_2)$
  associated with $K^{-1}$. Accordingly, the smash-product is
  Calabi-Yau.
\item This fact can be recovered from the theorem of
  \ref{sec:6}. Indeed, in case $5$, any action of
  $\mathcal U_q(\mathfrak{sl}_2)$ preserves the grading of $\mathbb
  C_q[x,y]$. Moreover the weak homological determinant is a homological
  determinant, which is moreover trivial (see
  Table~\ref{tab:4}). Therefore, taking $h_0=K^{-1}$ and $k_A=1$ in
  \ref{sec:6} yields the desired conclusion.
\end{enumerate}
\end{rem}
\section{Applications to Artin-Schelter algebras}
\label{sec:appl-artin-schelt}

Assume that $A$ is an augmented $H$-module algebra and that the
antipode of $H$ is invertible. This section investigates when $\Lambda$
is Artin-Schelter Gorenstein/regular. Sufficient conditions for this
to be the case are presented in Section~\ref{sec:artin-schelt-algebr}
using an analogue of Stefan's spectral sequence presented in
Section~\ref{sec:case-artin-shelter}. For this purpose, some useful
module structures are introduced in
Section~\ref{sec:usef-module-struct}.

\subsection{Useful module structures}
\label{sec:usef-module-struct}

In the rest of the text, whenever  $M,N$ are left $H$-modules, the following
structure of left $H$-module is considered on ${\mathrm{Hom}}_{\k}(M,N)$
\begin{equation}
  \label{eq:10}
(h\rightharpoonup f)(m) = h_2\rightharpoonup
f(S^{-1}(h_1)\rightharpoonup m)\,.
\end{equation}
Also $M\otimes N$ is considered as a left $H$-module for the following
action
\begin{equation}
  \label{eq:16}
h\rightharpoonup (m\otimes n) = h_1\rightharpoonup m \otimes
h_2\rightharpoonup n\,.
\end{equation}


Assume that $M,N$ are left $\Lambda$-modules.

It is proved in \cite[(1)]{MR2809906} that ${\mathrm{Hom}}_A(M,N)$ is an
$H$-submodule of ${\mathrm{Hom}}_{\k}(M,N)$. Moreover, according to \cite[Lemma
2.2]{MR2809906}, the following canonical mapping is well-defined and
 bijective.
\begin{equation}
  \label{eq:22}
\begin{array}{rcl}
{\mathrm{Hom}}_H(\,_H\k,{\mathrm{Hom}}_A(M,N)) & \xrightarrow{\cong} &
  {\mathrm{Hom}}_{\Lambda}(M,N)\\
\phi & \mapsto& \phi(1)\,.
\end{array}
\end{equation}


In order to analyse the spectral sequence mentioned earlier, it is
necessary to compare objects such as ${\mathrm{Hom}}_A(M,A)\otimes H$ and
${\mathrm{Hom}}_A(M,\Lambda)$.

Firstly, consider ${\mathrm{Hom}}_A(M,\Lambda)$. The
natural structure of right $\Lambda$-module on  ${\mathrm{Hom}}_A(M,\Lambda)$ 
commutes with the action of $H$. Hence, ${\mathrm{Hom}}_A(M,\Lambda)$
is an $H-\Lambda$-bimodule.

Secondly, consider ${\mathrm{Hom}}_A(M,A)\otimes H$. In view of \eqref{eq:10}
and \eqref{eq:16}, it is a left $H$-module. It is also a right
$\Lambda$-module for the action defined by $(f\otimes \ell) ah = 
f \cdot (\ell_1\rightharpoonup a)\otimes \ell_2 h$ (here ${\mathrm{Hom}}_A(M,A)$ is
naturally a right $A$-module). Given that the structures
of left $H$-module and  right $A$-module of ${\mathrm{Hom}}_A(M,A)$ are
compatible in the following sense
\begin{equation}
  \label{eq:38}
h\rightharpoonup (fa)=
(h_1\rightharpoonup f) (h_2\rightharpoonup a)\,,
\end{equation}
it follows that 
${\mathrm{Hom}}_A(M,A)\otimes H$ is also an $H-\Lambda$-bimodule.

For these structures of $H-\Lambda$-bimodule,
the following  canonical mapping is $H\otimes \Lambda^{\mathrm{op}}$-linear
\begin{equation}
  \label{eq:39}
\begin{array}{rcl}
  {\mathrm{Hom}}_A(M,A)\otimes H & \to & {\mathrm{Hom}}_A(M,\Lambda)\\
  f \otimes \ell & \mapsto & (m \mapsto
                                f(m) \ell)\,.
\end{array}
\end{equation}
Note that it is a functorial isomorphism if $M$ is finitely
presented in ${\mathrm{mod}}(A)$ of if ${\mathrm{dim}}_\k\,H<\infty$.


Using projective resolutions to construct
 ${\mathrm{Ext}}$-spaces, the previous considerations entail the
 following result.  
\begin{lem}
  Let $M,N$ be left $\Lambda$-modules and let $p$ be a natural
  integer. 
  \begin{enumerate}
  \item There is a functorial structure of left  $H$-module  on ${\mathrm{Ext}}^p_A(M,N)$
    which coincides with the one introduced previously when $p=0$.
  \item Taking into account the natural structure of right
    $\Lambda$-module, ${\mathrm{Ext}}^p_A(M,\Lambda)$ is then an $H-\Lambda$-bimodule.
\item Taking into account the natural structure of right
  $\Lambda$-module such that
  \[
  (e\otimes \ell) ah = e (\ell_1 \rightharpoonup a) \otimes \ell_2
  h\,,
  \]
  ${\mathrm{Ext}}_A^p(M,A)\otimes H$ is then an $H-\Lambda$-bimodule.
\item There is a functorial morphism of $H-\Lambda$-bimodules
 ${\mathrm{Ext}}^p_A(M,A)\otimes H\to {\mathrm{Ext}}^p_A(M,\Lambda)$. If $M$ has a
 resolution in ${\mathrm{mod}}(A)$ by finitely generated projectives, then
this  is an isomorphism. 
 \end{enumerate}
\end{lem}

\subsection{A spectral sequence for the cohomology of augmented algebras}
\label{sec:case-artin-shelter}

\subsubsection{}
\label{sec:10}

The following is an analogue of the spectral sequence constructed by
Stefan (\cite[Theorem 3.3]{MR1358765}) for the Hochschild cohomology
on Hopf-Galois extension. The proof is also analogous and omitted.
\begin{prop}
  Let $M,N$ be left $\Lambda$-modules. There is a Grothendieck spectral
  sequence functorial in $M$ and $N$
\[
{\mathrm{Ext}}^p_H(\,_H\k, {\mathrm{Ext}}^q_A(M,N))\Rightarrow {\mathrm{Ext}}^{p+q}_{\Lambda}(M,N)\,.
\]
\end{prop}

\subsubsection{}
\label{sec:11}

In view of determining when $\Lambda$ has the Artin-Schelter property,
it is useful to simplify the $E_2$ term of the spectral sequence.
\begin{lem}
  Assume that $_H\k$ has a resolution in ${\mathrm{mod}}(H)$ by 
  finitely generated projectives. Let $M\in {\mathrm{mod}}(\Lambda)$ admit a
  resolution in ${\mathrm{mod}}(A)$ by 
  finitely generated projectives. Then, for every  $p,q\in \mathbb N$, there
  is an isomorphism of right $\Lambda$-modules
\[
{\mathrm{Ext}}^p_H(\,_H\k,{\mathrm{Ext}}^q_A(M,\Lambda))\simeq {\mathrm{Ext}}^q_A(M,A)\otimes
 {\mathrm{Ext}}^p_H(\,_H\k,H)
\]
where the module structure of the right-hand side term is defined by 
\[
(e_M\otimes e_H) ah = (S^{-1}(h_1)\rightharpoonup (e_Ma) ) \otimes e_Hh_2\,.
\]
\end{lem}
\begin{proof}
  Let $P\to \,_H\k$ be a resolution by finitely generated projective
  left $H$-modules.   Deriving \eqref{eq:38} yields an analogous identity in ${\mathrm{Ext}}_A^q(M,A)$. Consider the following action of $\Lambda$ on 
  ${\mathrm{Ext}}^q_A(M,A) \otimes {\mathrm{Hom}}_H(P,H)$
  \[
  (e_M\otimes \varphi) ah := (S^{-1}(h_1)\rightharpoonup (e_Ma) )
  \otimes \varphi h_2\,.
  \]
  It is a structure of complex of right $\Lambda$-modules. Therefore,
  the action of $\Lambda$ on  
  ${\mathrm{Ext}}^q_A(M,A)\otimes
  {\mathrm{Ext}}^p_H(\,_H\k,H)$ given in the statement of the lemma is a
  structure of right $\Lambda$-module.


  The  $H-\Lambda$-bimodules
   ${\mathrm{Ext}}^q_A(M,\Lambda)$ and  ${\mathrm{Ext}}^q_A(M,A) \otimes
  H$ are isomorphic (see \eqref{eq:39}).  Consider the morphism of complexes
  \[
  \begin{array}{crcl}
    \theta\colon
    & {\mathrm{Ext}}^q_A(M,A) \otimes {\mathrm{Hom}}_H(P,H)
    & \to &
            {\mathrm{Hom}}_H(P, {\mathrm{Ext}}^q_A(M,A) \otimes H) \\
    & e_M \otimes \varphi
    & \mapsto &
                \left(
                p \mapsto (\varphi(p)_1\rightharpoonup e_M) \otimes
                \varphi(p)_2\right) \,.
  \end{array}
  \]
  Note that $\theta(e_M\otimes \varphi)$ is indeed $H$-linear because
  $\varphi$ is $H$-linear and because of the definition of the action
  of $H$ on ${\mathrm{Ext}}^q_A(M,A)\otimes H$ on the left (see
  \eqref{eq:16}). Since ${\mathrm{Ext}}^q_A(M,A)\otimes H\in {\mathrm{mod}}(H\otimes \Lambda^{\mathrm{op}})$ (see
  \ref{sec:usef-module-struct}), then ${\mathrm{Hom}}_H(P,{\mathrm{Ext}}^q_A(M,A)\otimes H)\in {\mathrm{mod}}(\Lambda^{\mathrm{op}})$.

  The mapping $\theta$ is $\Lambda$-linear. Indeed, let $e_M\otimes
  \varphi\in {\mathrm{Ext}}^q_A(M,A)\otimes {\mathrm{Hom}}_H(P,H)$, $a\in A$, $h\in H$
  and $p\in P$. Then,
  \[
  \begin{array}{rcl}
    (\theta(e_M\otimes \varphi)ah)(p) 
    & = &
          \theta(e_M\otimes \varphi)(p) ah \\
    & = &
          (\varphi(p)_1\rightharpoonup e_M\otimes \varphi(p)_2)ah \\
    & = &
          (\varphi(p)_1\rightharpoonup
          e_M)(\varphi(p)_2\rightharpoonup a) \otimes \varphi(p)_3 h
    \\
    & \underset{~\eqref{eq:38}}= &
          \varphi(p)_1\rightharpoonup (e_Ma) \otimes \varphi(p)_2h \\
    \theta((e_M\otimes \varphi)ah)(p)
    & = &
          \theta(S^{-1}(h_1) \rightharpoonup (e_Ma)  \otimes \varphi
          h_2)(p) \\
    & = &
          (\varphi h_2)(p)_1 \rightharpoonup (S^{-1}(h_1)
          \rightharpoonup (e_Ma)) \otimes (\varphi h_2) (p)_2 \\
    & = &
          (\varphi(p)_1h_2) \rightharpoonup
          (S^{-1}(h_1)\rightharpoonup (e_Ma)) \otimes \varphi(p)_2 h_3
    \\
    & = &
          \varphi(p)_1\rightharpoonup (e_Ma) \otimes \varphi(p)_2h\,,
  \end{array}
  \]
  which explains why $\theta$ is $\Lambda$-linear.

  In order to prove the lemma, it is therefore sufficient to prove
  that $\theta$ is bijective. When $P$ is replaced by $H$ in the
  description of $\theta$, the resulting mapping reduces to
  \[
  \begin{array}{rcl}
    {\mathrm{Ext}}_A^q(M,A)\otimes H & \to & {\mathrm{Ext}}_A^q(M,A)\otimes H \\
    e_M\otimes \varphi & \mapsto & \varphi_1\rightharpoonup e_M
                                   \otimes \varphi_2\,.
  \end{array}
  \]
  This is indeed bijective with inverse mapping given by $e_M\otimes
  \varphi \mapsto S^{-1}(\varphi_1)\rightharpoonup e_M \otimes
  \varphi_2$. Now, since $P$ consists of finitely generated projective left $H$-modules,
  the previous considerations show that $\theta$ is bijective.
\end{proof}

\subsection{The Artin-Schelter property of $A\sharp H$}
\label{sec:artin-schelt-algebr}

\subsubsection{}
\label{sec:12}

Here is how the left Artin-Schelter property behaves under taking
smash products. The dual statement for the right Artin-Schelter
property holds true.

\begin{prop}
  Let $H$ be a  Hopf algebra with invertible antipode. Assume that
   $_H\k$ has a resolution in ${\mathrm{mod}}(H)$ by finitely generated
  projectives. Let
  $A$ be an augmented $H$-module algebra. Assume that $_A\k$ has a
  resolution in ${\mathrm{mod}}(A)$ by finitely generated projectives. If $A$ and $H$
  have the left Artin-Schelter property in dimension $n$ and $d$,
  respectively, then  $A\sharp H$ has the left Artin-Schelter property
  in dimension $n+d$.
\end{prop}
\begin{proof}
  By assumption, ${\mathrm{Ext}}^q_A(\k,A)$ is one dimensional if $q=n$ and $0$
  otherwise. And  ${\mathrm{Ext}}^p_H(\,_H\k,H)$ is one dimensional if
  $p=d$ and $0$ otherwise. Let $M=\,_\Lambda\k$ and $N=\Lambda$. The proposition
  therefore follows from \ref{sec:10} and \ref{sec:11}.
\end{proof}

\subsubsection{}
\label{sec:30}

Whenever $B$ is an augmented $\k$-algebra with left Artin-Schelter
property in dimension $t$, there exists a unique algebra homomorphism
$\lambda_B\colon B\to \k$ such that the  right $B$-module structure
of ${\mathrm{Ext}}^t_B(\k,B)$ is given by $e_Bb=\lambda_B(b)e_B$ for every
$e_B\in {\mathrm{Ext}}^t_B(\k,B)$.

In the setting of \ref{sec:12}, the algebra homomorphism
$\lambda_{A\sharp H}$ may be described in terms of $\lambda_A$ and
$\lambda_H$. Since ${\mathrm{Ext}}^n_A(\k,A)$ is one dimensional, there exists an algebra
homomorphism $\delta\colon H\to \k$ such that the  left
$H$-module structure of ${\mathrm{Ext}}^n_A(\k,A)$ is given by $h\rightharpoonup
e_A=\delta(h)e_A$ for every $e_A\in {\mathrm{Ext}}^n_A(\k,A)$. Note that,
when $H$ is finite dimensional and $A$ is 
connected graded and Artin-Schelter Gorenstein, then $\delta\circ
S$ ($=\delta\circ S^{-1}$)
is the homological determinant defined in \cite[Definition
3.3]{MR2568355}. According to \ref{sec:11}, the algebra homomorphism
$\lambda_{A\sharp H}$ is given by
\[
\lambda_{A\sharp H}(ah) = \lambda_A(a)\delta (S^{-1}(h_1))\lambda_H(h_2)\,.
\]

\subsubsection{}
\label{sec:13}
Here are sufficient conditions for $\Lambda$ to be Artin-Schelter regular.
\begin{prop}
  Let $H$ be a  Hopf algebra with invertible antipode. Assume that
  $_H\k$ has a resolution in ${\mathrm{mod}}(H)$ by finitely generated
  projectives. Let
  $A$ be an augmented 
  $H$-module noetherian algebra. Assume that $H$ and $A$ are Artin-Schelter
  regular in dimension $n$ and $d$, respectively. Then, $A\sharp H$ is
  Artin-Schelter regular in dimension $n+d$.
\end{prop}
\begin{proof}
  Since the antipode is invertible, then $\k_H$ has a resolution in
  ${\mathrm{mod}}(H^{\mathrm{op}})$ by finitely generated projectives.
  According to \ref{sec:12} and its dual version, it suffices to prove
  that ${\mathrm{gl.dim.}}\,\Lambda\leqslant n+d$. This follows from
  \ref{sec:10}, from its dual version and from the hypotheses on the
  global dimensions of $A$ and $H$.
\end{proof}

\subsubsection{}
\label{sec:20}
In order to present an analog of the previous result for
Artin-Schelter Gorenstein algebras it is necessary to deal with the
finiteness of injective dimensions. This is taken care of by the
following result.
\begin{lem}
  Let $H$ be a Hopf algebra with invertible antipode. Let $A$ be an $H$-module algebra. Let
  $\Lambda=A\sharp H$. Assume
  that $A$ is left noetherian. Then, ${\mathrm{id}}(\,_{\Lambda}\Lambda)<\infty$ under any of the two following conditions
  \begin{enumerate}[(a)]
  \item $H$ is finite-dimensional and ${\mathrm{id}}(\,_AA)<\infty$,
  \item  ${\mathrm{pd}}_{H}(\,_H\k)<\infty$
    and ${\mathrm{id}}_A(A)<\infty$.
    \end{enumerate}
\end{lem}
\begin{proof}
  First, assume (a).  Note that ${\mathrm{id}}(\,_HH)=0$, since all finite
  dimensional Hopf algebras are Frobenius, and that $_H\k$ has a
  resolution in ${\mathrm{mod}}(H)$ by finitely generated projectives. Any
  finitely generated left $\Lambda$-module is finitely generated as an
  $A$-module. Therefore, \ref{sec:10} and \ref{sec:11} apply to any
  $M\in {\mathrm{mod}}(\Lambda)$ which is finitely generated. For any such
  $M$, it follows that ${\mathrm{Ext}}_{\Lambda}^m(M,\Lambda)=0$ for
  $m>{\mathrm{id}}(\,_AA)$.  Taking direct limits entails that
  ${\mathrm{id}}_{\Lambda}(\Lambda)<\infty$.


Next, assume (b).
  The following properties imply that ${\mathrm{id}}_A(\Lambda)<\infty$:
  \begin{itemize}
  \item since $A$ is left noetherian any direct sum of injective
    left $A$-modules is an injective $A$-module (see \cite[Proposition
    3.46, p. 80]{MR1653294}),
  \item ${\mathrm{id}}(\,_AA)<\infty$,
  \item  $\Lambda\simeq A\otimes H$ as left $A$-modules.
  \end{itemize}
The conclusion of the lemma then follows from \ref{sec:10} applied to $N=\Lambda$.
\end{proof}

\subsubsection{}
\label{sec:21}
Here are sufficient conditions for $\Lambda$ to be Artin-Schelter Gorenstein.
\begin{prop}
  Let $H$ be a  Hopf algebra. Let
  $A$ be an augmented
  $H$-module noetherian algebra which is moreover 
  Artin-Schelter Gorenstein in dimension $n$.  
  \begin{enumerate}
  \item If $H$ is finite dimensional, then $\Lambda$ is
    Artin-Schelter Gorenstein in dimension $n$.
  \item If $H$ has Van den Bergh duality in dimension $d$, then $\Lambda$ is
    Artin-Schelter Gorenstein in dimension $n+d$.
  \end{enumerate}
\end{prop}
\begin{proof}
  Any finite dimensional Hopf algebra has an invertible antipode and
  is selfinjective. In particular, it is Artin-Schelter Gorenstein.
  Also, recall that, if $H$ has Van den Bergh duality, then its
  antipode is invertible, $_H\k\in {\mathrm{per}}(H)$ and
  $\k_H\in {\mathrm{per}}(H^{\mathrm{op}})$, and $H$ has the Artin-Schelter
  property (see \ref{sec:25} and its dual version). Therefore (1) and
  (2) follow from \ref{sec:12}, from \ref{sec:20} and their dual
  versions.
\end{proof}

\section{Acknowledgements}

I thank an anonymous referee for several comments which improved the
presentation of this article and for having pointed out an error in a
previous version of the lemma in \ref{sec:22}.

\begin{table}[!ht]
  \centering
  \begin{tabularx}{\textwidth}{>{\tiny\hsize=.4\hsize}X>{\tiny\hsize=2.1\hsize}X>{\tiny\hsize=0.5\hsize}X}
    \hline
    \\
    \multicolumn{3}{l}{Index of notation}
    \\
    \hline Symbol & Meaning & Location
    \\
    \hline
    \\
    $\k$ & base field & Introduction \\
    $A$ & dg $H$-module $\k$-algebra & Introduction \\
    $A^e$ & $A\otimes_{\k} A^{\mathrm{op}}$ & Introduction \\
    $\mu_A$ & a Nakayama automorphism of $A$ & Introduction \\
    $H$ &  Hopf $\k$-algebra & Introduction \\
    $S$ & antipode of $H$ & Introduction \\
    $\Lambda$ & smash product $A\sharp H$ & Introduction \\
    $^\tau M^\sigma$ & twisting of a bimodule by algebra automorphisms
    $\tau,\sigma$ & Introduction \\
    $\otimes$ & $\otimes_\k$ & \ref{sec:conventions-notation} \\
    $\mathcal C(A)$ & category of left dg $A$-modules &
    \ref{sec:conventions-notation} \\
    $\mathcal D(A)$ & derived category of left dg $A$-modules &
    \ref{sec:conventions-notation} \\
    $\int_\ell,\int_r$ & left and right homological integrals &
    \ref{sec:homol-integr-wind} \\
    $\Xi^\ell_\bullet,\Xi^r_\bullet$ & winding automorphisms &
    \ref{sec:homol-integr-wind} \\
    $\bullet\uparrow^{H^e}$ & functor from left $H$-modules to
    $H$-bimodules & \ref{sec:24} \\
    $\Delta_i$ & dg algebra whose left dg modules are
    $H_{S^{2i}}$-equivariant dg $A$-bimodules &
    \ref{sec:algebras-delta_i} \\
    $D\sharp\,^\sigma H$ & $\Lambda$-bimodule extension of an
    $H_{S^{2i}}$-equivariant dg $A$-bimodule &
    \ref{sec:asharp-h-bimodules} \\
    $H^{\sigma}\sharp D$ & & \ref{sec:45} \\
    $\Pi_n(A)$ & Calabi-Yau completion of $A$ & Section
    \ref{sec:appl-constr-calabi} \\
    $\Pi_n(A,\alpha)$ & a deformed Calabi-Yau completion of $A$ &
    Section
    \ref{sec:appl-constr-calabi} \\
    $e_A$ & a free generator of the left $A$-module ${\mathrm{Ext}}^n_{A^e}(A,A^e)$ & \ref{sec:17} \\
    $\varphi_A$ & a cocycle representing $e_A$ & \ref{sec:17} \\
    ${\mathrm{whdet}}$ & weak homological determinant $H\to A$ & \ref{sec:35} \\
    $\theta_{\mathrm{whdet}}$ & algebra homomorphism $H\to \Lambda$
    corresponding to ${\mathrm{whdet}}$ & \ref{sec:35} \\
    ${\mathrm{hdet}}$ & homological determinant & \ref{sec:36} \\
    $\mu_\Lambda,\mu_H$ & Nakayama automorphisms of
    $\Lambda$ and $H$& \ref{sec:19} \\
    $\mathbb C_q[x,y]$ & the quantum plane &
    \ref{sec:remind-u_qm-mathbb} \\
    $\mathcal U_q(\mathfrak{sl}_2)$ & the quantum enveloping algebra &
    \ref{sec:remind-u_qm-mathbb}
    \\
    \hline
    \end{tabularx}
\end{table}

\bibliographystyle{plain}
\bibliography{biblio}

\end{document}